\documentclass[10pt]{article}
\usepackage[text={6.75in,9.5in},centering,letterpaper]{geometry}
\usepackage{amsmath, amsthm, amssymb}
\usepackage{graphicx}
\usepackage[margin = 10pt, font=small, labelfont=bf, labelsep = endash]{caption}
\usepackage{subcaption}
\usepackage{txfonts}
\usepackage[pdfpagemode=UseOutlines, hidelinks]{hyperref}
\numberwithin{equation}{section}

\newcommand{\R}{\ensuremath{\mathbb{R}}}
\newcommand{\eps}{\ensuremath{\varepsilon}}
\DeclareMathOperator{\sech}{sech}
\newcommand{\sgn}{\ensuremath{\mbox{sgn}}}

\title{The dynamics of disappearing pulses in a singularly perturbed reaction-diffusion system with parameters that vary in time and space}
\author{Robbin Bastiaansen and Arjen Doelman}

\begin{document}

\maketitle

\begin{abstract}
We consider the evolution of multi-pulse patterns in an extended Klausmeier equation -- as generalisation of the well-studied Gray-Scott system, a prototypical singularly perturbed reaction-diffusion equation -- with parameters that change in time and/or space. As a first step we formally show -- under certain conditions on the parameters -- that the full PDE dynamics of a $N$-pulse configuration can be reduced to a $N$-dimensional dynamical system that describes the dynamics on a $N$-dimensional manifold $\mathcal{M}_N$. Next, we determine the local stability of $\mathcal{M}_N$ via the quasi-steady spectrum associated to evolving $N$-pulse patterns. This analysis provides explicit information on the boundary $\partial \mathcal{M}_N$ of $\mathcal{M}_N$. Following the dynamics on $\mathcal{M}_N$, a $N$-pulse pattern may move through $\partial \mathcal{M}_N$ and `fall off' $\mathcal{M}_N$. A direct nonlinear extrapolation of our linear analysis predicts the subsequent fast PDE dynamics as the pattern `jumps' to another invariant manifold $\mathcal{M}_M$, and specifically predicts the number $N-M$ of pulses that disappear during this jump. Combining the asymptotic analysis with numerical simulations of the dynamics on the various invariant manifolds yields a hybrid asymptotic-numerical method that describes the full pulse interaction process that starts with a $N$-pulse pattern and typically ends in the trivial homogeneous state without pulses. We extensively test this method against PDE simulations and deduce a number of general conjectures on the nature of pulse interactions with disappearing pulses. We especially consider the differences between the evolution of (sufficiently) irregular and regular patterns. In the former case, the disappearing process is gradual: irregular patterns loose their pulses one by one, jumping from manifold $\mathcal{M}_k$ to $\mathcal{M}_{k-1}$ ($k = N, \ldots, 1$). In contrast, regular, spatially periodic, patterns undergo catastrophic transitions in which either half or all pulses disappear (on a fast time scale) as the patterns pass through $\partial \mathcal{M}_N$. However, making a precise distinction between these two drastically different processes is quite subtle, since irregular $N$-pulse patterns that do not cross $\partial \mathcal{M}_N$ typically evolve towards regularity.
\end{abstract}

\section{Introduction}

The far from equilibrium dynamics of solutions to systems of reaction-diffusion equations -- patterns -- often has the character of interacting localised structures. This is especially the case when the diffusion coefficients of different components -- species -- in the system vary significantly in magnitude. This property makes the system {\it singularly perturbed}. Such systems appear naturally in ecological models; in fact, the presence of processes that vary on widely different spatial scales is regarded as a fundamental mechanism driving pattern formation in spatially extended ecological systems \cite{RvdK08}. Moreover, while exhibiting behaviour of a richness comparable to general -- non singularly perturbed -- systems, the multi-scale nature of singularly perturbed systems provides a framework by which this behaviour can be studied and (partly) understood mathematically.

In this paper, we consider the interactions of singular pulses in an extended Klausmeier model \cite{SH13,Sherratt-origins,Eric-Striped,Eric-Beyond-Turing},
\begin{equation}	
		\begin{cases}
			U_t & = U_{xx} + h_x U_x + h_{xx} U + a - U - UV^2, \\
			V_t & = D^2 V_{xx} - m V + UV^2,
		\end{cases} \label{eq:extKmodel1}
	\end{equation}
sometimes also called the generalised Klausmeier-Gray-Scott system \cite{Lottes,van2013rise}. This model is a generalization of the original ecological model by Klausmeier on the interplay between vegetation and water in semi-arid regions \cite{klausmeier1999regular} -- which was proposed to describe the appearance of vegetation patterns as crucial intermediate step in the desertification process that begins with a homogeneously vegetated terrain and ends with the non-vegetated bare soil state: the desert -- see \cite{deblauwe2012determinants,BookMeron,Retal04} and the references therein for observations of these patterns and their relevance for the desertification process. In \eqref{eq:extKmodel1}, $U(x,t)$ represents (the concentration of) water and $V(x,t)$ vegetation; for simplicity -- and as in \cite{Lottes,SH13,Sherratt-origins,Eric-Beyond-Turing,van2013rise} -- we consider the system in a 1-dimensional unbounded domain, i.e. $x \in \R$; parameter $a$ models the rainfall and $m$ the mortality of the vegetation. Since the diffusion of water occurs on a much faster scale than the diffusion -- spread -- of vegetation, the system is indeed -- and in a natural way -- singularly perturbed: the diffusion coefficient of water is scaled to 1 in \eqref{eq:extKmodel1}, so that the diffusion coefficient of the vegetation $D$ can be assumed to be small, i.e. $0 < D \ll 1$. The topography of the terrain is captured by the function $h : \R \rightarrow \R$. The derivative $h_x$ is a measure of the slope in \eqref{eq:extKmodel1} -- see Appendix~\ref{A:hx} for a derivation of this effect. Unlike in \cite{klausmeier1999regular}, we allow (some of) the parameters of (\ref{eq:extKmodel1}) to {\it vary in time or space}: we consider topography functions $h$ that may vary in $x$, and -- most importantly -- we study the impact of slow variations -- typically decrease -- in time of the rainfall parameter $a$: by considering $a = a(t)$ we incorporate the effect of changing environmental -- climatological -- conditions into the model. It is crucial for all analysis in this work that if $a(t)$ varies with $t$ it {\it decreases}, i.e. that the external conditions worsen -- see also \cite{SH13,Sherratt-origins,Eric-Striped,Eric-Beyond-Turing}.

\begin{figure}
\begin{subfigure}[t]{0.33 \textwidth}
\begin{center}
\includegraphics[width = \textwidth]{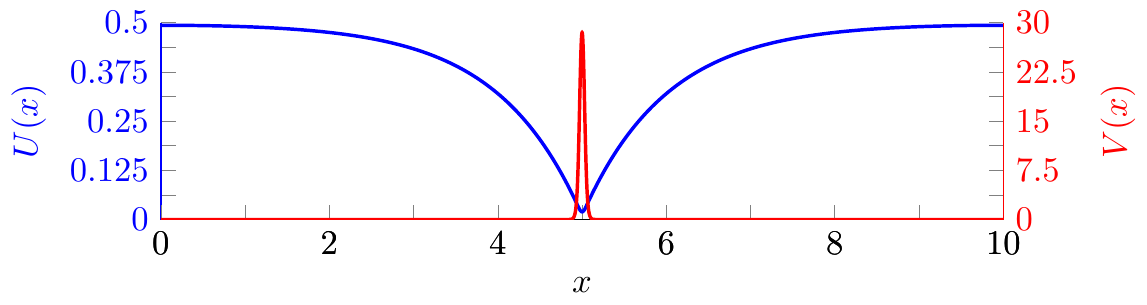}
\caption{A homoclinic $1$-pulse solution.}
\label{fig:intro-Irregular5Pulse}
\end{center}
\end{subfigure}	
\begin{subfigure}[t]{0.33 \textwidth}
\begin{center}
\includegraphics[width = \textwidth]{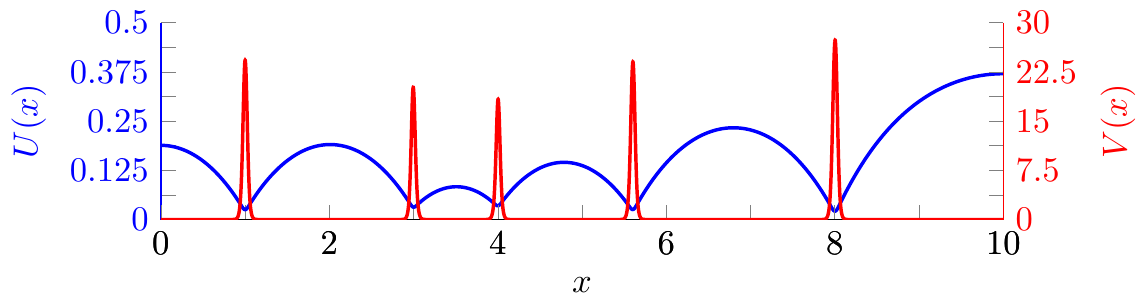}
\caption{An irregular $5$-pulse solution.}
\label{fig:intro-Irregular5Pulse}
\end{center}
\end{subfigure}	
\begin{subfigure}[t]{0.33 \textwidth}
\begin{center}
\includegraphics[width = \textwidth]{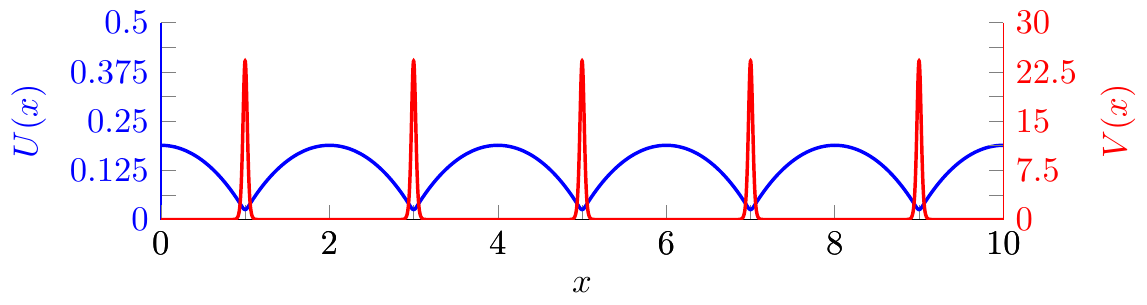}
\caption{A regular $5$-pulse solution.}
\label{fig:intro-Regular5Pulse}
\end{center}
\end{subfigure}	
\caption{Snapshots of several (multi-)pulse solutions of system \eqref{eq:extKmodel1} with $a = 0.5$, $m=0.45$, $h(x) \equiv 0$ and $D = 0.01$.}
\label{fig:intro-PulseSolutions}
\end{figure}

The pulse patterns studied in this paper -- see Figure~\ref{fig:intro-PulseSolutions} for some examples -- correspond directly to localised vegetation `patches'; trivially extending them in an $y$-direction leads to stripe patterns, the dominant structures exhibited by patchy vegetation covers on sloped terrains \cite{deblauwe2012determinants,klausmeier1999regular,Eric-Striped}. The central questions that motivated the research presented in this paper have their direct origins in ecological questions. Nevertheless, this paper focuses on fundamental issues in the dynamics of pulses in singularly perturbed reaction-diffusion systems with varying parameters. The ecological relevance of the insights obtained in the present work are subject of ongoing research. In that sense, the (alternative) name of generalised Klausmeier-Gray-Scott model \cite{Lottes,van2013rise} perhaps is a more suitable name for model (\ref{eq:extKmodel1}) in the setting of the present paper: by setting $h(x) \equiv 0$ -- i.e. in the ecological context of homogeneously flat terrains -- it reduces to the Gray-Scott model \cite{Pea93} that has served as paradigmatic model system for the development of our present day mathematical `machinery' by which pulses in singularly perturbed reaction-diffusion equations can be studied -- see \cite{chen2009oscillatory,dek1siam,dek2siam,doelman1998stability,Kolokolnikov2005PS,Kolokolnikov2005PSinGS} for research on pulse patterns in the Gray-Scott model and \cite{BjornRiccati,Split-Evans,DK03,kolokolnikov2005pulsesplitting,veerman2013pulses} for generalizations.
\\ \\
$N$-pulse patterns are solutions $(U((x,t),V(x,t))$ to~\eqref{eq:extKmodel1}, characterised by $V$-components that are exponentially small everywhere except for $N$ narrow regions in which they `peak': the $N$ pulses -- see Figure \ref{fig:intro-PulseSolutions} and notice that the heights of the pulses typically varies. In the setting of singularly perturbed reaction-diffusion models with constant coefficients, the evolution of $N$-pulse patterns can be regarded -- and studied -- as the {\it semi-strong interaction} \cite{DK03} of $N$ pulses. Under certain conditions -- see below -- the full infinite-dimensional PDE-dynamics can be reduced to an $N$-dimensional system describing the dynamics of the pulse locations $P_1(t) < P_2(t) < \ldots < P_N(t)$ -- see \cite{bellsky2013,chen2009oscillatory,DK03} and the references therein for different (but equivalent) methods for the explicit derivation of this system. Note that the heights of the pulses also vary in time, however, the pulse amplitudes are `slaved' to their (relative) locations. As starting point of our research, we show that this semi-strong pulse interaction reduction method can be -- straightforwardly -- generalised to systems like (\ref{eq:extKmodel1}) in which coefficients vary in time or space. We do so by following the matched asymptotics approach developed by Michael Ward and co-workers -- see \cite{chen2009oscillatory,chen2011stability,Kolokolnikov2005PS,Kolokolnikov2005,Kolokolnikov2005PSinGS,kolokolnikov2005pulsesplitting} and the references therein -- which also means that we apply -- when necessary -- the hybrid asymptotic-numerical approach of \cite{chen2011stability} in which the asymptotic analysis is sometimes `assisted' by numerical methods -- for instance when the `algebra' gets too involved or when a reduced equation can not be solved (easily) `by hand'.

This semi-strong interactions reduction mechanism has been rigorously validated -- by a renormalization group approach based on \cite{Pro02} -- for several specific systems \cite{bellsky2013,DKP,HeijsterFrontInt}. It is established by the approach of \cite{bellsky2013,DKP,HeijsterFrontInt} -- and for the systems considered in these papers -- that there indeed is an approximate $N$-dimensional manifold $\mathcal{M}_N$ (within an appropriately chosen function space in which the full PDE-dynamics takes place) that is attractive and nonlinearly stable and that the flow on $\mathcal{M}_N$ is (at leading order) governed by the equations for the pulse locations $P_j(t)$, $j = 1, ..., N$. However, this validity result only holds if the {\it quasi-steady spectrum} -- see Figure \ref{fig:intro-method:q-s} --  associated to the $N$-pulse pattern can be controlled. The quasi-steady spectrum is defined as the {\it approximate} spectrum associated to a `frozen' $N$-pulse pattern. Due to the slow evolution of the pattern -- and the singularly perturbed nature of the problem -- this spectrum can be approximated explicitly (by methods based on the literature on stationary pulse patterns, see \cite{chen2009oscillatory,veerman2013pulses} and the references therein). By considering (slow) time as a parameter, the elements of the quasi-steady spectrum trace orbits through the complex plane, driven by the pulse locations $P_j(t)$ and, in the case of (\ref{eq:extKmodel1}), by the slowly changing value of $a(t)$. The manifold $\mathcal{M}_N$ is attractive only when this spectrum is in the left half of the complex plane: the proof of the validity result breaks down when there is no spectral gap of sufficient width between the quasi-steady spectrum and the imaginary axis. Thus, the quasi-steady spectrum -- approximately -- determines a boundary of $\mathcal{M}_N$.

\begin{figure}
\centering
\begin{subfigure}[t]{0.27 \textwidth}
\begin{center}
\includegraphics[width = \textwidth]{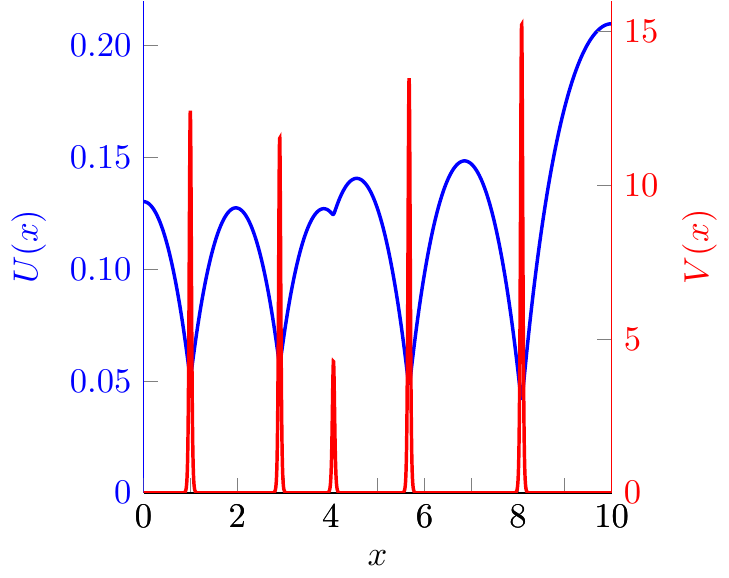}
\caption{Before disappearance.}
\label{fig:intro-method:before}
\end{center}
\end{subfigure}
\begin{subfigure}[t]{0.27 \textwidth}
\begin{center}
\includegraphics[width = \textwidth]{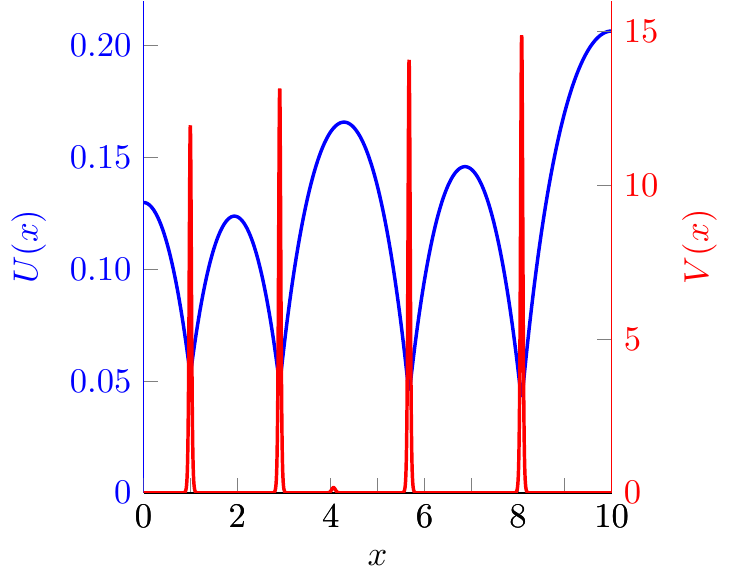}
\caption{After disappearance.}
\label{fig:intro-method:after}
\end{center}
\end{subfigure}
\begin{subfigure}[t]{0.20 \textwidth}
\begin{center}
\includegraphics[width = \textwidth]{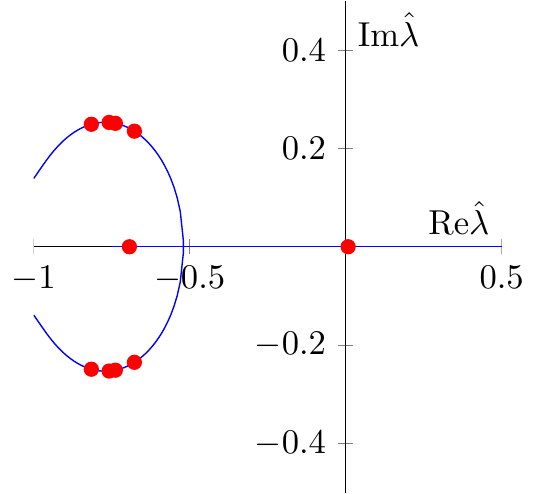}
\caption{Quasi-steady spectrum.}
\label{fig:intro-method:q-s}
\end{center}
\end{subfigure}
\begin{subfigure}[t]{0.24 \textwidth}
\begin{center}
\includegraphics[width = \textwidth]{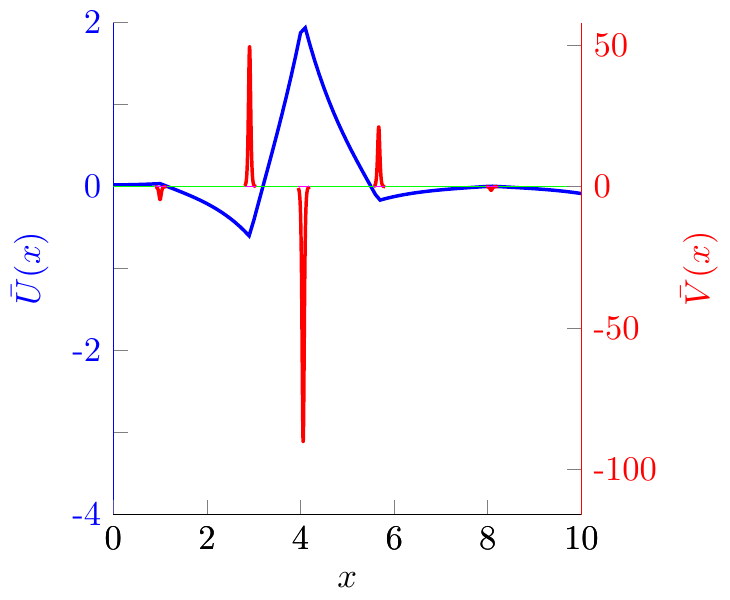}
\caption{The critical eigenfunction.}
\label{fig:intro-method:eigenfunction}
\end{center}
\end{subfigure}
\caption{(a), (b): Simulations of (\ref{eq:extKmodel1}) with $a(t) = 0.5 - 5 \cdot 10^{-4}\ t$, $m = 0.45$, $H = 0$, $L = 10$, $P(0) = (1,3,4,5.6,8)^T$,  $D = 0.01$, just before and after the 3rd pulse disappears (at $a(t) \approx 0.28$). (c): The red dots -- that must travel over the blue `skeleton structure' (section \ref{sec:linstabsections}) -- indicate the analytically determined quasi-steady spectrum associated to the pattern in (a). (d): The (analytically determined) eigenfunction associated to the critical (quasi-steady) eigenvalue in (c).}
\label{fig:intro-method}
\end{figure}

The boundary $\partial \mathcal{M}_N$ in general does not act as a threshold for the flow on $\mathcal{M}_N$; on the contrary, {\it an evolving $N$-pulse pattern may evolve towards -- and subsequently through -- the boundary $\partial \mathcal{M}_N$} -- as elements of the quasi-steady spectrum travel towards the imaginary axis. Or equivalently, in the case of parameters that vary in time, {\it the boundary $\partial \mathcal{M}_N$ may evolve towards the pulse pattern}.
\\ \\
In this paper, we do {\it not} consider the issue of the rigorous validation of the semi-strong reduction method -- although we do remark that the methods of \cite{bellsky2013} a priori seem sufficiently flexible to provide validity results for $N$-pulse dynamics in (\ref{eq:extKmodel1}) with non-homogeneous parameters (in fact, the results of \cite{bellsky2013} already cover specific parameter combinations in (\ref{eq:extKmodel1}) -- with $a$ constant and $h(x) \equiv 0$ -- see Figure \ref{fig:intro-SizeAssumptionsStability}). Here, we explore -- in as much (formal) analytic detail as possible -- the dynamics of $N$-pulse patterns near {\it and beyond} the boundary of the (approximate) invariant manifold $\mathcal{M}_N$. In other words, we intentionally consider situations in which we know that the rigorous theory cannot hold. As noted above, this is partly motivated by ecological issues: the final steps in the process of desertification are -- conceptually -- governed by interacting pulses -- vegetation patches. Under worsening climatological circumstances, these patches may either `disappear' in a gradual fashion -- patches wither and turn to bare soil {\it one by one} -- or  catastrophically -- all patches in a large region disappear {\it simultaneously} -- see \cite{BHM12,HW10,Retal04,Eric-Beyond-Turing} and the references therein. These types of transitions correspond to $N$-pulse patterns crossing through different components of the boundary $\partial \mathcal{M}_N$ of $\mathcal{M}_N$: the nature of these components of $\partial \mathcal{M}_N$ -- and especially the associated dynamics of pulse patterns crossing through the component -- clearly varies significantly. This leads us directly to the mathematical themes we explore here,
\\ \\
{\it Is it possible to analytically follow an $N$-pulse pattern as it crosses the boundary of a manifold $\mathcal{M}_N$? Can we predict the $M$-pulse pattern that emerges as the pattern `settles' on a lower dimensional manifold $\mathcal{M}_{M}$ -- and especially the value $M < N$? More specifically, can we distinguish between $N$-pulse patterns for which $M = 0$ (a catastrophic regime shift), $M = N/2$ (a period doubling) and $M = N-1$ (a gradual decline)?}
\\ \\
The essence of our approach is represented by Figure~\ref{fig:intro-method}. In Figures~\ref{fig:intro-method:before} and~\ref{fig:intro-method:after} two snapshots of a (full) PDE simulation of a (originally) $5$-pulse pattern is shown, just before and just after the 3rd pulse has disappeared, i.e before it `falls off' $\mathcal{M}_5$ and after it `lands' on $\mathcal{M}_4$. In Figure \ref{fig:intro-method:q-s}, the quasi-steady spectrum associated to the $5$-pulse pattern of Figure \ref{fig:intro-method:before} -- i.e. the pattern close to the boundary of is $\mathcal{M}_5$ -- is shown: as expected, a quasi-steady eigenvalue has approached the imaginary axis. The spectral configuration of Figure \ref{fig:intro-method:q-s} is determined by asymptotic analysis, an analysis that simultaneously provides the (leading order) structure of the (critical) eigenfunction associated to the critical eigenvalue -- see section \ref{sec:linstabsections}. This eigenfunction is given in Figure \ref{fig:intro-method:eigenfunction}. By construction, it describes the leading order structure of the (linearly) `most unstable perturbation' that starts to grow as the pattern passes through $\partial \mathcal{M}_5$. The eigenfunction is clearly localised around the -- disappearing -- 3rd pulse: the analytically obtained structure indicates that the unstable perturbation will mainly affect the 3rd pulse. By {\it formally} extrapolating this observation based on the {\it linear} asymptotic analysis -- i.e. the information exhibited by Figures~\ref{fig:intro-method:q-s} and~\ref{fig:intro-method:eigenfunction} that is based on the state of the 5-pulse pattern before it falls off $\mathcal{M}_5$ -- we are inclined to draw the {\it nonlinear} conclusion that the destabilised 3rd pulse will `disappear' as $\partial \mathcal{M}_5$ is crossed, while the other 4 pulses persist: $M = 4 = N-1$. The PDE-simulation of Figure~\ref{fig:intro-method:after} shows that this linear extrapolation indeed correctly predicts the full dynamics of (\ref{eq:extKmodel1}).

We develop a hybrid asymptotic-numerical method that describes the evolution of an $N$-pulse pattern by the  reduced $N$-dimensional system for the pulse locations $P_j(t)$ as long as the pulse pattern is in the interior of (approximate) invariant manifold $\mathcal{M}_N$. With the pulse locations as input, we (analytically) determine the associated (evolving) quasi-steady spectrum, and thus know whether the pulse configuration indeed is in this interior, i.e. bounded away from $\partial \mathcal{M}_N$. As elements of the quasi-steady spectrum approach the imaginary axis -- i.e. as the pattern approaches $\partial \mathcal{M}_N$ -- the method follows the above described -- relatively simple -- extrapolation procedure: based on the (approximate) structure of the critical eigenfunction(s) corresponding to the critical element(s) of the quasi-steady eigenvalues that end up on the imaginary axis, it is -- automatically -- decided which pulse(s) are eliminated and thus what is the value of $M < N$. Next, the process is continued by following the dynamics of the $M$-pulse configuration on $\mathcal{M}_M$, that has the locations of the $M$ remaining pulses as $\partial \mathcal{M}_N$ is crossed as initial conditions. Etcetera. Thus, this method provides a formal way to follow the PDE dynamics of an evolving $N$-pulse pattern throughout the `desertification' process of disappearing pulses, or -- equivalently -- as the pulse pattern falls off and subsequently lands on a sequence of invariant manifolds $\mathcal{M}_{N_i}$ of decreasing dimension $N_i$.

A priori, one would guess that this method cannot work -- even if there would be rigorous validation results on the reduced dynamical systems on the finite-dimensional manifolds $\mathcal{M}_{N_i}$. First, one can in principle not expect that the structure of the most critical eigenfunction always is as clear-cut as in Figure \ref{fig:intro-method:eigenfunction}: a priori one expects that the `automatic' decision on which pulse(s) to eliminate -- and thus how many -- must be incorrect in many situations. Moreover, it is not at all clear that the (fast) nonlinear dynamics that takes the pattern from $\mathcal{M}_N$ to $\mathcal{M}_M$ indeed only eliminates these `most vulnerable pulses'. For instance, if the destabilization is induced by a pair of complex conjugate (quasi-steady) eigenvalues, our method automatically assumes that the associated `quasi-steady Hopf bifurcation' is {\it subcritical} -- i.e. that there is no (stable) periodic oscillating pulse behaviour beyond the bifurcation; in fact, even if the bifurcation is subcritical, our method implicitly assumes that the oscillating process by which the affected pulse disappears is so fast, that it does not influence the other pulses and thus can be completely neglected.

Nevertheless, we found that this method is remarkably successful. Figure~\ref{fig:intro-SixPulsesOnHill:PDE} shows a full PDE simulation of a 5-pulse configuration `moving uphill', i.e. extended Klausmeier model~\eqref{eq:extKmodel1} in the (Klausmeier) setting of a constant slope, $h(x) =x$, on a bounded domain (with homogeneous Neumann boundary conditions). One by one, 3 pulses disappear from the system, eventually leading to a stationary stable 2-pulse pattern. Figure \ref{fig:intro-SixPulsesOnHill:ODE} shows the evolution of the same 5-pulse configuration (at $t=0$) as described by our -- finite-dimensional -- method: the pulse configuration `jumps' from $\mathcal{M}_5$ to $\mathcal{M}_4$ and $\mathcal{M}_3$, eventually settling down in a stable critical point of the 2-dimensional dynamical system that governs the flow on $\mathcal{M}_2$. This is quite a slow -- and nontrivial -- process and it takes quite a long time before the system reaches equilibrium, nevertheless, the ODE reduction method not only provides a qualitatively correct picture, it is remarkably accurate in a quantitative sense.

This latter observation is even more remarkable, since our approach is by an asymptotic analysis and thus based on the assumption that a certain parameter -- or parameter combination -- is `sufficiently small'. Nevertheless our methods remain valid for `relatively large values' of the `asymptotically small parameter'. This is not atypical for asymptotically derived insights. It yields another motivation to indeed set out to obtain rigorous results on the dynamics of systems like (\ref{eq:extKmodel1}): in practice, such results are expected to be relevant way beyond the necessary `for $\eps$ sufficiently small' caveat.

\begin{figure}
\begin{center}
\begin{subfigure}[t]{0.4 \textwidth}
\begin{center}
\includegraphics[width = \textwidth]{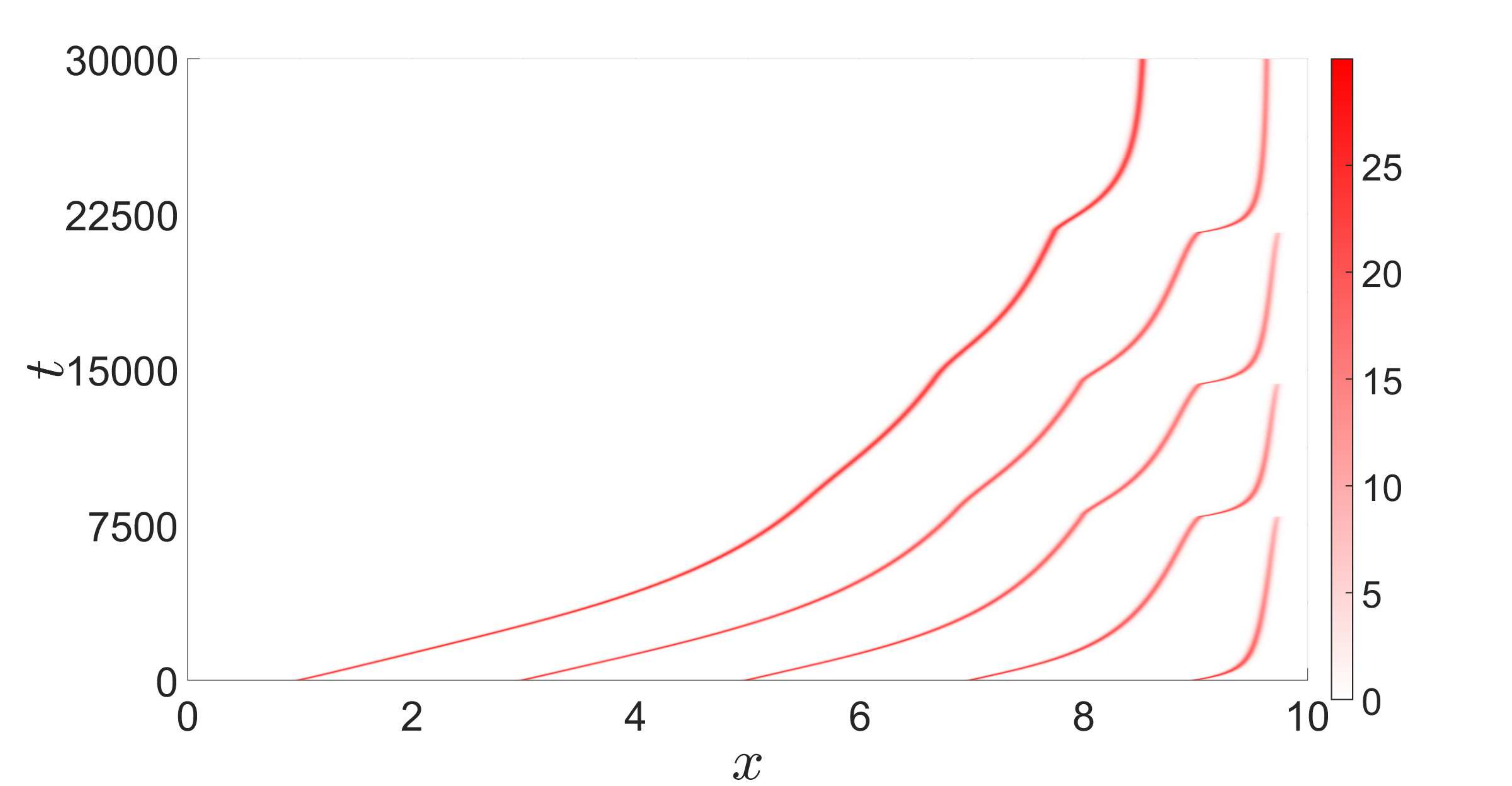}
\caption{PDE simulation.}
\label{fig:intro-SixPulsesOnHill:PDE}
\end{center}
\end{subfigure}
\begin{subfigure}[t]{0.4 \textwidth}
\begin{center}
\includegraphics[width = \textwidth]{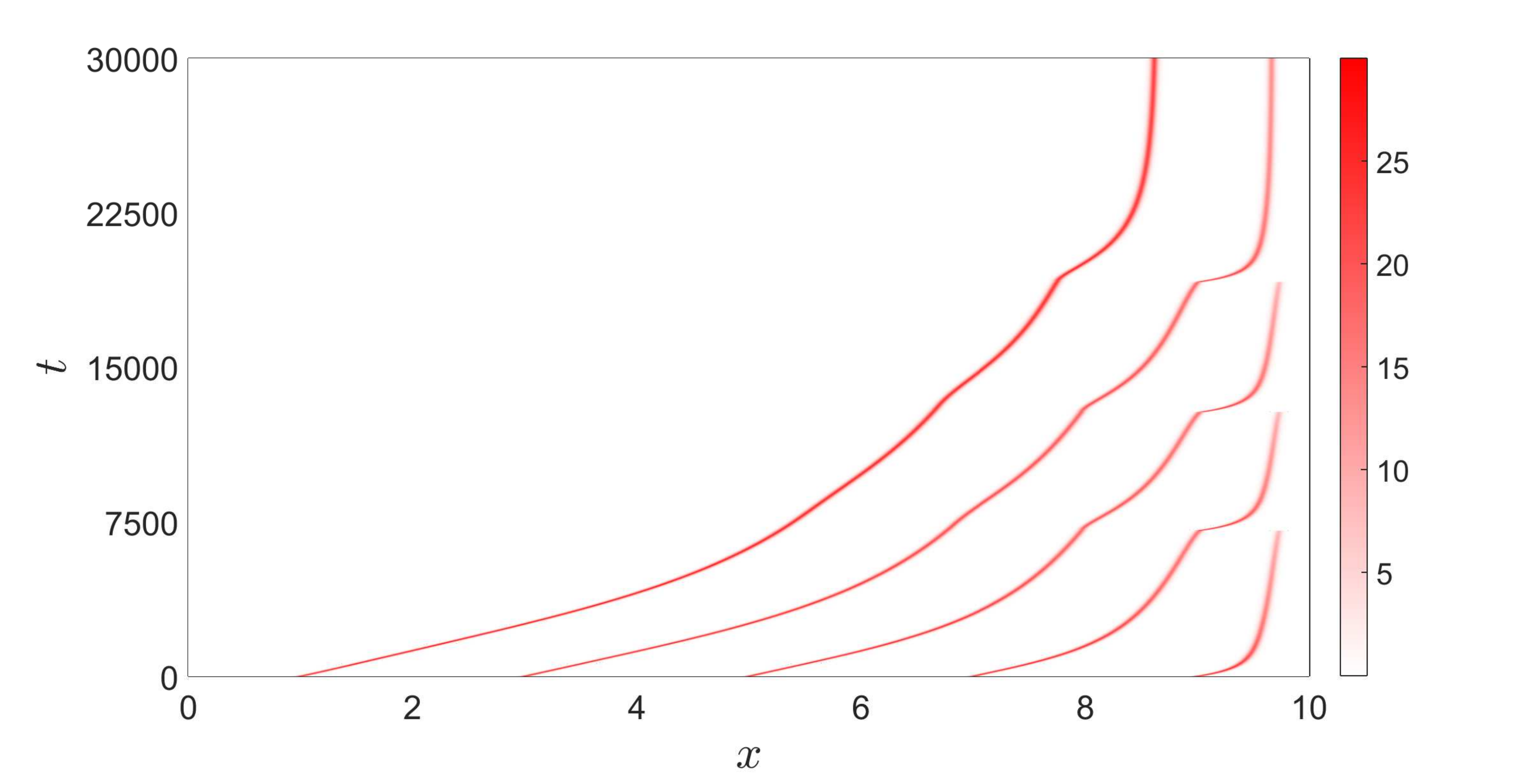}
\caption{ODE simulation.}
\label{fig:intro-SixPulsesOnHill:ODE}
\end{center}
\end{subfigure}
\caption{The evolution of a 5-pulse pattern in the extended Klausmeier model~\eqref{eq:extKmodel1} represented by the locations of the pulses (with $a = 0.5$, $m = 0.45$, $h(x) =x$, $D = 0.01$, $L = 10$). (a) A full PDE simulation. (b) The hybrid asymptotic-numerical ODE method developed in this work.}
\label{fig:intro-SixPulsesOnHill}
\end{center}
\end{figure}

The end-goal of the numerical simulations we present -- see section \ref{sec:NumericalSimulations} -- is to test our method, both to get a (formal) insight in its limitations, as well as to isolate typical behaviour of pulse configurations that may be formulated as conjectures -- i.e. as challenges for the development of the theory. As an example, we mention the `generalised Ni conjecture' \cite{Ddestab,Ni-conjecture} of section \ref{sss:irregular} (for systems with $h(x) \equiv 0$): {\it When a multi-pulse pattern is sufficiently irregular, the localised $V$-pulse with the lowest maximum is the most unstable pulse, and thus the one to disappear first}. In fact, one could claim that at a formal level, the evolution of sufficiently irregular $N$-pulse patterns can be understood by (successive applications of) this conjecture -- and thus be described accurately by our reduction method. However, even when the initial conditions form an irregular $N$-pulse pattern, the situation becomes more complex than that, since the reduced $N$-dimensional dynamics typically evolve towards a critical point on $\mathcal{M}_N$. In fact, our study indicates that $N$-pulse patterns (on bounded domains) always evolve to one specific configuration -- in the Gray-Scott setting of flat terrains, i.e. $h(x) \equiv 0$, this is a regularly spaced (spatially periodic) $N$-pulse pattern. The final pattern is less regular if $h(x) \not \equiv 0$ -- see the stable 2-pulse pattern of Figure~\ref{fig:intro-SixPulsesOnHill}.

The evolution towards spatially periodic patterns induces a mechanism that challenges our method. For irregular patterns, the elements of the quasi-steady spectrum typically `spread out' (over a certain skeleton structure, see Figure \ref{fig:intro-method:q-s} and section \ref{sec:linstabsections}). However, these elements might cluster together as the pattern becomes more and more regular (which agrees with the spectral analysis of spatially periodic patterns in Gray-Scott/Klausmeier type models, see \cite{doelman1998stability,Lottes}). Therefore, it gets harder to isolate the critical (quasi-steady) eigenvalue that induces the destabilization. Moreover, the structure of the associated eigenfunctions also changes significantly: in the irregular setting these have a structure that is centered around one well-defined pulse location (as in Figure \ref{fig:intro-method:eigenfunction}) -- which makes them very suitable for the application of our method; in the periodic case, the eigenfunctions have a more global structure. Nevertheless, as the regularised $N$-pulse pattern approaches the boundary of $\mathcal{M}_N$, two most critical quasi-steady eigenvalues can be distinguished  -- i.e. there typically are two (quasi-steady) eigenvalues that may cause the destabilization. The associated two critical eigenfunctions are also (almost) periodic, either with the same period of the underlying pattern, or with twice that period -- which is in agreement with analytical insights in the destabilization mechanisms of `perfect' spatially periodic patterns \cite{BjornEigenvalues,Ddestab,Hopf-Dances} (see also the two conjectures in section~\ref{NUM-flatTerrain-REG}). These critical eigenfunctions are plotted in Figure~\ref{fig:intro-eigenfunctions-REG} for a stationary regular 2-pulse pattern for $h(x) \equiv 0$ and $a$ fixed near its bifurcation value -- i.e. in the classical constant coefficients setting of (\ref{eq:extKmodel1}). The eigenfunction in Figure \ref{fig:intro-FullCollapse} has the same periodicity as the underlying pattern, it represents the catastrophic `full collapse' scenario in which all pulse disappear simultaneously. Of course, this statement is once again a fully nonlinear extrapolation of completely linear insight, but it is -- once again -- backed up by our numerical simulations: also in the regular case, the linear mechanisms are good predictors for the fast transitions between invariant manifolds.

This nonlinear extrapolation of a linear mechanism also works for the other critical eigenfunction represented by Figure \ref{fig:intro-PeriodDoubling}, which induces a period doubling bifurcation in which half of the pulses of an $N$ pulse pattern disappear. However, in this case -- that is quite dominant in simulations of desertification scenarios \cite{Eric-Striped,Eric-Beyond-Turing} -- our method faces an intrinsic problem, that gets harder the more regular the pattern becomes: if the number of pulses $N$ is odd, our method predicts that `half of the pulses' disappear, but it cannot decide whether the $N$-pulse configuration jumps from $\mathcal{M}_N$ to $\mathcal{M}_{(N+1)/2}$ -- in which all $(N-1)/2$ even numbered pulses disappear -- or from $\mathcal{M}_N$ to $\mathcal{M}_{(N-1)/2}$ -- in which the even numbered pulses are the surviving ones. A similar problem occurs in the jump from $\mathcal{M}_N$ to $\mathcal{M}_{N/2}$ for $N$ even: our method cannot predict whether the even or the odd numbered pulses survive. Nevertheless, also in this case our method is doing better than could be expected; moreover, also in direct PDE simulations, the resolution of this parity issue seems extremely sensitive on initial conditions.
\\ \\
The set-up of this paper is as follows. In section~\ref{sec:existence}, we first perform the PDE to ODE reduction for $N$-pulse patterns in (\ref{eq:extKmodel1}) with -- in its most general setting --  $a=a(t)$ varying in time and $h=h(x)$ varying in space (on unbounded domains and on bounded domains with various kinds of boundary conditions). As a result we obtain explicit expressions for the $N$-dimensional -- or $N-1$-dimensional\footnote{On unbounded domains or domains with periodic boundary conditions the ODE is essentially $N-1$-dimensional, as only the distances between the pulses is relevant, thus reducing the dimension by $1$.} -- systems that describe the evolution of the pulse locations $P_j(t)$, and thus of the $N$-pulse pattern on $\mathcal{M}_N$. Subsequently, the flow on $\mathcal{M}_N$ is studied -- the critical points and their characters are determined analytically; as a consequence, the special role of the spatially periodic patterns -- as attractive fixed points -- can be identified. These results need to be supplemented with an analysis of the stability of the manifold $\mathcal{M}_N$, especially since the analysis of section ~\ref{sec:existence} is not equipped to distinguish the boundaries of $\mathcal{M}_N$ -- i.e. it ignores the process of pulse patterns falling off $\mathcal{M}_N$. This is the topic of section~\ref{sec:linstabsections} in which $N$-pulse solution are frozen and their quasi-steady spectrum -- and thus the boundary of $\mathcal{M}_N$ -- is determined. A central part of the analysis is dedicated to determining the skeleton structure on -- or better: near -- which the quasi-steady eigenvalues must lie (see Figure \ref{fig:intro-method:q-s}). Moreover, the (linearised) nature of the bifurcations that occur when specific components of $\partial \mathcal{M}_N$ are crossed is studied. Next, in section~\ref{sec:NumericalSimulations}, we first numerically check the validity of our asymptotic analysis, then set up our hybrid asymptotic-numerical method -- based on the analysis of sections \ref{sec:existence} and \ref{sec:linstabsections} -- and subsequently extensively test its `predictions' against full PDE-simulations. We find that the asymptotic analysis is correct for parameter values beyond the reaches of current rigorous theory. Moreover, we observe that our method -- that is based on direct extrapolations of linear insights -- works better than a priori could be expected, but also couple this to a search for the limitations of this approach. Based on these tests and simulations, we formulate general conjectures on the nature of multi-pulse dynamics generated by models as (\ref{eq:extKmodel1}). Finally, we briefly discuss the implications of our findings and indicate future lines of research in the concluding section \ref{sec:disc}.

\begin{figure}
\begin{subfigure}[t]{0.49 \textwidth}
\begin{center}
\includegraphics[width = \textwidth]{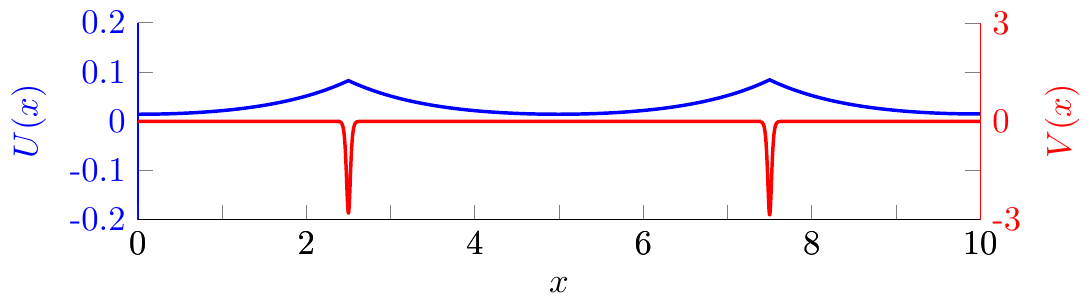}
\caption{A full collapse eigenfunction (eigenvalue $\hat{\lambda} = -0.087$).}
\label{fig:intro-FullCollapse}
\end{center}
\end{subfigure}	
\begin{subfigure}[t]{0.49 \textwidth}
\begin{center}
\includegraphics[width = \textwidth]{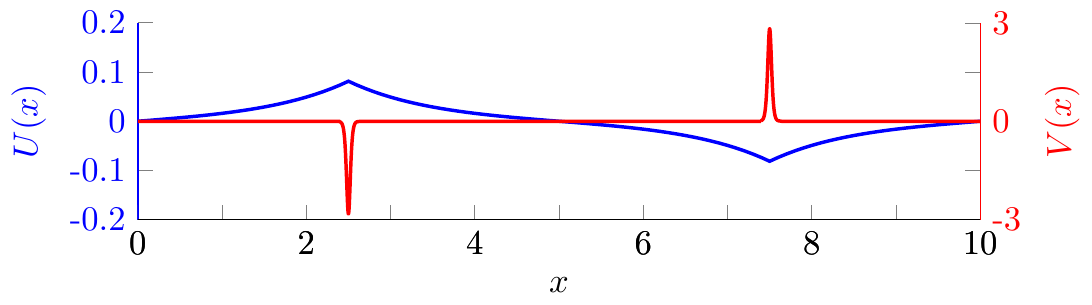}
\caption{A period doubling eigenfunction (eigenvalue $\hat{\lambda} = -0.063$).}
\label{fig:intro-PeriodDoubling}
\end{center}
\end{subfigure}	
\caption{The 2 critical eigenfunctions of a regular $2$-pulse pattern of extended Klausmeier model (\ref{eq:extKmodel1}) on a domain with periodic boundaries with $a \equiv 0.19187$ (near bifurcation), $m = 0.45$, $h(x) \equiv 0$ and $D = 0.01$.}
\label{fig:intro-eigenfunctions-REG}
\end{figure}

\subsection{Size assumptions}

The asymptotic analysis presented in this paper does not hold for all magnitudes of the parameters $a$, $m$, $D$ and all height functions $h$. We therefore need to make several assumptions on the (relative) magnitudes of the parameters in (\ref{eq:extKmodel1}). These assumptions are listed here, together with the type of bifurcation that occurs when these assumptions are violated.

\begin{itemize}
	\item[(A1)] $\frac{a^2}{m^2} \ll 1$ [Pulse Splitting bifurcation]
	\item[(A2)] $\frac{D a^2}{m \sqrt{m}} \ll 1$ [Travelling Wave bifurcation]
	\item[(A3)] $\frac{m \sqrt{m} D}{a^2} \ll 1$ [Saddle-Node bifurcation]
	\item[(A4)] $\frac{m^2 D}{a^2} \ll 1$ [Hopf bifurcation]
	\item[(A5)] $\frac{D m \sqrt{m}}{a^2} h_x(x) \ll 1$ and $\frac{a^2}{m^2} \left(\frac{D m \sqrt{m}}{a^2}\right)^2 h_{xx}(x) \ll 1$ for all $x \in \R$ [Saddle-Node bifurcation]
	\item[(A6)] $\frac{m^2 D}{a^2} h_x(x) \ll 1$ for all $x \in \R$ [Hopf bifurcation].
\end{itemize}

Previous studies of the Gray-Scott system indicate the necessity of three size assumptions to ensure the existence of (one-)pulse solutions~\cite{dek2siam,Lottes,chen2009oscillatory}. The assumptions found in those previous studies can be directly linked\footnote{A handy conversion table between different scalings of the Gray-Scott model can be found in~\cite[section 2.2]{Lottes}.} to our assumptions (A1)-(A3). In Figure~\ref{fig:intro-SizeAssumptionsExistence} we have visualised the assumptions on parameters $a$ and $m$ that follow from (A1)-(A3). Asymptotic stability analysis has shown that a pulse solution is stable if it satisfies an additional fourth size assumption, which corresponds to our assumption (A4). We have also visualised the assumptions on $a$ and $m$ that follow from the assumptions (A1)-(A4) in Figure~\ref{fig:intro-SizeAssumptionsStability}. Finally, the assumptions on the height function $h$ in assumptions (A5) and (A6) are new, and include the case studied in~\cite{Lottes} (but are more general). These guarantee that the height function $h$ does not change too rapidly, i.e. $h$ changes on a slower scale than the $V$-pulse does. This ensures that the standard `flat-terrain' (i.e. $h(x) \equiv 0$) existence theory can be reproduced almost directly.

\begin{figure}
	\centering
		\begin{subfigure}[t]{0.35 \textwidth}
			\centering
			\includegraphics[width=\textwidth]{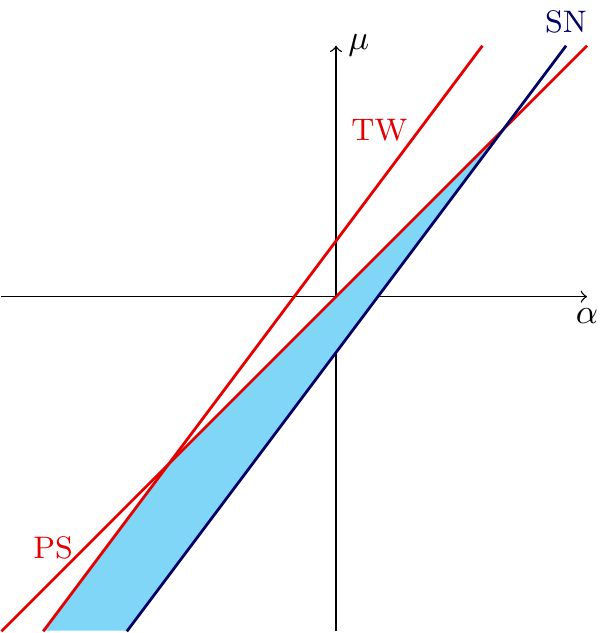}

			\caption{Size assumptions Existence}
			\label{fig:intro-SizeAssumptionsExistence}
		\end{subfigure}
~
		\begin{subfigure}[t]{0.35 \textwidth}
			\centering
			\includegraphics[width=\textwidth]{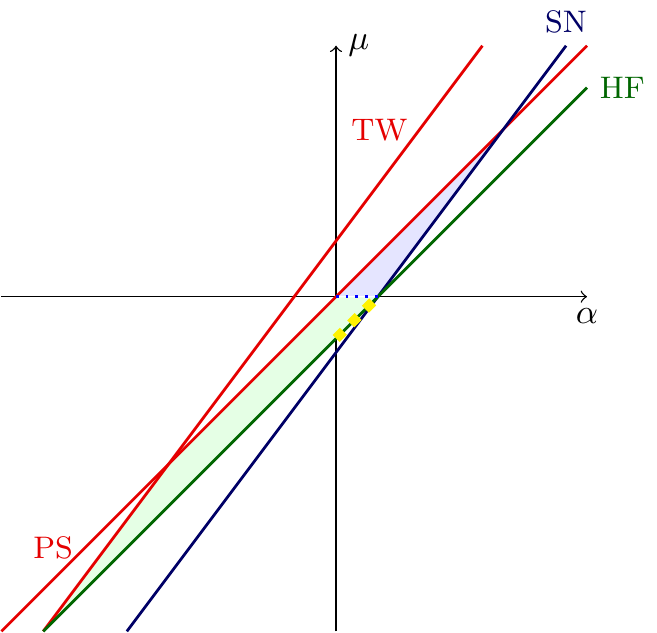}

			\caption{Size assumptions Stability}
			\label{fig:intro-SizeAssumptionsStability}
		\end{subfigure}
	\caption{Graphical summary of size assumptions (A1)-(A3) in (a) and assumptions (A1)-(A4) in (b). Here $\mu$ and $\alpha$ denote the size of $m$ respectively $a$ in order of magnitude of $0 < D \ll 1$. That is, $m = \mathcal{O}(D^\mu)$ and $a = \mathcal{O}(D^\alpha)$. This positions the pulse splitting bifurcation (PS) on the line $\mu = \alpha$, the Travelling-Wave bifurcation (TW) on $\mu = \frac{2}{3}(1+2\alpha)$, the Saddle-Node bifurcation (SN) on $\mu = \frac{2}{3}(2\alpha-1)$ and the Hopf bifurcation (HF) on $\mu = \alpha - \frac{1}{2}$. The coloured-in region in (a) indicate the region in which pulse solutions exist (under the additional assumptions (A5)). The coloured-in regions in (b) indicate the region in which stable pulse solutions exist (under the additional assumptions (A5),(A6)). We have also plotted the line $\mu =0$, which indicates the boundary between the cases $m \gg 1$ and $m \ll 1$ which becomes relevant in the distinction between coupled and decoupled stability problem in our linear stability study in section~\ref{sec:linstabsections}. The dashed yellow line (on the Hopf line) indicates the scaling regime for which validity of the ODE reduction has been proven~\cite{bellsky2013}.}
	\label{fig:intro-SizeAssumptions}
\end{figure}

In principle assumptions (A3) and (A4) can be extended to include the $\mathcal{O}(1)$ cases. In fact, to study the bifurcations that occur when the rainfall $a$ is decreased, it is necessary to include these cases. This leads to the alternative assumptions (A3') and (A4') which are stated below.

\begin{itemize}
	\item[(A3')] $\frac{m \sqrt{m} D}{a^2} \leq \mathcal{O}(1)$
	\item[(A4')] $\frac{m^2 D}{a^2} \leq \mathcal{O}(1)$
\end{itemize}
Note that assumption (A3') corresponds to the so-called `low feed-rate regime' in~\cite{Sun2005, Kolokolnikov2005}


\section{PDE to ODE reduction}\label{sec:existence}

In this section we study the dynamical movement of a $N$-pulse solution to the scaled extended Klausmeier model~\eqref{eq:extKmodel1}. We assume that there are $N$ localised vegetation $V$-pulses at positions $P_1(t) < P_2(t) < \ldots < P_N(t)$, as depicted in Figure~\ref{fig:sketchOuterRegions}. Depending on the domain of our problem we may put additional requirements on the first and last positions (e.g. $0 < P_1(t)$ and $P_N(t) < L$ on the bounded domain $[0,L]$). The positions of the $N$ pulses are not fixed in time. In fact, the $j$-th pulse turns out to move with a time-dependent movement speed $\hat{c}_j(t) = \frac{dP_j(t)}{dt}$ so that its location is given by $P_j(t) = \int_0^t \hat{c}_j(s) ds + P_j(0)$. Our goal is to derive an ODE that describes the evolution of the locations of these pulses, that is, to find expressions for the speeds $\hat{c}_j(t)$. To do so, we first need to find the approximate form of a $N$-pulse solution to~\eqref{eq:extKmodel1}. For this, we divide the domain in several regions: near each pulse we have an inner region and between pulses we have outer regions. Note that in the context of geometric singular perturbation theory these regions are called fast (the inner regions) respectively slow (the outer regions).

We follow the asymptotic approach developed by Michael Ward and co-workers -- see \cite{chen2009oscillatory,chen2011stability,Kolokolnikov2005PS,Kolokolnikov2005,Kolokolnikov2005PSinGS,kolokolnikov2005pulsesplitting} and references therein -- to find approximate solutions in the $N$ inner regions and in the $N+1$ outer regions. In the outer regions we find $V = 0$ and in the inner regions we find $U$ to be constant (both to leading order). A combination of a Fredholm condition and the matching of the inner and outer solution at the pulse locations then gives us the speed of the $j$-th pulse as a function of the solution $U$ in the outer regions \cite{chen2009oscillatory}. The latter is, in the end, determined by $N+1$ linear ODEs that are coupled via internal boundary conditions at all the pulse locations. Therefore we find a pulse-location ODE that depends only on the (current) positions of the pulses. Hence this ODE-description is a reduction of the infinite-dimensional flow of the PDE to a finite-dimensional flow on a $N$-dimensional\footnote{On unbounded domains or domains with periodic boundary conditions this manifold is essentially $N-1$-dimensional, as only the distances between pulses matters, thus reducing the dimension of the manifold by $1$.} manifold $\mathcal{M}_N$ on which $N$-pulses live.

After we have found this ODE description, we study the dynamics of generic $N$-pulse configurations in section~\ref {sec:Pulse-Location-ODE-(A3)} and section~\ref{sec:pulseODE-SaddleNode}. Here the difference between assumption (A3) and (A3') and the need for a hybrid aymptotic-numerics approach becomes apparent: in the former case analytical results can be found, whereas numerics are necessary to study the possibilities in the latter case. Note that assumptions (A4) and (A6) are not needed for the analysis in this section.

\subsection{The inner regions}

We start inspecting the inner regions of the $N$-pulse solution. To zoom in to the $j$-th inner region, close to $x = P_j(t)$, we introduce the stretched traveling wave coordinate centered around $P_j(t)$
\begin{equation}
	\xi_j = \frac{\sqrt{m}}{D} (x-P_j(t)) = \frac{\sqrt{m}}{D} \left( x - P_j(0) - \int_0^t \hat{c}_j(s) ds \right).
\label{eq:ex-inner-scaling-xi}
\end{equation}
Note that by assumptions (A3) and (A1) this is a stretched coordinate since $\frac{D}{\sqrt{m}} \leq \frac{a^2}{m^2} \ll 1$. We will denote this $j$-th inner region by $I_j^{in}$. As is common practice in geometric singular perturbation theory, we explicitly define $I_j^{in}$ by assuming that $\xi_j \in [-\frac{1}{\sqrt{\varepsilon}},\frac{1}{\sqrt{\varepsilon}}]$, with $\varepsilon = \frac{a}{m}$.

Following the scalings introduced in~\cite{chen2009oscillatory, dek2siam, Eric-Striped} we set $\hat{c}_j(t) =  \frac{D a^2}{m \sqrt{m}} c_j(t)$ where $c_j(t) = \mathcal{O}(1)$. By assumption (A2) we thus have $\hat{c}_j(t) \ll 1$, i.e. pulses move only slowly in time. We can thus use a quasi-steady approximation and treat $t$ as a parameter in our analysis (cf.~\cite{dek1siam,dek2siam, Sun2005, chen2009oscillatory}). At the pulse location we also need to scale $U$ and $V$. Again following the previously mentioned scalings~\cite{chen2009oscillatory, dek2siam, Eric-Striped}, it turns out we need to scale these in the inner regions as
\begin{equation}
U = \frac{m\sqrt{m}D}{a} u; \hspace{1cm} V = \frac{a}{\sqrt{m}D} v.\label{eq:ex-inner-scaling}
\end{equation}
Putting in these scalings gives us the following problem for the inner region at the $j$-th pulse:
\begin{equation}
	\begin{cases}
- \frac{a^2}{m^2} \frac{D m \sqrt{m}}{a^2} \frac{D a^2}{m \sqrt{m}} c_j(t) u_j'
& = u_j'' - \frac{a^2}{m^2} u_jv_j^2 + \frac{a^4}{m^4} \frac{D m \sqrt{m}}{a^2} - \frac{a^4}{m^4} \left(\frac{D m \sqrt{m}}{a^2}\right)^2 u_j + \frac{a^2}{m^2} \frac{D m \sqrt{m}}{a^2} h_x\left( P_j + \frac{D}{\sqrt{m}} \xi_j \right) u_j' \\
&\quad + \frac{a^4}{m^4} \left(\frac{D m \sqrt{m}}{a^2}\right)^2 h_{xx}\left(P_j + \frac{D}{\sqrt{m}} \xi_j \right) u_j \\
- \frac{a^2}{m^2} c_j(t) v_j'
& = v_j'' - v_j + u_jv_j^2,
	\end{cases}
\label{eq:inner-equation}
\end{equation}
where the prime denotes derivatives with respect to $\xi_j$ and the subscript $j$ is here to remind us that we are looking for a solution in the $j$-th inner region. To find solutions in the inner region, we use regular expansions for $u$ and $v$. The equations~\eqref{eq:inner-equation} suggest that the main small parameter is $\frac{a^2}{m^2}$ -- which is small by assumption (A1). Hence we look for solutions of the form
\begin{equation}
	\begin{cases}
		u_j & = u_{0j} + \frac{a^2}{m^2} u_{1j} + \ldots \\
		v_j & = v_{0j} + \frac{a^2}{m^2} v_{1j} + \ldots
	\end{cases}
\end{equation}
The leading order problem in the $j$-th inner region is then given by the following set of equations. This system is usually called the fast-reduced system in the context of geometric singular perturbation theory.
\begin{equation}
	\begin{cases}
		0 & = u_{0j}'', \\
		0 & = v_{0j}'' - v_{0j} + u_{0j} v_{0j}^2.
	\end{cases}
\end{equation}
Hence we find $u_{0j}$ to be constant and
\begin{equation}
v_{0j}(\xi) = \frac{3}{2} \frac{1}{u_{0j}} \sech^2(\xi/2). \label{eq:v0j}
\end{equation}
Thus, all $V$-pulses are at leading order given by the same $\sech$-function. However, their amplitudes vary, as these are determined by the values of $u_{0j}$, which are, so far, unknown. Later on, we will see that the values of $u_{0j}$ will be determined by (all) the pulse locations $P_1(t), \ldots, P_N(t)$. Note that the pulses thus influence each other (only) through this mechanism.
By assumptions (A1)-(A3) and (A5) we notice that the next order problem is given by
\begin{equation}
	\begin{cases}
		u_{1j}''  & = u_{0j} v_{0j}^2, \\
		v_{1j}'' - v_{1j} + 2 u_{0j} v_{0j} v_{1j} & = - c_j(t) v_{0j}' - v_{0j}^2 u_{1j},
	\end{cases}
	\label{eq:fast-equation-next-order}
\end{equation}
Unlike the $u$-equation, it is not clear a priori whether the $v$-equation is solvable. We define the self-adjoint operator $\mathcal{L}:= \partial_{\xi}^2  - 1 + 2 u_{0j} v_{0j}$. $\mathcal{L}$ has a non-empty kernel, since $\mathcal{L} v_{0j}' = 0$. Hence the inhomogeneous equation $\mathcal{L}v_{1j} = -c_j(t)v_{0j}' - v_{0j}^2 u_{1j}$ might not be solvable and we need to impose a Fredholm solvability condition
\begin{equation}
	\int_{I_j^{in}} c_j(t) v_{0j}'(\eta)^2 d\eta = \int_{I_j^{in}} - v_{0j}(\eta)^2 u_{1j}(\eta) v_{0j}'(\eta) d\eta.
\end{equation}
Applying integration by parts twice to the right-hand side yields
\begin{align*}
&\ \int_{I_j^{in}} u_{1j}(\eta) v_{0j}(\eta)^2 v_{0j}'(\eta) d\eta = \frac{1}{3} \int_{I_j^{in}} u_{1j}(\eta) \frac{d}{d\eta} [v_{0j}(\eta)^3] d\eta  = \frac{1}{3} \left[ v_{0j}(\eta)^3 u_{1j}(\eta) \right]_{\eta = -\frac{1}{\sqrt{\varepsilon}}}^{\eta = \frac{1}{\sqrt{\varepsilon}}}\ - \frac{1}{3} \int_{I_j^{in}} u_{1j}'(\eta) v_{0j}(\eta)^3 d\eta \\
= &\ - \frac{1}{3} \left[ u_{1j}'(\eta) \int_0^\eta v_{0j}(y)^3 dy \right]_{\eta = -\frac{1}{\sqrt{\varepsilon}}}^{\eta = \frac{1}{\sqrt{\varepsilon}}}\ + \frac{1}{3} \int_{I_j^{in}} u_{1j}''(\eta) \int_0^\eta v_{0j}(y)^3 dy\ d\eta +h.o.t.
\end{align*}
To get from the second to the third line, we have used that $v_{0j}$ gets exponentially small near the boundaries of $I_j^{in}$ and that $u_{1j}$ does not get exponentially large there. We note that $v_{0j}$ is an even function. Therefore $u_{1j}''$ is an even function and $\eta \mapsto \int_0^\eta v_{0j}(y)^3 dy$ is an odd function. So the last integral over the inner region vanishes. Finally, because $v_{0j}^3$ is even, we can reformulate the solvability condition and obtain
\begin{equation}
c_j(t) \int_{I_j^{in}} v_{0j}'(\eta)^2 d\eta = \frac{1}{6} \left[ u_{1j}'\left(\frac{1}{\sqrt{\varepsilon}}\right) + u_{1j}'\left(-\frac{1}{\sqrt{\varepsilon}}\right) \right] \int_{I_j^{in}} v_{0j}(\eta)^3 d\eta.
\end{equation}
The integrals over the inner region can be approximated by integrals over $\R$, because $v_{0j}$ is exponentially small outside $I_j^{in}$. As we know the function $v_{0j}$ explicitly, it is possible to evaluate the integrals in this Fredholm condition explicitly. This gives us an expression for the (scaled) speed of the $j$-th pulse as
\begin{equation}
	c_j(t) = \frac{1}{u_{0j}} \left[ u_{1j}'\left(\frac{1}{\sqrt{\varepsilon}}\right) + u_{1j}'\left(-\frac{1}{\sqrt{\varepsilon}}\right) \right].
\label{eq:speed-eq1}
\end{equation}
It follows from the $u$-equation in~\eqref{eq:fast-equation-next-order} that,
\begin{align}
	u_{1j}'\left(\frac{1}{\sqrt{\varepsilon}}\right) - u_{1j}'\left(-\frac{1}{\sqrt{\varepsilon}}\right)
& = \int_{I_j^{in}} u_{1j}''(\eta) d\eta = \int_{I_j^{in}} u_{0j} v_{0j}(\eta)^2 d\eta  = \int_{-\infty}^\infty u_{0j} v_{0j}(\eta)^2 d\eta + h.o.t. = \frac{6}{u_{0j}} + h.o.t.
\label{eq:jumpcond1}
\end{align}
Combining this with ~\eqref{eq:speed-eq1}, we conclude
\begin{align}
	c_j(t) &
= \frac{1}{6} \left[u_{1j}'\left(\frac{1}{\sqrt{\varepsilon}}\right) + u_{1j}'\left(-\frac{1}{\sqrt{\varepsilon}}\right) \right] \left[u_{1j}'\left(\frac{1}{\sqrt{\varepsilon}}\right) - u_{1j}'\left(-\frac{1}{\sqrt{\varepsilon}}\right) \right]  = \frac{1}{6} \left[u_{1j}'\left(\frac{1}{\sqrt{\varepsilon}}\right)^2 - u_{1j}'\left(-\frac{1}{\sqrt{\varepsilon}}\right)^2 \right].
\label{eq:speed-eq2}
\end{align}
The values $u_{1j}'\left(\pm 1/\sqrt{\varepsilon} \right)$ can be found by matching this inner solution to the outer solutions for $U$. Note that the speed of the $j$-th pulse does not seem to depend explicitly on the other pulses. However, the values of $u_{1j}'$ are not yet determined and we will find that these \emph{do} depend on the location of (all) other pulses.

\subsection{The outer regions}

In the outer regions, the $V$-component should be exponentially small, since $v_{0j}$ gets exponentially small near the boundaries of the inner regions. Since the $V$-equation is automatically solved by $V = 0$, we can set $V = 0$ in the outer regions to acquire a leading order approximation and we thus only need to deal with the $U$-equation. In each of the outer regions, equation~\eqref{eq:extKmodel1} reduces to the ODE
\begin{equation}
	0 = U_{xx} + h_x U_x + h_{xx} U + a - U.
\end{equation}
Since the pulses only travel asymptotically slow, the solutions of these equations are expected to be of order $\mathcal{O}(a)$ because of the forcing term. Therefore we rescale $U$ as $U = a \tilde{U}$, so that
\begin{equation}
	0 = \tilde{U}_{xx} + h_x \tilde{U}_x + h_{xx} \tilde{U} + 1 - \tilde{U}.
\end{equation}
Without explicitly solving these equations, we can already match the outer solutions to the inner solutions. For this we need to recall the scalings in equations~\eqref{eq:ex-inner-scaling-xi} and~\eqref{eq:ex-inner-scaling}. Careful bookkeeping then reveals that
\begin{align*}
	\tilde{U}(P_j) &= \frac{m \sqrt{m} D}{a^2} u_{0j} + h.o.t. \\
	\tilde{U}_x(P_j^\pm) &= \frac{m \sqrt{m} D}{a^2} \frac{\sqrt{m}}{D} \frac{a^2}{m^2} u_{1j}'\left(\pm\frac{1}{\sqrt{\varepsilon}}\right) + h.o.t. = u_{1j}'\left(\pm\frac{1}{\sqrt{\varepsilon}}\right) + h.o.t.
\end{align*}
where $P_j^+$ denotes taking the limit from above, and $P_j^-$ the limit from below.
Thus at this moment we have reduced the full PDE problem to a ODE problem with (undetermined) internal boundary conditions. We thus need to find a function $\tilde{U}$ and constants $u_{0j}$ that simultaneously satisfy the ODE
\begin{align}
	0 = \tilde{U}_{xx} + h_x \tilde{U}_x + h_{xx} \tilde{U} + 1 - \tilde{U}; \hspace{1cm} \tilde{U}(P_j) = \frac{m \sqrt{m} D}{a^2} u_{0j}. \label{eq:Ex-outer1}
\end{align}
\emph{and}, by~\eqref{eq:jumpcond1}, the jump conditions
\begin{equation}
	\tilde{U}_x(P_j^+) - \tilde{U}_x(P_j^-) = \frac{6}{u_{0j}},
\label{eq:jumpcond}
\end{equation}
Note that the ODE should also be accompanied by two boundary conditions, which -- of course -- depend on the type of domain we are interested in. Moreover, the expression~\eqref{eq:speed-eq2} for the speed $c_j(t)$ can be rewritten to
\begin{equation}
	c_j(t) = \frac{1}{6} \left[ \tilde{U}_x(P_j^+)^2 - \tilde{U}_x(P_j^-)^2\right].
	\label{eq:ODEc1}
\end{equation}
Thus the speed of the $j$-th pulse is determined by the (differences of the) squares of the derivative of $\tilde{U}$ at the pulse location. Since we are interested in this pulse movement, our next task is to actually solve the problem given by~\eqref{eq:Ex-outer1}-\eqref{eq:jumpcond}. We separate this problem into two different cases: (i) the case of assumption (A3) and (ii) the case of assumption (A3'), in particular when $\frac{m \sqrt{m} D}{a^2} = \mathcal{O}(1)$. The former case will be significantly simpler as the internal boundaries are approximately zero.

\subsection{Pulse location ODE under assumption (A3)} \label{sec:Pulse-Location-ODE-(A3)}

Under assumption (A3), the internal boundary conditions are approximated by $\tilde{U}(P_j) = 0$ so that $\tilde{U}$ is independent of $u_{0j}$ at leading order,
\begin{equation}
	0 = \tilde{U}_{xx} + h_x \tilde{U}_x + h_{xx} \tilde{U} + 1 - \tilde{U}; \hspace{1cm} \tilde{U}(P_j) = 0. \hspace{2cm} (j = 1, \ldots, N)
	\label{eq:outer-ODE-U-A3}
\end{equation}
This immensely reduces the complexity of the problem, as $\tilde{U}$ in the $k$-th outer region now only depends on the positions $P_{k-1}(t)$ and $P_k(t)$ -- and not on any of the others. It is therefore relatively easy to analytically approximate these expressions -- and the pulse location ODE -- \emph{if} we know the explicit solutions to the ODE. For general $h = h(x)$ it is, however,  in general not possible to find explicit solutions (in closed form) of this ODE. This does not obstruct the fact that also in this case the PDE can be reduced to a finite dimensional system of ODEs. However, to explicitly evaluate the ODE dynamics, we need to turn to numerical boundary value problem solvers. Note that although the value of $u_{0j}$ does not play a leading order role in the outer region expressions $\tilde{U}$, it does play a leading order role in the linear stability analysis -- therefore it is important to (also) still find a leading order expression of $u_{0j}$.

\subsubsection{Terrain with constant slope, i.e. $h(x) = H x$}

\begin{figure}
	\centering
	\includegraphics {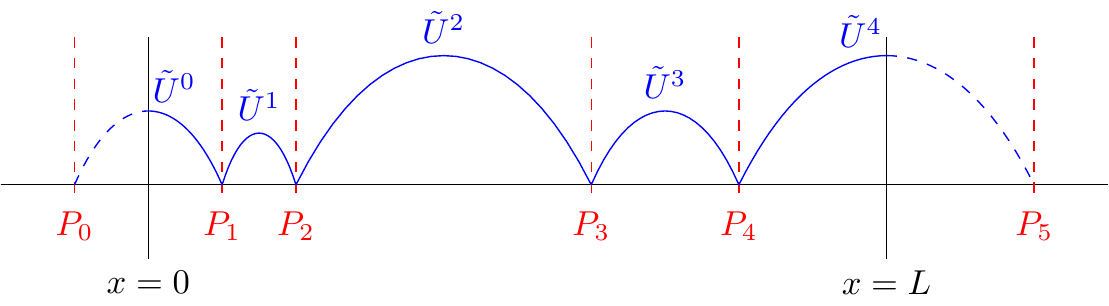}
	\caption{Sketch of the outer regions and the solutions $\tilde{U}^k$ in the corresponding $k$-th outer region.}
	\label{fig:sketchOuterRegions}
\end{figure}

When we consider a terrain with a constant slope, we do have access to explicit solutions for the outer region ODE~\eqref{eq:outer-ODE-U-A3}. Equation~\eqref{eq:outer-ODE-U-A3} then becomes
\begin{equation}
	0 = \tilde{U}_{xx} + H \tilde{U}_x + 1 - \tilde{U}; \hspace{1cm} \tilde{U}(P_j) = 0. \hspace{2cm} (j = 1, \ldots, N)
\end{equation}
The general solution is
\begin{align*}
	\tilde{U}(x) & = 1 + C_1 e^{D_1 x} + C_2 e^{D_2 x},
\intertext{where}
	D_{1,2} &:= \frac{1}{2} \left( - H \pm \sqrt{H^2 + 4} \right).
\end{align*}
We denote the solution in the $k$-th outer region by $\tilde{U}^k$ (see Figure~\ref{fig:sketchOuterRegions}). All, but the first and last, satisfy two internal boundary conditions $\tilde{U}^k(P_k) = 0$ and $\tilde{U}^k(P_{k+1}) = 0$, and are then given by
\begin{equation}
	\tilde{U}^k(x) = 1 + \frac{ \left(1 - e^{D_2 \Delta P_k}\right) e^{D_1(x-P_k)} + \left(e^{D_1 \Delta P_k}-1\right) e^{D_2(x-P_k)}}{e^{D_2 \Delta P_k} - e^{D_1 \Delta P_k}}, \hspace{1cm} (k = 1,\ldots, N-1)
\end{equation}
where $\Delta P_k := P_{k+1}-P_k$ is the distance between the two consecutive pulses.
To derive an expression for the pulse-location ODE, it is necessary to find $\tilde{U}^k_x(P_k)$ and $\tilde{U}^k_x(P_{k+1})$. Direct computation of these derivatives yields after some algebra:
\begin{align}
	\tilde{U}^k_x(P_k)
& = \frac{ D_1 \left(1 - e^{D_2 \Delta P_k}\right) + D_2 \left( e^{D_1 \Delta P_k} - 1\right) }{e^{D_2 \Delta P_k} - e^{D_1 \Delta P_k}} &
& = \frac{H}{2} - \frac{\sqrt{H^2+4}}{2} \frac{e^{H \Delta P_k / 2} - \cosh( \sqrt{H^2+4} \Delta P_k /2 )}{\sinh( \sqrt{H^2+4} \Delta P_k / 2)}, \nonumber \\
	\tilde{U}^k_x(P_{k+1})
& = \frac{D_1\left(1 - e^{D_2\Delta P_k}\right) e^{D_1 \Delta P_k} + D_2 \left( e^{D_1 \Delta P_k} - 1\right) e^{D_2 \Delta P_k}}{e^{D_2 \Delta P_k} - e^{D_1 \Delta P_k}} &
& = \frac{H}{2} + \frac{\sqrt{H^2+4}}{2} \frac{e^{-H \Delta P_k / 2} - \cosh( \sqrt{H^2+4} \Delta P_k / 2)}{\sinh(\sqrt{H^2+4} \Delta P_k / 2)}. \nonumber \\
&&& \hspace{3.5cm} (k = 1,\ldots,N-1) \label{eq:outerRegionDerivatives}
\end{align}
Substitution of these expression in equation~\eqref{eq:ODEc1} gives the movement of pulses on a terrain given by $h(x) = Hx$ as
\begin{multline}
\frac{dP_j}{dt} = \frac{D a^2}{m \sqrt{m}} \frac{1}{6}
\left[ \left( \frac{H}{2} - \frac{\sqrt{H^2+4}}{2} \frac{e^{H\Delta P_j/2} - \cosh\left( \sqrt{H^2+4}\Delta P_j/2\right)}{\sinh \left( \sqrt{H^2+4} \Delta P_j/2\right)} \right)^2 \right. \\
 - \left. \left( \frac{H}{2} + \frac{\sqrt{H^2+4}}{2} \frac{e^{-H\Delta P_{j-1}/2} - \cosh\left( \sqrt{H^2+4} \Delta P_{j-1}/2 \right)}{\sinh \left( \sqrt{H^2+4}\Delta P_{j-1}/2 \right)} \right)^2 \right]. \hspace{1cm} (j = 2, \ldots N-1)
\label{eq:ODEforgenH}
\end{multline}
For completely flat terrains we have a slope $H = 0$ so that the ODE reduces to
\begin{equation}
 \frac{dP_j}{dt} = \frac{D a^2}{m \sqrt{m}} \frac{1}{6} \left[ \tanh(\Delta P_j / 2)^2 - \tanh(\Delta P_{j-1}/2 )^2 \right], \hspace{1cm} (j = 2, \ldots, N-1)
\label{eq:ODEforHzero}
\end{equation}
which is in agreement with~\cite[Equation (2.28)]{chen2009oscillatory}. The values for $u_{0j}$ are obtained by combining the expressions in~\eqref{eq:outerRegionDerivatives} with equation~\eqref{eq:jumpcond}. We obtain
\begin{equation}
	\frac{6}{u_{0j}} = - \frac{\sqrt{H^2+4}}{2} \left[ \frac{e^{H \Delta P_j/2} - \cosh(\sqrt{H^2+4}\Delta P_j/2)}{\sinh(\sqrt{H^2+4}\Delta P_j/2)} + \frac{e^{-H \Delta P_{j-1}/2} - \cosh( \sqrt{H^2+4}\Delta P_{j-1}/2)}{\sinh(\sqrt{H^2+4}\Delta P_{j-1}/2)} \right].\hspace{1cm} (j = 2, \ldots, N-1)
\label{eq:u0jforgenH}
\end{equation}
For $H = 0$ this expression reduces to
\begin{equation}
	\frac{6}{u_{0j}} = \tanh(\Delta P_j /2) + \tanh(\Delta P_{j-1}/2). \hspace{4cm} (j = 2, \ldots, N-1).\label{eq:u0jforHzero}
\end{equation}
Note that in principle these expressions \eqref{eq:ODEforgenH}-\eqref{eq:u0jforHzero} do not hold for $j = 1$ and $j = N$ as these pulses do not have two neighbours. In fact, the solutions $\tilde{U}$ in the first and last outer region do not satisfy the same boundary conditions as the solution in the other regions. One should therefore recompute $\tilde{U}_x^1(P_1)$ and $\tilde{U}^{N+1}(P_N)$ for each type of domain. However, it is possible to introduce the two auxiliary locations $P_0$ and $P_{N+1}$ in such a way that expressions~\eqref{eq:ODEforgenH} and~\eqref{eq:ODEforHzero} still holds true for $j=1$ and $j = N$ (see Figure~\ref{fig:sketchOuterRegions}). Below we inspect several type of domains and explain this reasoning further

\paragraph{Unbounded domains}

On unbounded domains, we only have the requirement that solutions stay bounded as $|x| \rightarrow \infty$. So $\tilde{U}^1$ should satisfy this boundedness requirement, the ODE and the boundary condition $\tilde{U}(P_1) = 0$. From this it follows that $\tilde{U}_x^1(P_1) = \frac{H}{2} - \frac{\sqrt{H^2+4}}{2}$. Similarly, $\tilde{U}_x^{N+1}(P_N) = \frac{H}{2} + \frac{\sqrt{H^2+4}}{2}$. When we introduce $P_0 \rightarrow -\infty$ and $P_{N+1} \rightarrow \infty$ in equation~\eqref{eq:ODEforgenH} we see that the pulse location ODE is given by~\eqref{eq:ODEforgenH}, even for $j = 1$ and $j = N$.

\paragraph{Bounded domains with periodic boundary conditions} When we consider the bounded domain $[0,L]$ with periodic boundary conditions, we set $\tilde{U}(0) = \tilde{U}(L)$. That is, the first pulse has the last pulse as a neighbour. Therefore expression~\eqref{eq:ODEforgenH} is directly applicable when we set $\Delta P_0 = \Delta P_N = L - P_N + P_1$ or -- equivalently -- $P_0 := P_N - L$ and $P_{N+1} := L + P_1$.

\paragraph{Domains with Neumann boundary conditions} When the domain $[0,L]$ has Neumann boundary conditions, we impose the boundary conditions $\tilde{U}_x(0) = 0$ and $\tilde{U}_x(L) = 0$. A similar and straightforward computation then yields
\begin{align*}
	\tilde{U}_x^1(P_1) & = \frac{-2 \sinh(\sqrt{H^2+4} P_1/2)}{H \sinh(\sqrt{H^2+4} P_1/2) + \sqrt{H^2+4} \cosh(\sqrt{H^2+4} P_1/2)}, \\
	\tilde{U}_x^{N+1}(P_N) & =\frac{2 \sinh(\sqrt{H^2+4} P_N/2)}{H \sinh(\sqrt{H^2+4} P_N/2) + \sqrt{H^2+4} \cosh(\sqrt{H^2+4} P_N/2)}.
\end{align*}
The positions of the auxiliary locations $P_0 < 0$ respectively $P_{N+1} > L$ are determined as the negative zero of $\tilde{U}^0$ extended below $x = 0$ respectively the second zero of $\tilde{U}^N$ extended beyond $x = L$. However, for general $H$ there is no simple expression (in closed form) for $P_0$ and $P_{N+1}$, though we find that $\Delta P_0 = P_1 - P_0$ decreases as $P_0$ decreases and $\Delta P_N = P_{N+1}-P_N$ decreases as $L-P_N$ decreases (i.e. as $P_N$ increases). In the specific case $H = 0$ we do find explicit expressions: $P_0 = -P_1$ and $P_{N+1} = 2L - P_N$.

\subsubsection{Fixed points of the pulse-location ODE}\label{sec:ODEFixedPoints}

It is natural to study the fixed points of the pulse-location ODE~\eqref{eq:ODEforgenH}. Whether this ODE has any fixed points depends on the type of domain and boundary conditions. Below we summarise the results we acquired for bounded domains with Neumann boundary conditions, for bounded domains with periodic boundary conditions and for unbounded domains. The proofs of these statements rely on the fact that the derivatives $\Delta \tilde{U}(P_k^\pm)$ strictly increase/decrease as a function of the distance to the neighbouring pulse. The (mostly technical) details of the proofs can be found in appendix~\ref{sec:app-FixedPointProofs}.

Note that the results in this section only consider the behaviour of the pulse-location ODE~\eqref{eq:ODEforgenH} in itself and do \emph{not} take the behaviour of the full PDE into account. Specifically we do not take the stability of the $N$-pulse manifold $\mathcal{M}_N$ into account. It can happen that a fixed point of the ODE is stable under the flow of the ODE, but not under the flow of the complete PDE (as we will see in section~\ref{sec:NumericalSimulations}, e.g. Figure~\ref{fig:NUM-FP-unstable-N10}).

\paragraph{Bounded domains with Neumann boundary conditions}
For these domains the pulse location ODE~\eqref{eq:ODEforgenH} has precisely one fixed point. This fixed point is stable under the flow of the ODE. An example of this is given in Figure~\ref{fig:NUM-FP-unstable-N10}.

\paragraph{Bounded domain with periodic boundary conditions}
On these domains the ODE does not have any fixed points, unless $H = 0$, for which there is a continuous family of fixed points. All of these fixed points are regularly spaced configurations, i.e. $\Delta P_j = L/N$ for all $j$. This family of fixed points is stable under the flow of the ODE.

Moreover, on bounded domains with periodic boundary conditions, the pulse-location ODE~\eqref{eq:ODEforgenH} does have a continuous family of uniformly traveling solutions in which all pulses move with the same speed and the distance between two consecutive pulses is $\Delta P_j = L /N$ for all $j$, i.e. the pulses are regularly spaced. This family of solutions is stable under the flow of the ODE.

\paragraph{Unbounded domains}

In this situation the ODE~\eqref{eq:ODEforgenH} does not have any fixed points and there does not exist any uniformly traveling solution either, unless $N = 1$. In fact, the distance between the first and last pulse, $P_N - P_0$, is ever increasing.

\subsection{Pulse location ODE under assumption (A3')}\label{sec:pulseODE-SaddleNode}

When $\frac{m \sqrt{m} D}{a^2} = \mathcal{O}(1)$, equation~\eqref{eq:Ex-outer1} can no longer be simplified to~\eqref{eq:outer-ODE-U-A3}. Thus we do need to determine the values of $u_{0j}$ directly and we do need to make sure these lead to a solution $\tilde{U}$ that satisfies the jump conditions in equation~\eqref{eq:jumpcond}. More concretely, for a given $\vec{u}_0 := (u_{01},\ldots,u_{0N})^T$, a vector of the values of the internal boundary conditions, the boundary value problem~\eqref{eq:Ex-outer1} is well-posed and has a (uniquely determined) solution $\tilde{U}$ on all subdomains. With this $\tilde{U}$ we can validate the jump conditions~\eqref{eq:jumpcond}. The following quantity defines a way to measure how good the internal boundary conditions $\vec{u}_0$ satisfy the jump conditions
\begin{equation*}
	\vec{F}(\vec{u}_0) := \left( \tilde{U}_x(P_1^+;\vec{u}_0) - \tilde{U}_x(P_1^-;\vec{u}_0) - \frac{6}{u_{01}}, \ldots, \tilde{U}_x(P_N^+;\vec{u}_0) - \tilde{U}_x(P_N^-;\vec{u}_0) - \frac{6}{u_{0N}} \right)^T.
\end{equation*}
The correct internal boundary conditions $\vec{u}_0^*$ should satisfy $\vec{F}(\vec{u}_0^*) = \vec{0}$. If~\eqref{eq:Ex-outer1} has closed-form solutions, the function $\vec{F}(\vec{u}_0^*)$ can be constructed explicitly. Hoever, in general one needs a numerical root-finding scheme to solve $\vec{F}(\vec{u}_0^*) = \vec{0}$. We have used the standard Newton scheme for this. Note that $\vec{F}(\vec{u}_0) = \vec{0}$ does not necessarily have any solution and if it has, those solutions are -- in general -- not unique. Some cases for which we can find the roots explicitly are studied below. For notational convenience we define $\delta := \frac{m \sqrt{m} D}{a^2}$.

\subsubsection{Terrain with constant slope, i.e. $h(x) = Hx$}

The reasoning in section~\ref{sec:Pulse-Location-ODE-(A3)} leading to the pulse-location ODE~\eqref{eq:ODEforgenH} in the case of $\delta \ll 1$, can be repeated here. The only difference is the addition of non-zero internal boundary conditions. The derivatives $\tilde{U}^k_x(P_k)$ and $\tilde{U}^k_x(P_{k+1})$ can be computed in a similar way as before. This time -- when $\delta = \O(1)$ -- we find
{\small
\begin{align*}
	\tilde{U}^k_x(P_k) =
& \ \left(1 - \delta u_{0,k}\right) \frac{H}{2} - \frac{\sqrt{H^2+4}}{2} \frac{ \left(1 - \delta u_{0,k+1}\right)e^{H \Delta P_k/2} - \left(1 - \delta u_{0,k}\right) \cosh(\sqrt{H^2+4} \Delta P_k/2)}{\sinh(\sqrt{H^2+4}\Delta P_k/2)}, \\
	\tilde{U}^k_x(P_{k+1}) =
&\ \left(1 - \delta u_{0,k+1}\right) \frac{H}{2} + \frac{\sqrt{H^2+4}}{2} \frac{ \left(1 - \delta u_{0,k}\right) e^{-H \Delta P_k/2} - \left(1 - \delta u_{0,k+1}\right) \cosh(\sqrt{H^2+4}\Delta P_k/2)}{\sinh(\sqrt{H^2+4}\Delta P_k/2)}. \hspace{1cm} (k = 1, \ldots, N-1)
\end{align*}}
Substitution of these expressions in equation~\eqref{eq:ODEc1} gives the movement of the pulses as
\begin{multline}
\frac{dP_j}{dt} = \frac{D a^2}{m \sqrt{m}} \frac{1}{6}
\left[ \left( \kappa_k \frac{H}{2} - \frac{\sqrt{H^2+4}}{2} \frac{\kappa_{k+1} e^{H\Delta P_j/2} - \kappa_k \cosh\left( \sqrt{H^2+4}\Delta P_j/2\right)}{\sinh \left( \sqrt{H^2+4} \Delta P_j/2\right)} \right)^2 \right. \\
 - \left. \left( \kappa_k \frac{H}{2} + \frac{\sqrt{H^2+4}}{2} \frac{\kappa_{k-1} e^{-H\Delta P_{j-1}/2} - \kappa_k \cosh\left( \sqrt{H^2+4} \Delta P_{j-1}/2 \right)}{\sinh \left( \sqrt{H^2+4}\Delta P_{j-1}/2 \right)} \right)^2 \right],
\label{eq:ODEforgenH-deltaO1}
\end{multline}
where $\kappa_j := 1 - \delta u_{0j}$. However, the $u_{0j}$-values are still unknown at this moment. To obtain these we need to solve $\vec{F}(\vec{u}_0) = \vec{0}$. With the explicit expressions for the derivatives $\tilde{U}^k_x(P_k)$ and $\tilde{U}^k_x(P_{k+1})$ at hand we can express the components of this function explicitly
\begin{equation}
	F^k(\vec{u}_0)
= - \frac{\sqrt{H^2+4}}{2} \left[ \frac{ \kappa_{k+1} e^{H\Delta P_k/2} - \kappa_k \cosh(\sqrt{H^2+4}\Delta P_k/2)}{\sinh(\sqrt{H^2+4}\Delta P_k/2)} \right. \left. + \frac{ \kappa_{k-1} e^{-H \Delta P_{k-1}/2} - \kappa_k \cosh(\sqrt{H^2+4}\Delta P_{k-1}/2)}{\sinh(\sqrt{H^2+4}\Delta P_{k-1}/2)} \right] - \frac{6}{u_{0k}}.
\label{eq:Fk-jumpcondition}
\end{equation}
As before equations~\eqref{eq:ODEforgenH-deltaO1} and~\eqref{eq:Fk-jumpcondition} do not hold true for $j=1$ and $j = N$ because these do not have two neighbour pulses. Again it is possible to derive expressions for $\tilde{U}_x^1(P_1)$ and $\tilde{U}_x^{N+1}(P_N)$ as we did in section~\ref{sec:Pulse-Location-ODE-(A3)} when $\delta \ll 1$. As the procedure is so similar, we refrain from doing that here. In general, one cannot expect to be able to determine the roots of~\eqref{eq:Fk-jumpcondition} explicitly. Therefore we only consider the upcoming one-pulse example explicitly. We refrain from studying the pulse-location ODE analytically and use a numerical root-solving algorithm in section~\ref{sec:NumericalSimulations}.

\subsubsection{A one-pulse on $\R$}

The simplest, explicitly solvable, case is a $1$-pulse on $\R$. The solution of ODE~\eqref{eq:Ex-outer1} is for all $u_{01}$ and given by
\begin{align*}
	\tilde{U}^0(x) & = 1 + \left( \delta u_{01} - 1 \right) e^{D_1(x-P_1)}, &
	\tilde{U}^1(x) & = 1 + \left( \delta u_{01} - 1 \right) e^{D_2(x-P_1)}.
\end{align*}
Thus the function $F$ is given by
\begin{equation*}
	F(u_{01}) = \sqrt{H^2+4} \left(1 - \delta u_{01} \right) - \frac{6}{u_{01}}.
\end{equation*}
so that $F(u_{01}) = 0$ is solved by
\begin{equation}
	\left(u_{01}\right)_\pm = \frac{1}{2} \frac{ 1 \pm \sqrt{1 - 24 \delta / \sqrt{H^2+4}}}{\delta}.
	\label{eq:EX-u0-homoclinic}
\end{equation}
This expression agrees with the expressions found in the literature~\cite{dek1siam,Lottes}. It is also clear from this expression that there are two solutions as long as $\delta < \delta_c := \frac{1}{24} \sqrt{H^2+4}$. So for $H = 0$ we find $\delta_c = \frac{1}{12}$, again in correspondence with the literature~\cite[page 8]{dek1siam}. When $\delta = \delta_c$ a saddle-node bifurcation occurs where the two solutions coincide and for $\delta > \delta_c$ solutions no longer exist. The pulse-location ODE for this situation is given by
\begin{equation}
	\frac{dP_1}{dt} = \frac{D a^2}{m \sqrt{m}} H \sqrt{H^2+4}\ \left( 1 - \delta u_{01} \right)^2.
	\label{eq:EX-speed-homoclinic}
\end{equation}
In the asymptotic limit $\delta \ll 1$, equation~\eqref{eq:EX-u0-homoclinic} yields two solutions, given to leading order by
\begin{align*}
	(u_{01})_+ & = \frac{1}{\delta} + \mathcal{O}(1), 	&
	(u_{01})_- & = \frac{6}{\sqrt{H^2+4}} + \mathcal{O}(\delta).
\end{align*}
In section~\ref{sec:Pulse-Location-ODE-(A3)}, in equation~\eqref{eq:u0jforgenH} we found only one value for $u_{0j}$. Carefully taking the limit $\Delta P_0 \rightarrow - \infty$ and $\Delta P_1 \rightarrow \infty$ of~\eqref{eq:u0jforgenH} reveals that only $(u_{01})_-$ is found. This is because $(u_{01})_+ \gg 1$ in this asymptotic limit and it therefore does not satisfy the (implicit) assumption that $u_{01} = \mathcal{O}(1)$. This focus on $(u_{01})_-$ is justified; if one were to study the other possibility, i.e. pulses that have the internal boundary condition $(u_{01})_+$, one would quickly find out that these pulses are always unstable~\cite{doelman1998stability}.

\section{Linear Stability}\label{sec:linstabsections}

In this section, we look at perturbations of $N$-pulse solutions and study the associated quasi-steady spectrum.  For this we freeze the $N$-pulse solution and (at leading order) its time-dependent movement on the manifold $\mathcal{M}_N$. We then linearise around this $N$-pulse configuration to obtain a quasi-steady eigenvalue problem, which can be solved along the very same lines as the existence problem. This gives us quasi-steady eigenvalues and eigenfunctions. We can compute these for any given time $t$ and as such these quasi-steady eigenvalues and eigenfunctions are parametrised by time $t$ (via the pulse locations $P_j(t)$) -- see also~\cite{DKP, HeijsterFrontInt, bellsky2013}). Although our approach in principle works in a general setting -- thus for instance with a general topography $h(x)$ -- both its interpretation and its presentation are significantly facilitated when we restrict ourselves to pulse-solutions of the extended Klausmeier model~\eqref{eq:extKmodel1} for terrains with constant slope, i.e. $h(x) = Hx$. For other kind of terrains the PDE has space-dependent coefficients and explicit expressions are not present in general. Here other techniques need to be used~\cite{varyingTerrainArticle}.

We start with the classical case of a single pulse (i.e. $N = 1$) on $\mathbb{R}$ in section~\ref{sec:stab-1-pulse}. This illustrates the concepts and shows how it generalises to other boundaries or multiple pulses, which we will study subsequently in section~\ref{sec:stab-N-pulses}. In both sections we find essential differences between the asymptotic cases $m \gg 1 + H^2/4$ and $m \ll 1 + H^2/4$. In the former case ($m \gg + H^2/4$) we find Hopf bifurcations. Moreover, we find that pulses in the stability problem are far apart such that the eigenfunctions decouple and can be studied per pulse. In the latter case ($m \ll 1 + H^2/4$) we find saddle-node bifurcations. However, in this situation the eigenfunctions are coupled, which leads to a more involved eigenvalue problem and more involved eigenfunctions~\cite{chen2009oscillatory}.

The first step in the stability analysis consists of linearizing the extended Klausmeier model around a (frozen) $N$-pulse solution. We denote the $N$-pulse configuration of this equation by $(U_p^N, V_p^N)$ and set $(U,V) = (U_p^N,V_p^N) + e^{\lambda t} (\bar{U},\bar{V})$ to study its linear stability. Following the scalings in~\cite{chen2009oscillatory, dek2siam} we scale the eigenvalue as $\lambda = m \hat{\lambda}$ to study the so-called large eigenvalues that correspond to perturbations non-tangent to the manifold $\mathcal{M}_N$ of $N$-pulse solutions. Thus we obtain the quasi-steady eigenvalue problem
\begin{equation}
	\begin{cases}
		0 & = \bar{U}_{xx} + H \bar{U}_x - (1 + m \hat{\lambda} + (V_p^N)^2) \bar{U} - 2 U_p^N V_p^N \bar{V} \\
		0 & = D^2 \bar{V}_{xx} + (2 U_p^N V_p^N - m - m \hat{\lambda} ) \bar{V} + (V_p^N)^2 \bar{U}.
	\end{cases}
\label{eq:stab-eigenvalueProblem}
\end{equation}
Our aim is to find the values $\hat{\lambda}$ for which we can solve this eigenvalue problem. To find these eigenvalues $\hat{\lambda}$ we can exploit the inner and outer regions of our previously obtained $N$-pulse solution. Because $V_p^N$ is localised near the pulse locations, we see that in the outer regions this problem reduces in leading order to
\begin{equation}
	\begin{cases}
		0 & = \bar{U}_{xx} + H U_{x} - (1 + m \hat{\lambda}) \bar{U} \\
		0 & = - (m+m \hat{\lambda}) \bar{V}.
	\end{cases}
\label{eq:stability-ODE-outer}
\end{equation}
Hence $\bar{V} = 0$ in the outer regions; $\bar{V}$ is also concentrated around the pulse locations in the stability problem.

Our approach now essentially boils down to the following. We first solve the $\bar{U}$-equation in the outer regions for general $\hat{\lambda}$. We then need to glue these solution together at the pulse locations. For this we require continuity of $\bar{U}$ and we additionally obtain a $\hat{\lambda}$-dependent jump condition for $\bar{U}_x$ at each pulse location, which is imposed by the solution in the inner regions. The correct eigenvalues $\hat{\lambda}$ are then those values that allow solutions $\bar{U}$ which satisfy the boundary conditions at both ends of the domain. This method thus also immediately gives us the form of the eigenfunction as well.

\subsection{Stability of homoclinic pulses on $\mathbb{R}$}\label{sec:stab-1-pulse}

We first consider the case of a homoclinic pulse on $\mathbb{R}$ that is located at $x = P_1$. In this setting we have one inner region, $I_1^{in}$, and two outer regions, $I_{1,2}^{out}$. Since we are working on $\mathbb{R}$ we do not have boundary conditions, but only require solutions in the outer regions to be bounded. Solving the homogeneous ODE~\eqref{eq:stability-ODE-outer} in the outer regions gives the solutions $\bar{U}_1$ in the first outer field and $\bar{U}_2$ in the second outer region as
\begin{equation*}
	\bar{U}_1(x) = C_1 e^{\frac{1}{2} [-H + \sqrt{H^2 + 4(1+m\hat{\lambda})}](x-P_1)}, \hspace{2cm}
	\bar{U}_2(x) = C_2 e^{\frac{1}{2} [-H - \sqrt{H^2 + 4(1+m\hat{\lambda})}](x-P_1)},
\end{equation*}
where $C_1$ and $C_2$ are some constants. To satisfy the continuity condition on $\bar{U}$, we set $C_1 = C_2 = \rho_1$. We then only need to impose a jump condition on the derivative $\tilde{U}_x$ at the pulse location. With respect to the outer regions the jump in $\bar{U}_x$ is given by
\begin{equation}
	\Delta_{out} \bar{U}_x(P_1) := \bar{U}_x(P_1^+) - \bar{U}_x(P_1^-) = - \rho_1 \sqrt{H^2+4(1+m\hat{\lambda})}.
\label{eq:stab-homoclinic-outerjump}
\end{equation}
This must be the same as the total change in $\bar{U}_x$ generated by the dynamics in the inner region. In the inner domain the system is given by
\begin{equation}
	\begin{cases}
		0 & = \frac{m}{D^2} \bar{U}'' + \frac{\sqrt{m}}{D} \bar{U}' - (1+m\hat{\lambda})\bar{U} - 2 U_p^N V_p^N \bar{V} - (V_p^N)^2 \bar{U}; \\
		0 & = m \bar{V}'' - (m+ m\hat{\lambda}) \bar{V} + 2 U_p^N V_p^N \bar{v} + (V_p^N)^2 \bar{U},
	\end{cases}
\end{equation}
where primes again denote derivatives with respect to the stretched coordinate $\xi_1$. From equation~\eqref{eq:v0j} and the scalings of~\eqref{eq:ex-inner-scaling} we know the approximate form of $V_p^1$ and $U_p^1$ in the inner region. For notational convenience we write $\omega(\xi):= \frac{3}{2} \sech(\xi/2)^2$. Moreover we note that $\bar{U} \approx \rho_1$ in the inner region by matching with the solutions in the outer region. Therefore in the inner region the stability problem reduces to
\begin{equation}
	\begin{cases}
		\bar{U}'' + \frac{D H}{\sqrt{m}} \bar{U}' - \frac{D^2}{m} (1+m\hat{\lambda}) \rho_1 - 2 D^2 \omega \bar{V} - \frac{a^2}{m^2} \frac{\rho_1}{u_{0j}^2} \omega^2 & = 0;\\
		\bar{V}'' - (1 + \hat{\lambda}) \bar{V} + 2 \omega \bar{V} & = - \frac{a^2}{m^2} \frac{1}{D^2} \frac{\rho_1}{ u_{01}^2} \omega^2.
	\end{cases}
\end{equation}
The $\bar{V}$-equation indicates that we need to scale $\bar{V}$ as
\begin{equation}
\bar{V} = - \frac{a^2}{m^2} \frac{1}{D^2} \frac{\rho_1}{u_{01}^2} V_{in},
\end{equation}
where $V_{in}$ thus satisfies
\begin{equation}
\left(\mathcal{L}_f(\zeta) - \hat{\lambda} \right) V_{in} := V_{in}'' - (1 + \hat{\lambda}) V_{in} + 2 \omega V_{in} = \omega^2. \label{eq:VINdifeq}
\end{equation}
We write the $\bar{U}$-equation as
\begin{align*}
\bar{U}'' + \frac{D H}{\sqrt{m}} \bar{U}' - \frac{D^2}{m} (1+m\hat{\lambda}) \bar{U} - \frac{a^2}{m^2} \left( 2 \omega V_{in} - \omega^2 \right) & = \\
\bar{U}'' + \frac{a^2}{m^2} \frac{D m \sqrt{m} H}{a^2} \bar{U}' - \frac{a^4}{m^4} \left( \frac{D m \sqrt{m}}{a^2} \right)^2 (1+m\hat{\lambda}) \bar{U} - \frac{a^2}{m^2} \left( 2 \omega V_{in} - \omega^2 \right) & = 0 .
\end{align*}
Because of assumptions (A1), (A3) and (A5) we find the leading order change of $\bar{U}'$ in the inner region to be
\begin{equation}
\Delta_{in} \bar{U}'_1:= \int_{I_{1}^{in}} \bar{U}''(\zeta) d\zeta = \frac{a^2}{m^2} \frac{\rho_1}{u_{01}^2} \int_{-\infty}^\infty (\omega^2 - 2 \omega V_{in}) d\zeta + h.o.t.
\end{equation}
For notational simplicity we write
\begin{equation}
\hat{C}(\hat{\lambda}) := \int_{-\infty}^\infty (\omega(\zeta)^2 - 2\omega(\zeta) V_{in}(\zeta;\hat{\lambda}) ) d\zeta.
\label{eq:stab-definition-Clambda}
\end{equation}
Because $\bar{U}_x = \frac{\sqrt{m}}{D} \bar{U}'$, we find the total jump in $\bar{U}_x$ over $I_{1}^{in}$,
\begin{equation}
\Delta_{in} \bar{U}_x(P_1) = \frac{a^2}{m\sqrt{m}D} \frac{\rho_1}{u_{01}^2} \hat{C}(\hat{\lambda}). \label{eq:weak-jump-1}
\end{equation}
Combining the outer and inner approximations of $\Delta \bar{U}_x$ in equations~\eqref{eq:stab-homoclinic-outerjump} and~\eqref{eq:weak-jump-1} yields
\begin{equation}
\frac{a^2}{m \sqrt{m} D} \frac{\rho_1}{u_{01}^2} \hat{C}(\hat{\lambda}) = -\sqrt{H^2 + 4 (1+m\hat{\lambda})}\rho_1.
\end{equation}
Since $\rho_1 = 0$ corresponds to the trivial solution of the eigenvalue problem~\eqref{eq:stab-eigenvalueProblem}, we take $\rho_1 \neq 0$ and find
\begin{equation}
\frac{a^2}{m\sqrt{m}D} \frac{1}{u_{01}^2} \hat{C}(\hat{\lambda}) = - \sqrt{H^ 2 + 4(1+m\hat{\lambda})}. \label{eq:LinStabWeak1}
\end{equation}
Now, $\hat{\lambda}$ is an eigenvalue when this expression holds true. This procedure can be followed for any given, fixed value of the parameters $a$ and $u_{01}$. Both $a$ and $u_{01}$ may vary (slowly) as a function of time, while we treat the time as an additional parameter. Therefore it is more insightful to rewrite~\eqref{eq:LinStabWeak1} into the equivalent, but more convenient form:
\begin{equation}
	m^2 D \frac{u_{01}^2}{a^2} = \frac{ \int_{-\infty}^\infty \omega V_{in} d\zeta - 3}{\sqrt{\hat{\lambda} + \frac{H^2 + 4}{4m}}}
\label{eq:weaklycondition}
\end{equation}
where we have used~\eqref{eq:stab-definition-Clambda} and $\int_{-\infty}^\infty \omega(\zeta)^2 d\zeta = 6$. We can only get a detailed understanding of the eigenvalues $\hat{\lambda}$ of this problem, once we understand the form of the right-hand side of~\eqref{eq:weaklycondition}, which boils down to studying the integral
\begin{equation}
	\mathcal{R}(\hat{\lambda}) := \int_{-\infty}^\infty \omega(\zeta) V_{in}(\zeta;\hat{\lambda}) d\zeta,
	\label{eq:definitionR}
\end{equation}
where $V_{in}(\zeta;\hat{\lambda})$ is a bounded function that solves~\eqref{eq:VINdifeq}.

\subsubsection{Properties of the integral $\mathcal{R}(\hat{\lambda})$}\label{sec:RlambdaProperties}

\begin{figure}[t]
	\centering
		\begin{subfigure}[t]{0.25\textwidth}
			\centering
			\includegraphics[width=\textwidth]{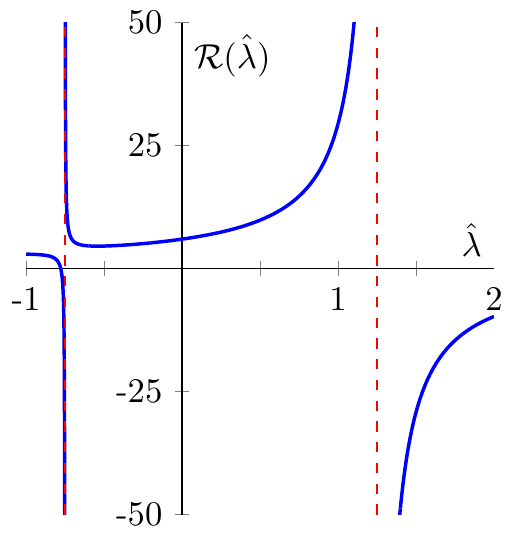}
			\caption{}
		\end{subfigure}
~
		\begin{subfigure}[t]{0.25\textwidth}
			\centering
			\includegraphics[width=\textwidth]{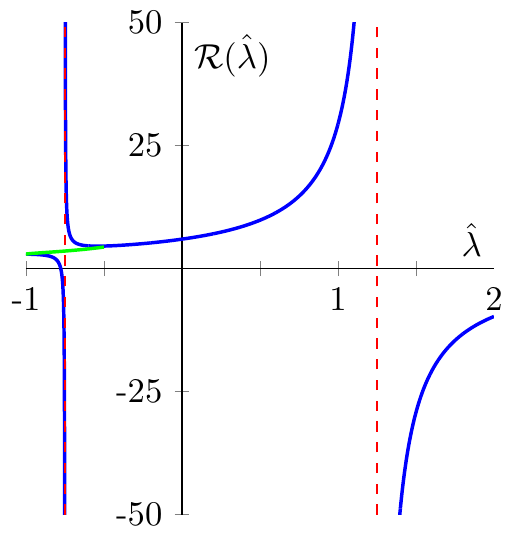}
			\caption{}
		\end{subfigure}
	\caption{The function $\mathcal{R}(\hat{\lambda})$. In (a) we show the form of $\mathcal{\R}(\hat{\lambda})$ only for real-valued $\hat{\lambda}$, whereas in (b) we also show the complex values of $\hat{\lambda}$ that lead to $\mathcal{R}(\hat{\lambda})$ that do not have an imaginary part (shown in green). In both figures the poles at $\hat{\lambda} = -3/4$ and $\hat{\lambda} = 5/4$ are indicated with dashed red lines.}
	\label{fig:RlambdaPlots}
\end{figure}

To get a detailed understanding of $R(\hat{\lambda})$, we need to solve~\eqref{eq:VINdifeq}. It is possible to transform this differential equation to a hypergeometric differential equation. The details of this procedure can be found in~\cite[section 5]{doelman1998stability} and~\cite[section 5.2]{Split-Evans} -- see Figure~\ref{fig:RlambdaPlots} for evaluations of $R(\hat{\lambda})$ based on this procedure. For several specific values of $\hat{\lambda}$ it is possible to get a direct grip on $\mathcal{R}(\hat{\lambda})$. Foremost, $V_{in}$ is only uniquely defined for $\hat{\lambda}$ that are not eigenvalues of the operator $\mathcal{L}_f$. When $\hat{\lambda}$ is an eigenvalue of $\mathcal{L}_f$, the solution $V_{in}$ is either not defined or not uniquely defined. When $V_{in}(\xi;\hat{\lambda})$ does not exist, the function $\mathcal{R}(\hat{\lambda})$ has a pole for this value of $\hat{\lambda}$. When $V_{in}(\xi;\hat{\lambda})$ is not uniquely defined for an eigenvalue $\hat{\lambda}$ of $\mathcal{L}_f$, the value of $\mathcal{R}(\hat{\lambda})$ is still uniquely defined~\cite{Split-Evans}.

The operator $\mathcal{L}_f$ is well-studied. The eigenvalues are known to be $\hat{\lambda} = \frac{5}{4}$, $\hat{\lambda} = 0$ and $\hat{\lambda} = - \frac{3}{4}$ and the essential spectrum is $(-\infty,-1)$ \cite{doelman2002homoclinic}. It turns out that $\mathcal{R}(\hat{\lambda})$ has poles for $\hat{\lambda} = \frac{5}{4}$ and for $\hat{\lambda} = - \frac{3}{4}$. For $\hat{\lambda} = 0$ one can verify that $V_{in}$ is given by $V_{in}(\xi;0) = C \omega'(\xi) + \omega(\xi)$, where $C$ is a constant. Direct substitution in~\eqref{eq:definitionR} shows that this constant $C$ drops out and we obtain
\begin{equation}
	\mathcal{R}(0) = 6. \label{eq:Rlambda0}
\end{equation}
The derivative of $\mathcal{R}$ at $\hat{\lambda} = 0$ can also be determined. For this, we first observe that $\mathcal{R}'(\hat{\lambda}) = \int_{-\infty}^\infty \omega(\xi) \partial_{\hat{\lambda}} V_{in}(\xi;\hat{\lambda}) d\xi$, where $\partial_{\hat{\lambda}} V_{in}(\xi;\hat{\lambda})$ satisfies
\begin{equation}
	\mathcal{L}_f \left( \partial_{\hat{\lambda}} V_{in}(\xi;\hat{\lambda}) \right) = V_{in}(\xi;\hat{\lambda}).
\end{equation}
Thus we must solve $\mathcal{L}_f \left( \partial_{\hat{\lambda}} V_{in}(x;0) \right) = \omega$. This yields $\partial_{\hat{\lambda}} V_{in}(\xi;0) = C\omega'(\xi) + \omega(\xi) + \frac{1}{2} \xi \omega'(\xi)$ and hence
\begin{equation}
	\mathcal{R}'(0) = \frac{9}{2}. \label{eq:Rlambdader0}
\end{equation}
Finally, at $\hat{\lambda} = -1$, the boundary of the essential spectrum, the differential equation for $V_{in}(\xi;-1)$ has a family of bounded solutions given as
\begin{equation}
	V_{in}(\xi;-1) = \omega(\xi) - \frac{1}{2} + C \left( 2 \tanh(\xi/2) + \frac{10}{3} \omega'(\xi) \right)
\end{equation}
and thus
\begin{equation}
	\mathcal{R}(-1) = 3.
\end{equation}
The above properties are the most important properties of $\mathcal{R}$ for the analysis in this article. A more extensive study of the properties of $\mathcal{R}$ is presented in~\cite[section 4.1]{doelman2002homoclinic} and~\cite[section 5]{Split-Evans}\footnote{Be aware though, that the $\mathcal{R}$ in this article has a different factor in front of it and is defined in terms of $\hat{\lambda}$, whereas the cited articles define it as function of $P := 2 \sqrt{1 + \hat{\lambda}}$.}.

\subsubsection{Finding eigenvalues}

There is also a square root in the right-hand side of~\eqref{eq:weaklycondition}. Thus, real solutions are only possible when $\hat{\lambda} > \hat{\lambda}^H := \frac{H^2 + 4}{4m}$. Moreover, this term can create an additional pole at $\hat{\lambda} = \hat{\lambda}^H$. Depending on the value of $\hat{\lambda}^H$ one of three things can happen.
\begin{itemize}
	\item $\hat{\lambda}^H \leq -1$: The new pole falls in the essential spectrum and the whole form of $\mathcal{R}(\hat{\lambda})$ is visible.
	\item $-1 < \hat{\lambda}^H < -\frac{3}{4}$: The new pole is seen, in addition to the two poles of $\mathcal{R}(\hat{\lambda})$.
	\item $-\frac{3}{4} \leq \hat{\lambda}^H$: The new pole at $\hat{\lambda} = \hat{\lambda}^H$ `replaces' the pole of $\mathcal{R}(\hat{\lambda})$ that is located at $\hat{\lambda} = - \frac{3}{4}$.
\end{itemize}
All three cases lead to different forms for the right-hand side of~\eqref{eq:weaklycondition} -- see Figure~\ref{fig:NLEPconditions}.

Now that we understand the right-hand side, we can determine the eigenvalues for our problem with a simple procedure. For this we compute the (current) value of the left-hand side of~\eqref{eq:weaklycondition} and then we see which values of $\hat{\lambda}$ lead to the same value on the right-hand side. Note that the value for $u_{01}$ is thus crucial in our stability problem. In section~\ref{sec:Pulse-Location-ODE-(A3)} and section~\ref{sec:pulseODE-SaddleNode} we determined $u_{01}$ and thus how it changes in time. When we let the rainfall parameter $a$ decrease over time, we typically see that $\frac{u_{01}}{a}$ increases. From this observation it is natural to study what happens to the eigenvalues when the left-hand side of~\eqref{eq:weaklycondition} increases.

\begin{figure}[t]
		\centering
		\begin{subfigure}[t]{0.25\textwidth}
			\centering

				\includegraphics[width=\textwidth]{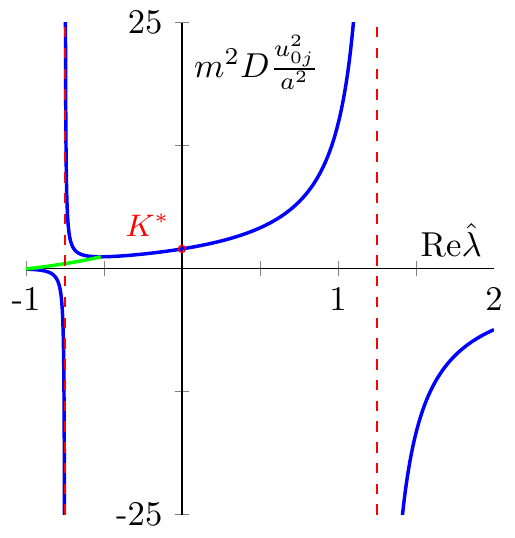}

			\caption{$H=0$, $m=0.45$}\label{fig:NLEPconditions-SN}
		\end{subfigure}
~
		\begin{subfigure}[t]{0.25\textwidth}
			\centering

				\includegraphics[width=\textwidth]{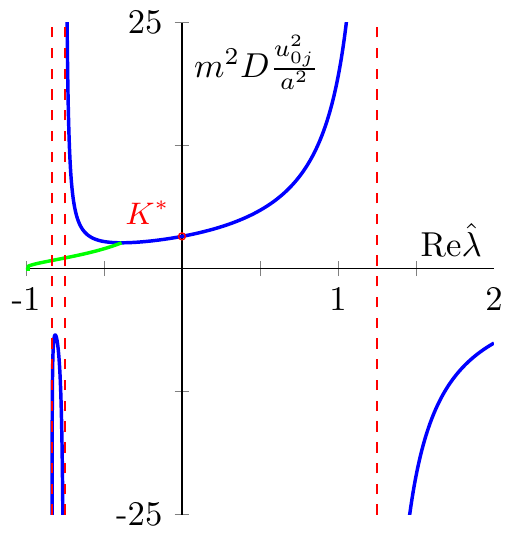}

			\caption{$H=0$, $m = 1.2$}
		\end{subfigure}
~
		\begin{subfigure}[t]{0.25\textwidth}
			\centering

				\includegraphics[width=\textwidth]{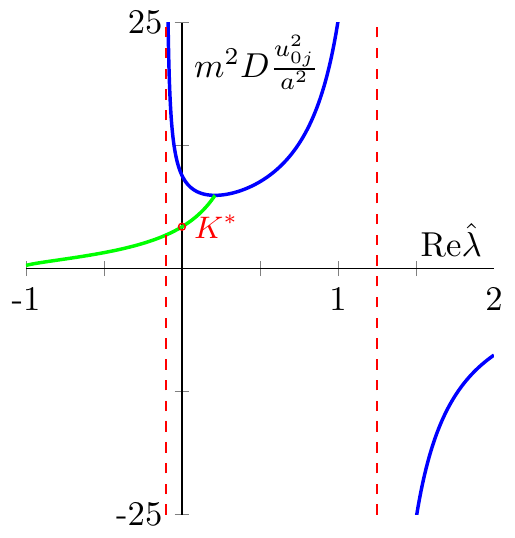}

			\caption{$H=0$, $m = 10$}\label{fig:NLEPconditions-hopf}
		\end{subfigure}	

	\caption{The right-hand side of~\eqref{eq:weaklycondition} for different possible values of $\hat{\lambda}^H$: in (a) $\hat{\lambda}^H = -20 / 9 < -1$, in (b) $\hat{\lambda}^H = - 5 / 6 \in (-1, -3/4)$ and in (c) $\hat{\lambda}^H = - 1 / 10 > - 3/4$.The location of the poles are indicated with dashed red lines. Additionally we show the complex $\hat{\lambda}$ that lead to real values in green and the value of $K^*(m,H)$, see~\eqref{eq:definitionKstar}.}
	\label{fig:NLEPconditions}
\end{figure}

The left-hand side of~\eqref{eq:weaklycondition} is always real-valued and positive. Therefore the right-hand side needs to be as well. Thus for given $H$ and $m$ only a specific set of $\hat{\lambda}$ are possible eigenvalues of the eigenvalue problem~\eqref{eq:weaklycondition} -- precisely those $\hat{\lambda}$ that lead to a real-valued and positive right-hand side in~\eqref{eq:weaklycondition}. This leads to a skeleton in $\mathbb{C}$ on which all eigenvalues necessarily lie. These skeletons come in three qualitatively different forms, which we show in Figure~\ref{fig:skeletons}. The difference between those skeletons is the place where the complex eigenvalues land on the real axis. For a critical value $m_c(H)$ they land precisely on $\hat{\lambda} = 0$. For $m > m_c(H)$ they land to the right of the imaginary axis and for $m < m_c(H)$ they land to the left of it.

The point where the complex eigenvalues land on the real axis, needs to be a local minimum\footnote{Otherwise there is a range of left-hand side values that have four eigenvalues, which is impossible as indicated by a winding number argument.} of the right-hand side of~\eqref{eq:weaklycondition}. Therefore the critical value $m_c(H)$ must be such that this minimum is attained at $\hat{\lambda} = 0$. Differentiating~\eqref{eq:weaklycondition} and setting the result to zero then indicates that $m_c(H)$ must satisfy
\begin{equation*}
	2 \mathcal{R}'(0) \frac{H^2 + 4}{4 m_c(H)}  - \mathcal{R}(0) + 3 = 0.
\end{equation*}
Substitution of~\eqref{eq:Rlambda0} and~\eqref{eq:Rlambdader0} then yields the critical value
\begin{equation}
	m_c(H) = 3 \left( 1 + \frac{H^2}{4} \right).
	\label{eq:stab-mcritical}
\end{equation}
The eigenvalues of~\eqref{eq:weaklycondition} can now simple be read of, and depend on the value of the left-hand side. For small values the eigenvalues approach the points $A_{1,2}$ in Figure~\ref{fig:skeletons}. When the left-hand side is increased, we follow the skeletons and see that the pair of complex eigenvalues changes into two real eigenvalues, points $B_{1,2}$ in Figure~\ref{fig:skeletons}. Increasing the value even further we end up close to the poles, points $C_{1,2}$ in Figure~\ref{fig:skeletons}.

Somewhere along this trajectory a bifurcation has occurred, when an eigenvalue $\hat{\lambda}$ gets a positive real part. For $m < m_c(H)$ this happens for one eigenvalue that has no imaginary part. Thus here we find a saddle-node bifurcation; the corresponding eigenfunction is shown in Figure~\ref{fig:stab-EF-SN}. For $m > m_C(H)$ a pair of complex eigenvalues enters the right-half plane and we thus have a Hopf bifurcation; the corresponding eigenfunction is shown in Figure~\ref{fig:stab-EF-HF}. Finally for $m = m_C(H)$ a codimension 2 Bogdanov-Takens bifurcation occurs. In all of these situations we find that there is a critical value $K^*(m,H)$ of the right-hand side. For values below $K^*(m,H)$ the pulse is stable and for values above it the pulse is unstable. This critical value is given by
\begin{equation}
	K^*(m,H) := \frac{ \mathcal{R}(\hat{\lambda}^*) - 3}{\sqrt{\hat{\lambda}^* + \frac{H^2+4}{4m}}}.
	\label{eq:definitionKstar}
\end{equation}
For $m \leq m_C(H)$ the critical eigenvalue is $\hat{\lambda}^* = 0$ and therefore $K^*(m,H) = 6 \sqrt{\frac{m}{H^2+4}}$. Also note that $K^*(m_c(H),H) = 3 \sqrt{3}$. For $m > m_C(H)$ there is no explicit expression, but for given parameters $m$ and $H$ it is not hard to obtain it by numerical evaluation. Note that this value necessarily needs to be smaller than $6 \sqrt{\frac{m}{H^2+4}}$.

\begin{figure}
\centering
		\begin{subfigure}[t]{0.25\textwidth}
			\centering

				\includegraphics[width=\textwidth]{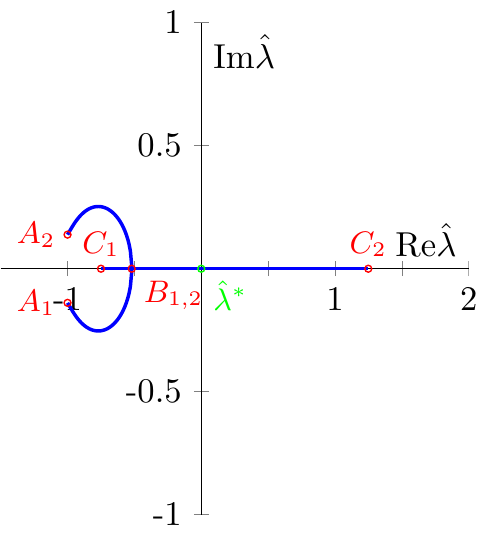}

			\caption{$H=0$, $m=0.45$ ($m < m_c$)}\label{fig:skeletonm045}
		\end{subfigure}
~
		\begin{subfigure}[t]{0.25\textwidth}
			\centering

				\includegraphics[width=\textwidth]{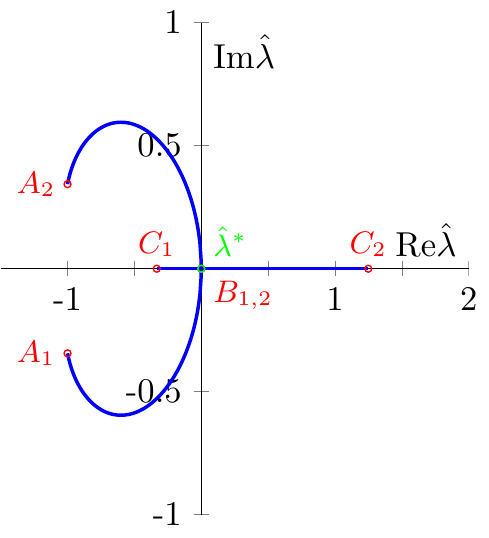}

			\caption{$H=0$, $m = 3$ ($m = m_c$)}\label{fig:skeletonm3}
		\end{subfigure}
~
		\begin{subfigure}[t]{0.25\textwidth}
			\centering

				\includegraphics[width=\textwidth]{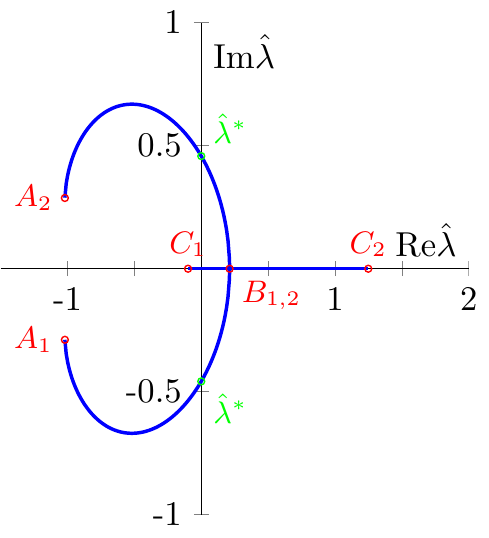}

			\caption{$H=0$, $m = 10$ ($m > m_c$)}\label{fig:skeletonm10}
		\end{subfigure}
	\caption{Plots of the skeletons on which eigenvalues of~\eqref{eq:weaklycondition} necessarily lie. In (a) $H=0$, $m=0.45$ ($m < m_c$), in (b) $H = 0$, $m = 3$ ($m=m_c$) and in (c) $H = 0$, $m = 10$ ($m > m_c$). When the right-hand side of~\eqref{eq:weaklycondition} is small -- e.g. for a high rainfall parameter $a$ -- the eigenvalues are located at $A_{1,2}$ and when the right-hand side is big -- e.g. for a low rainfall parameter $a$ -- the eigenvalues are located at $C_{1,2}$. In between they follow the pictured skeleton, changing from a pair of complex eigenvalues to two real eigenvalues at $B_{1,2}$. The critical, destabilizing eigenvalue $\hat{\lambda}^*$ is also depicted in these figures. Note that in (a) eigenvalues cross the imaginary axis by a Hopf bifurcation and in (c) only by a saddle-node bifurcation; (b) corresponds to a Bogdanov-Takens bifurcation.}
	\label{fig:skeletons}
\end{figure}

\subsubsection{Asymptotic considerations}

Although we now understand the eigenvalue problem~\eqref{eq:weaklycondition} completely for any set of parameters, it is useful to still study the asymptotic cases. There are two parameter regimes that will play a role in the analysis of multi-pulse solutions, (1) $\frac{H^2+4}{4m} \gg 1$ (i.e. $m \ll 1 + H^2/4$) and (2) $\frac{H^2+4}{4m} \ll 1$ (i.e. $m \gg 1 + H^2/4$).

\begin{itemize}
	\item[(1)] In the first case, we see that~\eqref{eq:weaklycondition} reduces to
\begin{equation*}
	\frac{m^2 D}{a^2} u_{01}^2 = \frac{\mathcal{R}(\hat{\lambda}) - 3}{\sqrt{ \frac{H^2+4}{4m}}}.
\end{equation*}
Rewriting this gives the condition
\begin{equation*}
	\sqrt{1 + H^2/4}\ \frac{m \sqrt{m} D}{a^2} u_{01}^2 = \mathcal{R}(\hat{\lambda}) - 3.
\end{equation*}
From our previous analysis we know that the eigenvalues $\hat{\lambda}$ have negative real parts, when the left-hand side is small enough. Thus assumptions (A3) and (A5) now guarantee that the left-hand side is (asymptotically) small and therefore that the pulse solution is stable. Only when these size assumptions are violated is it possible for the pulse to become unstable, in this regime. Moreover, destabilisation happens through a saddle-node bifurcation in this regime.
	\item[(2)] In the second case, equation~\eqref{eq:weaklycondition} reduces to
\begin{equation*}
	\frac{m^2 D}{a^2} u_{01}^2 = \frac{\mathcal{R}(\hat{\lambda})-3}{\sqrt{\hat{\lambda}}}.
\end{equation*}
This time we see that assumptions (A4) and (A6) indicate that the left-hand side is (asymptotically) small and therefore that the pulse is stable. When these size assumptions are violated the pulse may become unstable, via a Hopf bifurcation.
\end{itemize}

\begin{figure}
	\centering
	\begin{subfigure}[t]{0.3 \textwidth}

		\centering

			\includegraphics[width=\textwidth]{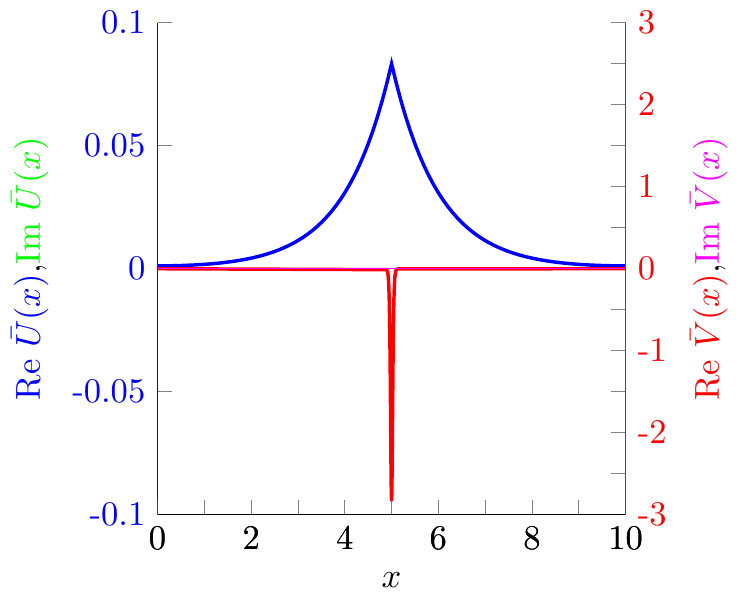}

		\caption{}
		\label{fig:stab-EF-SN}
	\end{subfigure}	
~
	\begin{subfigure}[t]{0.3 \textwidth}
		\centering

			\includegraphics[width=\textwidth]{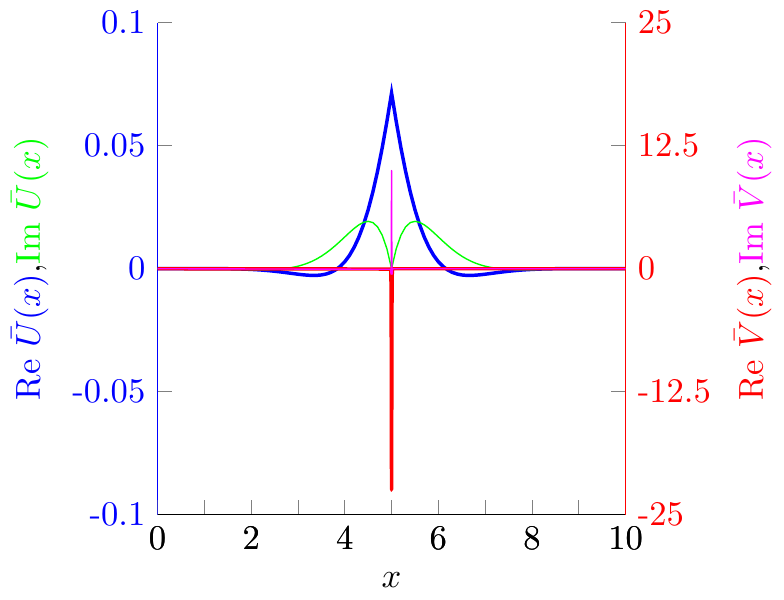}

		\caption{}
		\label{fig:stab-EF-HF}

	\end{subfigure}

	\caption{Approximation of an analytically obtained eigenfunction of a (single) pulse when we encounter a saddle-node bifurcation (a) or a Hopf bifurcation (b). The real part of the $U$-value is given in blue, the imaginary part in green, the real part of the $V$-pulse in red and its imaginary part in pink. Parameter values $H = 0$, $D = 0.01$, $L = 10$ and $m = 0.45$, $a = 0.19032$ (in a) and $m = 10$, $a = 2.1065$ (in b).}
	\label{fig:STAB-weak-eigenfunction}
\end{figure}

\subsection{Stability of $N$-pulse solutions}\label{sec:stab-N-pulses}

This section is devoted to the stability of multi-pulse solutions and pulse-solutions on bounded domains. The pulses in these solutions interact with each other and the boundary and are therefore moving in space, see section~\ref{sec:existence}. In the stability problem these interactions can show up as well, leading to a more involved stability problem than in the previous section. We consider $N$-pulse solutions, with pulses located at $P_1(t) < \ldots < P_N(t)$. Similar to the existence problem of these solutions, we again have an inner region, $I_j^{in}$, near each pulse and outer regions $I_{j}^{out}$ between each pulses and between the first/last pulse and the boundary.

The stability problem in the outer region is again described by equation~\eqref{eq:stability-ODE-outer}. Here we again see an important distinction between the $m \gg 1 + H^2/4$ and the $m \ll 1 + H^2/4$ situations. In the former case the eigenvalue $\hat{\lambda}$ has a leading order role in the outer problem, whereas in the latter case it only has a higher order role. Moreover, in the situation with $m \gg 1 + H^2/4$ the pulses are far apart in the stability problem. As a result the background state $(\bar{U},\bar{V}) = (0,0)$ is approached in between pulses (to leading order). Therefore there is no direct interaction between the pulses in the stability problem in this regime. This leads to a decoupled stability problem in which we can treat the stability of each pulse separately. In the other situations, when $m \leq \mathcal{O}(1 + H^2/4)$, this effect does not occur and the stability problem of all pulses is coupled. We will consider these situations separately.

\subsubsection{$m \gg 1 + H^2/4$ -- decoupled stability problem}\label{sec:DSP}

Solving the homogeneous ODE in the outer region, equation~\eqref{eq:stability-ODE-outer}, for $\tilde{U}$ gives the general solution
\begin{equation}
	\bar{U}(x) = C_1\ e^{\frac{1}{2} [-H - \sqrt{H^2+4(1+m\hat{\lambda})}] x} + C_2\ e^{\frac{1}{2}[-H + \sqrt{H^2+4(1+m\hat{\lambda})}] x},
\end{equation}
where $C_1$ and $C_2$ are some constants. For easier notation in the forthcoming computations, we let the solution in the outer region between the $j$-th and the $j+1$-th pulse be denoted, equivalently, by
\begin{equation}
	\bar{U}_j(x) = S_{1j}\ e^{ \frac{1}{2}\left[-H-\sqrt{H^2+4(1+m\hat{\lambda})}\right](x-P_j)} + S_{2j}\ e^{\frac{1}{2}\left[-H+\sqrt{H^2+4(1+m\hat{\lambda})}\right](x-P_{j+1})}, \label{eq:STAB-outerregions}
\end{equation}
where $S_1$ and $S_2$ are constants. We can also define the solution in the outer regions with $x < P_1$ and $x > P_N$ in a consistent manner with the definition of $P_0$ and $P_{N+1}$ as described in section~\ref{sec:Pulse-Location-ODE-(A3)}. Since $m \gg 1 + H^2/4$, we see that $- H \pm \sqrt{H^2 + 4 (1 + m \hat{\lambda})} \gg 1$, regardless of the size of $H$ compared to $m$. Therefore $\bar{U}_j(P_j) = S_{1j}$ and $\bar{U}_j(P_{j+1}) = S_{2j}$ to leading order. Thus we can approximate the outer solutions by setting the constants $S_{1j}$ and $S_{2j}$ as follows:
\begin{equation} S_{1j} = \rho_j;\hspace{1cm} S_{2j} = \rho_{j+1};\hspace{1cm} S_{10} = 0;\hspace{1cm} S_{2N} = 0,\label{eq:STAB-decoupled-consts}\end{equation}
where $\rho_j$ is (an approximation of) the value $\bar{U}_j(P_j)$. Note that the thus constructed outer solution $\bar{U}$ automatically is continuous in each pulse location, again to leading order. Similar to the $1$-pulse case, we need to impose jump conditions on the derivative $\bar{U}_x$ at each pulse location. In the outer regions this jump is approximated by
\begin{equation}
	\Delta_{out} \tilde{U}_x(P_j) := \tilde{U}_x(P_j^+) - \tilde{U}_x(P_j^-) = - \rho_j \sqrt{H^2 + 4(1+m\hat{\lambda})}.
\end{equation}
Note the similarities with equation~\eqref{eq:stab-homoclinic-outerjump}. The jump in the inner region can be computed at each pulse. This computation is identical as for the homoclinic pulse in section~\ref{sec:stab-1-pulse}. Hence we obtain (see equation~\eqref{eq:weak-jump-1}):
\begin{equation}
\Delta_{in} \bar{U}_x(P_j) = \frac{a^2}{m\sqrt{m}D} \frac{\rho_j}{u_{0j}^2} \hat{C}(\hat{\lambda}). \label{eq:weak-jump-N}
\end{equation}
where $C(\hat{\lambda})$ is defined in equation~\eqref{eq:stab-definition-Clambda}. Equating both descriptions of the jump gives us $N$ equation that a solution of the stability problem should satisfy:
\begin{equation}
	\frac{a^2}{m\sqrt{m}D} \frac{\rho_j}{u_{0j}^2} \hat{C}(\hat{\lambda}) = - \rho_j \sqrt{H^2 + 4 (1+m\hat{\lambda})}.
\end{equation}
This condition is immediately satisfied when $\rho_j = 0$. After division by $\rho_j$ it is clear that the left-hand side depends on the pulse number $j$, whereas the right side does not. Therefore, we know the eigenfunctions of the linear stability problem generically have one $n \in \{1,\ldots, N\}$ such that $\rho_n \neq 0$ and $\rho_j = 0$ for all $j \neq n$. Thus the pulses are decoupled in the stability problem and eigenfunctions are always localised near a single pulse. A solution with $\rho_n \neq 0$ is only a solution to the stability problem if the jump condition is satisfied, i.e. if it satisfies
\begin{equation}
\frac{a^2}{m\sqrt{m}D} \frac{1}{u_{0n}^2} \int_{-\infty}^\infty (\omega^2 - 2 \omega V_{in}) ds = - \sqrt{H^ 2 + 4(1+m\hat{\lambda})}. \label{eq:LinStabWeakN}
\end{equation}
Note that this is precisely the same condition as we found for the stability of a homoclinic pulse in equation~\eqref{eq:LinStabWeak1}. Thus we can use the conclusions from that case here. That is, eigenvalues necessarily need to lie on the skeleton given in Figure~\ref{fig:skeletonm10}. Moreover, the $n$-th eigenfunction has an eigenvalue with positive real part when $K_n := \frac{m^2 D}{a^2} u_{0n}^2 > K^*(m,H)$. Therefore when $K_j < K^*(m,H)$ for all $j \in \{1,\ldots,N\}$ we know that the solution is stable. However, if $K_k > K^*(m,H)$ for some $k$, we know that the $N$-pulse solution is unstable. More specifically we know that the corresponding eigenfunction has $\rho_k \neq 0$ and consists of a localised pulse, located at $x = P_k$. This linear reasoning now suggests that, as the pattern is destabilised, the $k$-th pulse should disappear.

In degenerate cases, it is possible that multiple pulses have the same $K_j$-value, say the value $\bar{K}$. If that happens, then there exists eigenfunctions that have more than one non-zero $\rho_j$-value. That is, the eigenspace corresponding to the corresponding eigenvalue $\hat{\lambda}$ is multidimensional. To really get a grip on what's happening at a bifurcation in these cases, we need to zoom in on the corresponding eigenvalues $\hat{\lambda}$, where the stability problem becomes a coupled stability problem once again. This has already been done in the case of spatially periodic pulse configurations~\cite{Hopf-Dances}, where Floquet theory has been used to find the form of the possible eigenfunctions. From this we know that in these situations -- when there are multiple pulses with the same $K_j$-value -- the eigenvalues are asymptotically close together, though still separated. Moreover, the eigenfunctions become combinations of the single-pulse eigenfunctions that we have already encountered. In fact, in~\cite{Hopf-Dances}, it is found that the most unstable eigenfunction will always be a period-doubling Hopf bifurcation (when $\gamma = -1$) or a full desertification bifurcation (when $\gamma = 1$). At present, it is not clear how we get from the simple, one-pulse eigenfunctions to these more involved (periodic) eigenfunctions as patterns evolve towards regularity. These two types of destabilisations are intertwined in an involved way, which is explained by the appearance of `Hopf dances'~\cite{Hopf-Dances,Ddestab}. We refrain from going in the details here.

\subsubsection{$m \leq \mathcal{O}(1 + H^2/4)$ -- coupled stability problem}\label{sec:CSP}

As before, we can use the outer solution~\eqref{eq:STAB-outerregions}. However, we can no longer use the approximations in~\eqref{eq:STAB-decoupled-consts}, which leads to more involved eigenfunctions that have localised structures at all pulse locations. To find eigenfunctions, we need to understand when a function is an eigenfunction of this (now) coupled stability problem. Foremost, we need to have continuity of $\bar{U}$ at each pulse location, i.e. $\bar{U}_j(P_j) = \bar{U}_{j-1}(P_j)$. Secondly at each pulse location there will be -- as before -- a jump in the derivative $\bar{U}_x$, of size $\Delta_{in} \bar{U}_x(P_j) = \frac{a^2}{m \sqrt{m} D} \frac{\rho_j}{u_{0j}^2} \hat{C}(\hat{\lambda})$, where $\rho_j := \bar{U}(P_j)$ and $\hat{C}$ as in~\eqref{eq:stab-definition-Clambda}.

With these two conditions it is possible to find the value for the constants $S_{1j+1}$ and $S_{2j+1}$ when we are given the values of $S_{1j}$ and $S_{2j}$ and the eigenvalue $\hat{\lambda}$. Thus, when given the value $S_{10}$ and $S_{20}$ that satisfy the left boundary condition, it is possible to deduce the constants $S_{1N}$ and $S_{2N}$ by using the algebraic relations coming from the continuity of $\bar{U}$ and the jump in $\bar{U}_x$ at each pulse location. The concept of finding the eigenfunctions is now simple: the eigenvalues $\hat{\lambda}$ are precisely those values that lead to constants $S_{1N}$ and $S_{2N}$ that satisfy the right boundary conditions. Note that when we are using periodic boundary conditions things get a bit more involved. In this case we can only fix either $S_{10}$ \emph{or} $S_{20}$. Say we've fixed $S_{10}$. This time we must then find a \emph{combination} of $S_{20}$ and $\hat{\lambda}$ that lead to $S_{1N}$ and $S_{2N}$ that are identical to $S_{10}$ and $S_{20}$ respectively.

We recall that $V_{in}$ can be found explicitly, as function of $\hat{\lambda}$ with the use of hypergeometrical functions~\cite{doelman1998stability, Split-Evans}, as we have seen before in section~\ref{sec:RlambdaProperties}. Therefore it is possible to find good approximations of the eigenfunctions -- and the corresponding eigenvalues -- using this outlined method. Depending on the precise configuration of the $N$ pulses and the parameters of the model, the form of the eigenfunctions changes. Because these eigenfunctions have localised structures at all of the pulse locations -- unlike in the case $m \gg 1 + H^2/4$ -- it is in general hard to draw strong conclusions about the dynamics of the pattern beyond the linear destabilisation, i.e. what happens when an eigenvalue crosses the imaginary axis and the solution `falls' off the manifold $\mathcal{M}_N$. In section~\ref{sec:NumericalSimulations} we will see that there essentially are two distinct possibilities: when pulses are irregularly arranged and when the pulses form a regular pattern.

\subsubsection*{Eigenvalues when $m \ll 1 + H^2/4$}

Even for the most simple $N$-pulse configurations it is hard to find the correct values for $S_{1j}$ and $S_{2j}$ by hand. It is, however, possible to say something about the \emph{eigenvalues} in the asymptotic case $m \ll 1 + H^2/4$. When $m \ll 1 + H^2/4$ we see that the exponents in~\eqref{eq:STAB-outerregions} become independent of $\hat{\lambda}$. To be more precise we find $\bar{U}_j$ is given up to exponentially small errors by
\begin{equation*}
	\bar{U}_j(x) = S_{1j} e^{ \frac{1}{2} [-H + \sqrt{H^2 + 4}]} + S_{2j} e^{\frac{1}{2} [-H - \sqrt{H^2+4}]}.
\end{equation*}
Therefore the jump of the derivative $\bar{U}_x$ at each pulse location, as dictated by the stability problem in the outer regions, becomes independent of $\hat{\lambda}$ as well:
\begin{align*}
	\Delta_{out}\bar{U}_x (P_j) =
&\frac{1}{2} \left( H (S_{2j-1}-S_{1j}) - \sqrt{H^2+4}(S_{1j}+S_{2j-1}) \right) + \frac{1}{2}  H \left( S_{1j-1} e^{\frac{1}{2}[-H-\sqrt{H^2+4}]\Delta P_{j-1}} - S_{2j} e^{-\frac{1}{2}[-H+\sqrt{H^2+4}]\Delta P_j} \right) \\
& + \frac{1}{2} \sqrt{H^2+4} \left( S_{1j-1} e^{\frac{1}{2}[-H-\sqrt{H^2+4}]\Delta P_{j-1}} + S_{2j} e^{-\frac{1}{2}[-H+\sqrt{H^2+4}]\Delta P_j} \right)
\end{align*}
As always we need this jump to be equal to the jump as indicated by the fast, inner regions. That is, we need to have
\begin{equation*}
\Delta_{out} \bar{U}_x(P_j) = \Delta_{in} \bar{U}_x(P_j) = \frac{a^2}{m \sqrt{m} D} \frac{\rho_j}{u_{0j}^2} \hat{C}(\hat{\lambda}),
\end{equation*}
where -- as before in~\eqref{eq:STAB-decoupled-consts} -- $\rho_j = \bar{U}(P_j)$; however this time $\rho_j$ does not (implicitly) depend on $\hat{\lambda}$. Note that the only place where $\hat{\lambda}$ comes into play is in the term $\hat{C}(\hat{\lambda})$. This enables us to rearrange the terms such that we find the eigenvalue condition
\begin{equation*}
\frac{m \sqrt{m} D}{a^2} \frac{u_{0j}^2}{\rho_j} \Delta_{out} \bar{U}_x(P_j) = \hat{C}(\hat{\lambda}).
\end{equation*}
Now we note that the right-hand side of this expression does depend only on $\hat{\lambda}$ and not on the pulse $j$ and the left-hand side does only depend on the pulse $j$ and not on $\hat{\lambda}$. Since we have a similar jump condition at all of the pulse locations, we know that the constants $S_{1j}$ and $S_{2j}$ of an eigenfunction must be chosen such that the left-hand side of this equation is the same for all $N$ pulses. That is, we can define
\begin{equation}
	\hat{C}^* = \frac{u_{0j}^2}{\rho_j} \Delta_{out} \bar{U}_x(P_j).
\end{equation}
An eigenvalue must now satisfy the equation
\begin{equation}
	- \frac{ m \sqrt{m} D}{a^2} \frac{\hat{C}^*}{2} = \mathcal{R}(\hat{\lambda}) - 3.
	\label{eq:NPulseConditionSN}
\end{equation}
The right-hand side of this equation is similar to the condition~\eqref{eq:weaklycondition} that we studied for the stability of homoclinic pulses in the limit $m \ll 1 + H^2/4$. Therefore the right-hand side of~\eqref{eq:NPulseConditionSN} is represented by Figure~\ref{fig:NLEPconditions-SN} up to a multiplicative constant and eigenvalues necessarily need to lie on a skeleton, see Figure~\ref{fig:skeletonm045}. The reasoning of said section can be applied here immediately as well: if $-\frac{m \sqrt{m} D}{a^2} \frac{\hat{C}^*}{2}$ is small enough the pulse configuration is stable and when it is too big the configuration becomes unstable. The destabilisation now occurs via a saddle-node bifurcation.

Finally we notice that the left-hand side of~\eqref{eq:NPulseConditionSN} is of order $\O\left( \frac{m\sqrt{m}D}{a^2}\right)$. Therefore if $\frac{m\sqrt{m}D}{a^2} \ll 1$ we know that the configuration necessarily is stable and if $\frac{m\sqrt{m}D}{a^2} \gg 1$ it is unstable. When $\frac{m \sqrt{m}D}{a^2} = \O(1)$ the stability can change and a (saddle-node) bifurcation occurs. A precise computation of the value $C^*$ is necessary to establish stability.


\section{Numerical Simulations} \label{sec:NumericalSimulations}

In this section, we study the behaviour of pulse solutions using the methods developed in the previous sections. We employ our method -- in the form of a Matlab code -- to determine the dynamics of pulses via the ODE as explained in section~\ref{sec:pulseODE-SaddleNode} -- note that this thus does \emph{not} assume $U(P_j) = 0$. Simultaneously, we determine the evolution of the quasi-steady spectrum associated to the evolving multi-pulse configuration. Thus we check whether the pulse configuration approaches the boundary of the $N$-pulse manifold $\mathcal{M}_N$ beyond which it is no longer attracting in the PDE flow -- see section~\ref{sec:linstabsections}. When this happens, we deduce from the eigenfunction analysis which specific pulse -- or pulses -- of the multi-pulse configuration destabilises and in our method we then simply cut out these pulses. This essentially means that we have to assume that the associated quasi-steady bifurcation is subcritical, and thus that the pulse/pulses annihiliate at a fast time scale. Note that this is based on numerical observation in all literature on pulse dynamics in Gray-Scott and Gierer-Meinhardt type models, see~\cite{veerman2015breathing} for a mathematical analysis of this bifurcation in the homoclinic $1$-pulse context (that establishes the subcritical nature of the bifurcation in the Gierer-Meinhardt setting) and~\cite{veerman2013pulses, veerman2015breathing} for a more thorough discussion and examples of systems that do not satisfy this condition.

In our code, the determination of the quasi-steady spectrum can be done in two different ways:
\begin{itemize}
	\item[(DSP)] We treat the quasi-steady spectral problem as if it were a decoupled stability problem, see section~\ref{sec:DSP};
	\item[(CSP)] We treat the quasi-steady spectral problem as if it were a coupled stability problem, see section~\ref{sec:CSP}.
\end{itemize}
There are pros and cons to both methods. The main benefit of (DSP) is that is easy to determine which pulse disappears when a bifurcation happens. On the other hand, this simplification is only valid in the asymptotic region in which $m \gg 1 + H^2/4$ (and when pulses are distinguishable, see section~\ref{sec:DSP}). However, we will see in this section that it also provides useful information when $m \leq \mathcal{O}(1+H^2/4)$. The other method, (CSP), does hold true for all $m$ (and all configurations). However, the eigenfunctions are no longer restricted to a single pulse and can become quite involved. This makes it significantly harder to determine which -- and especially how many -- pulses annihilate as we will see later in this section. Moreover, the (CSP) approach becomes unreliable when the eigenfunctions get large spikes at one pulse location (i.e. for $m \gg 1 + H^2/4$) and when eigenvalues are close, as the underlying root-finding Newton scheme cannot easily distinguish these closely packed eigenvalues.

In our numerical studies in this section we employ our aforementioned approach and test it against direct simulations of the full PDE. We will show that our method is in general good -- even in situations for which our analysis should normally not hold -- but we will also point out its limitations. The outcome of these endeavours will be captures in several conjectures throughout the text. Our numerical study starts with pulse solutions on flat terrains ($h(x) \equiv 0$) in section~\ref{sec:NUM-flatTerrain}. We focus here on the difference between irregular and regular configurations. Subsequently, in section~\ref{sec:NUM-slopedTerrain}, we investigate the effect of topography. Here we encounter downhill movement -- which a priori is counter intuitive from the ecological point of view -- and we study the infiltration of vegetation into bare soil among other things.

In all of our simulations -- both the simulations using our method and the simulations of the full PDE -- we found Hopf bifurcations when $m$ was large and saddle-node bifurcations when $m$ was small. In cases of a Hopf bifurcation, the PDE simulations show a (fast) vibration of the pulses height. In cases of saddle-node bifurcation this vibration was absent. Moreover, the computation of the $u_{0j}$-values, as explained in section~\ref{sec:pulseODE-SaddleNode} was slower. This indicates that the Jacobian determinant is very small, which happens near a (existence) bifurcation -- precisely as expected with a saddle-node bifurcation.

\begin{figure}
	\centering
		\begin{subfigure}[t]{0.4 \textwidth}
			\centering
			\includegraphics[width = \textwidth]{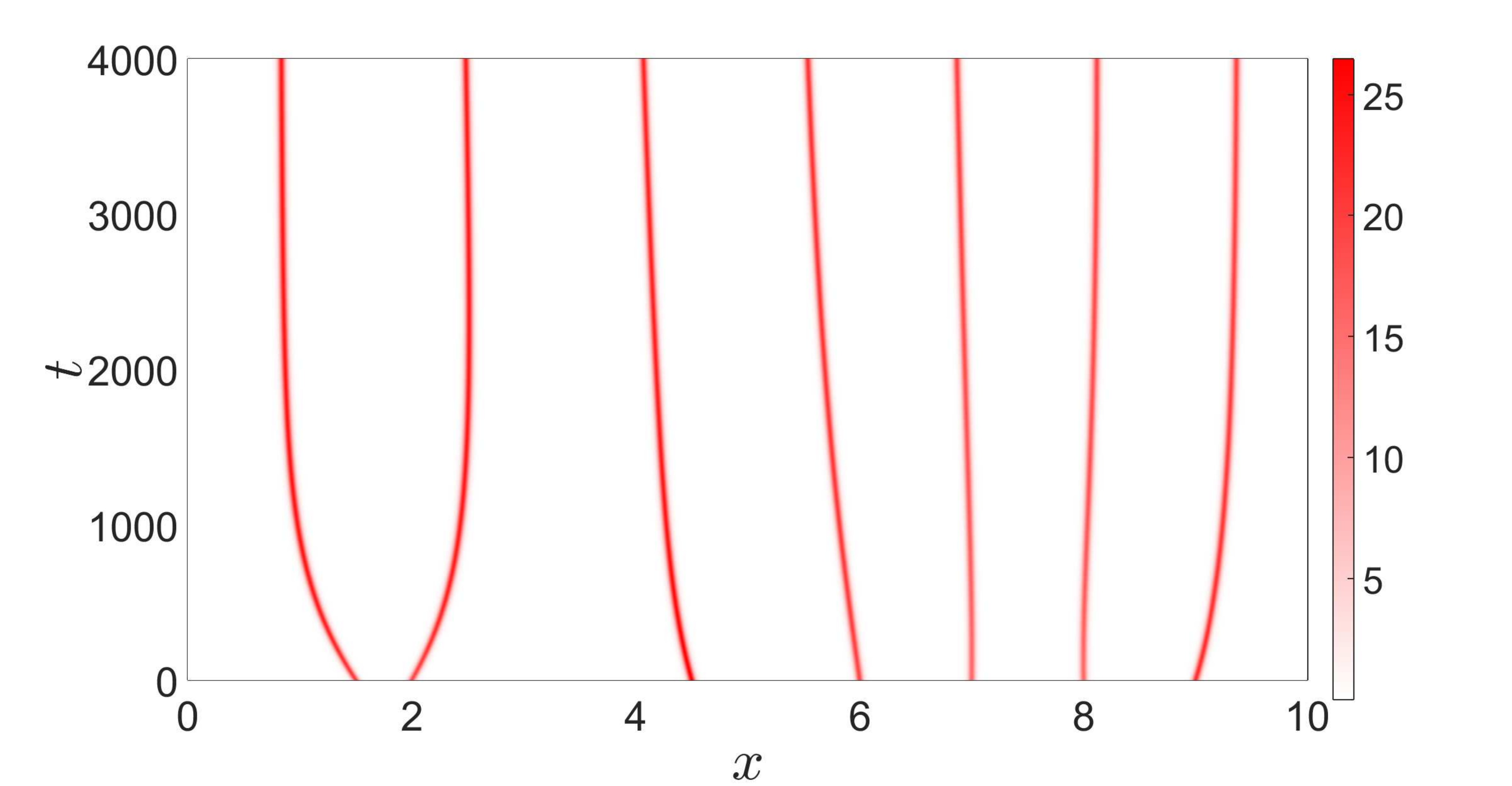}
		\caption{ODE}
		\end{subfigure}
~
		\begin{subfigure}[t]{0.4 \textwidth}
			\includegraphics[width = \textwidth]{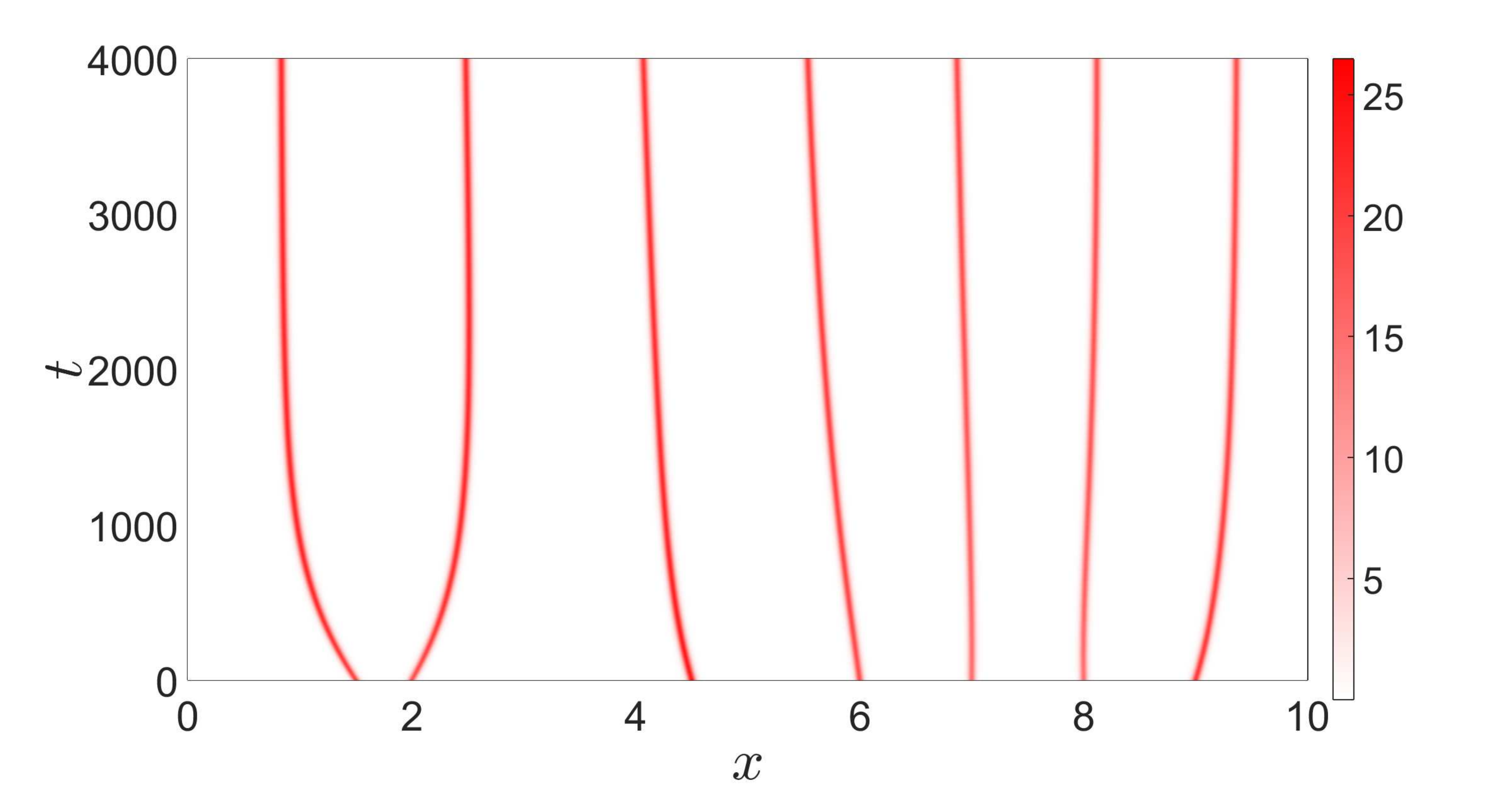}
		\caption{PDE}
		\end{subfigure}
	\caption{Plot of the vegetation $V$ obtained from simulations using the pulse-location ODE (in a) and the full PDE (in b). In these plots the shade of red indicates the concentration of vegetation, with darker meaning more vegetation is located at that position. In both simulations we have used a constant height function $h(x) \equiv 0$ and parameters $m = 0.45$ and $a = 0.5$, $D = 0.01$ and $L = 10$ (ecological relevant parameter values~\cite{klausmeier1999regular, Eric-Beyond-Turing}), and the starting configurations are the same. From these plots it is clear that the ODE and the PDE simulation agree to a great extend, and that the seven pulses evolve to a equally distributed seven pulse solution.}
	\label{fig:FLAT-irrToReg}
\end{figure}
\subsection{Flat terrains}\label{sec:NUM-flatTerrain}

On flat terrains on a bounded domain $[0,L]$ our asymptotic analysis in section~\ref{sec:ODEFixedPoints} -- valid for $\frac{m \sqrt{m} D}{a^2} \ll 1$ -- indicates that a regularly spaced configuration is a stable fixed point of the pulse-location ODE~\eqref{eq:ODEc1}. Both the direct PDE simulation as simulations using our method indicate that these regular patterns are still fixed points and that all $N$-pulse solutions evolve to these regular configurations -- even when $\frac{m \sqrt{m} D}{a^2} = \mathcal{O}(1)$. In Figure~\ref{fig:FLAT-irrToReg} we give an example of this for the situation of a $7$-pulse solution starting from an irregular configuration. So the dynamical movement drives pulse solutions to a regularly spaced configuration (on flat terrains). At the same moment, the flow of the PDE determines the boundaries of the manifold $\mathcal{M}_N$, where $N$-pulse solutions stop to exist and pulses may disappear. We want to understand the bifurcations that occur when a pulse configuration becomes unstable. For this we took the rainfall parameter $a$ as our main bifurcation parameter. In our simulations we let the rainfall parameter decrease such that a bifurcation occurs\footnote{For irregular patterns we need to make sure that the bifurcation occurs fast enough that the pulses have not moved to form a regular pattern yet.}. Our study shows a significant difference between destabilisations of irregular patterns and regular patterns.

\subsubsection{Irregular patterns -- irregular arranged pulses}
\label{sss:irregular}

Two typical configurations with irregularly placed pulses are shown in Figures~\ref{fig:STAB-SS-random-ex} and~\ref{fig:STAB-SS-random-ex-Hopf}. In these configurations we see that the $V$-pulses have varying heights. Consequently the values for $u_{0j}$ differ, with the highest $V$-pulses having the lowest values $u_{0j}$. We have determined the eigenfunction near the bifurcation point for these situations, using the (CSP) method. In all our studies of similar irregular configurations, we have found that the eigenfunctions always look the same (see Figures~\ref{fig:STAB-SS-random-stab} and~\ref{fig:STAB-SS-random-stab-Hopf}): there is a big $\bar{V}$-peak at the location of the pulse with the highest $u_{0j}$-value and the neighbouring pulses have a smaller $\bar{V}$-pulse in the opposite direction. If we -- for a moment -- assume that the pulses are not coupled (like was the case in section~\ref{sec:DSP}), it is clear that the pulse with the highest $u_{0j}$-value is the most unstable one. Indeed, this pulse has the highest value $K_j = m^2 D \frac{u_{0j}^2}{a^2}$, indicating that it is the most unstable one. The corresponding eigenfunction has a single $\bar{V}$-pulse located at this pulses location. When the pulses in the stability problem are coupled, they are relatively close-packed. Consequently, we find (relatively small) $\bar{V}$-pulses for the neighbouring pulses as well. Nevertheless this suggests that such kind of eigenfunctions leads to the death of the pulse with the highest $u_{0j}$-value. Note that linear stability theory does not guarantee this (at all): a priori it cannot be excluded that the neighbouring pulses (also) disappear.

\begin{figure}[t!]
	\centering

		\begin{subfigure}[t]{0.27 \linewidth}
			\centering

			\includegraphics[width=\textwidth]{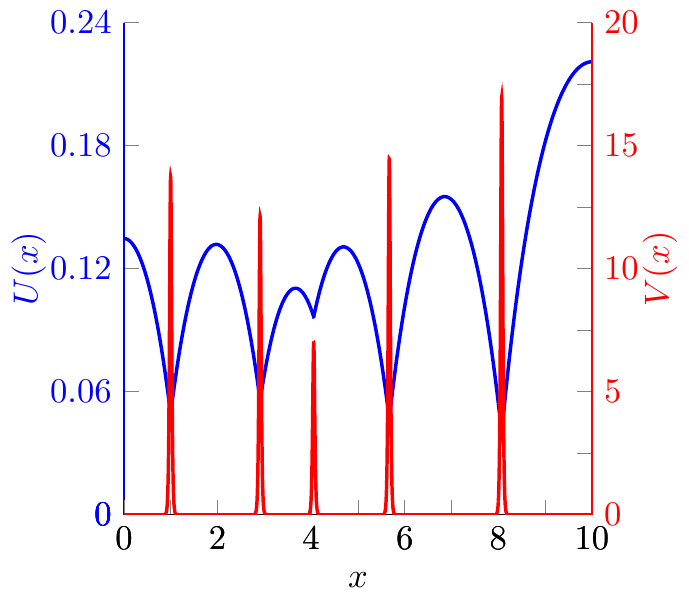}

			\caption{sideview}\label{fig:STAB-SS-random-ex}
		\end{subfigure}
		\begin{subfigure}[t]{0.27 \linewidth}
			\centering

			\includegraphics[width=\textwidth]{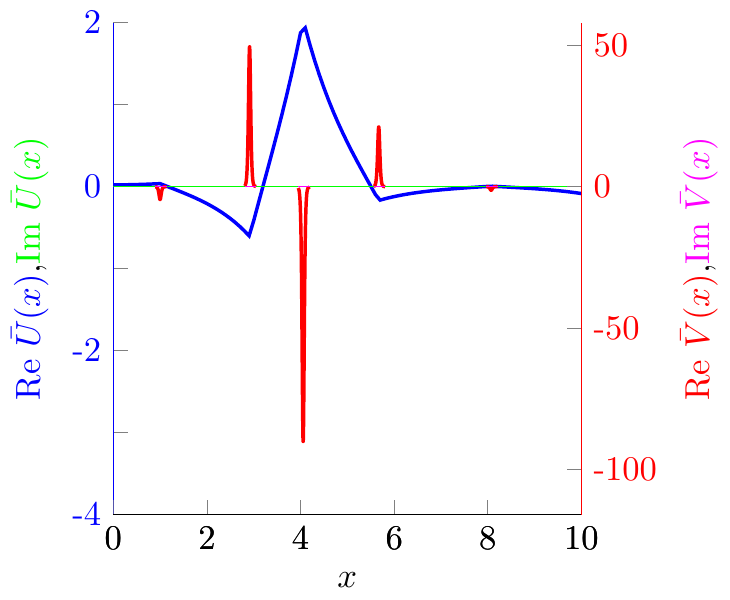}

			\caption{destabilising eigenfunction}\label{fig:STAB-SS-random-stab}
		\end{subfigure}
		\begin{subfigure}[t]{0.4\textwidth}
			\centering
			\includegraphics[width=\linewidth]{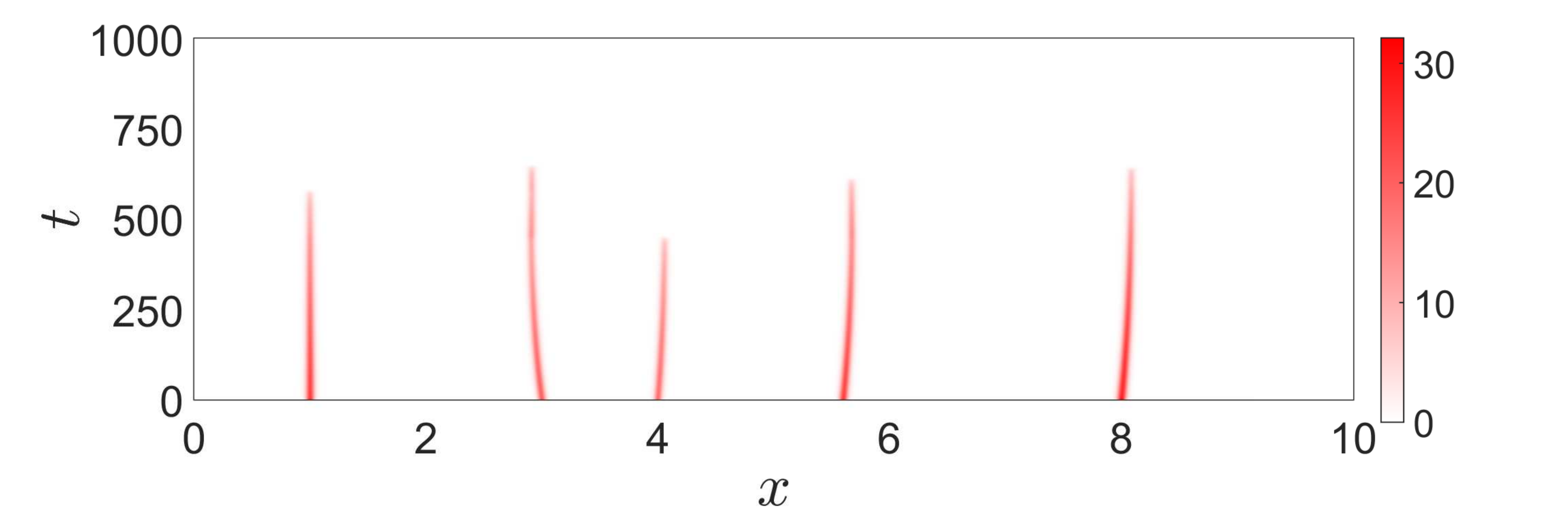}
			\caption{simulation of the full PDE}\label{fig:STAB-SS-random-patternplot}
		\end{subfigure}

	\caption{Sideview (a) of a $5$-pulse configuration near a bifurcation and the destabilising eigenfunction (b) for this bifurcation (on a bounded domain with Neumann boundary conditions) and a simulation of the full PDE (c) with $m = 0.45$, $h(x) \equiv 0$, $D = 0.01$, $L = 10$ and $a_{bif} = 0.296$. The analytically determined quasi-steady eigenvalue is $\hat{\lambda} \approx 0$, suggesting a saddle-node bifurcation (in agreement with our theory for $m < m_c(H)$, see~\eqref{eq:stab-mcritical} and the surrounding text). Here we can clearly see that the most unstable pulse is the pulse with the lowest $V$-peak. This is also found in the eigenfunction plot where this pulse has the highest peak. In the simulation of the PDE we see that our prediction was correct: the third pulse is the first to become unstable. In the PDE simulation we let the rainfall parameter $a$ decrease starting from $a = 0.5$.}
	\label{fig:STAB-SS-random}
\end{figure}

\begin{figure}[t!]
	\centering
		\begin{subfigure}[t]{0.27 \linewidth}
			\centering

			\includegraphics[width=\textwidth]{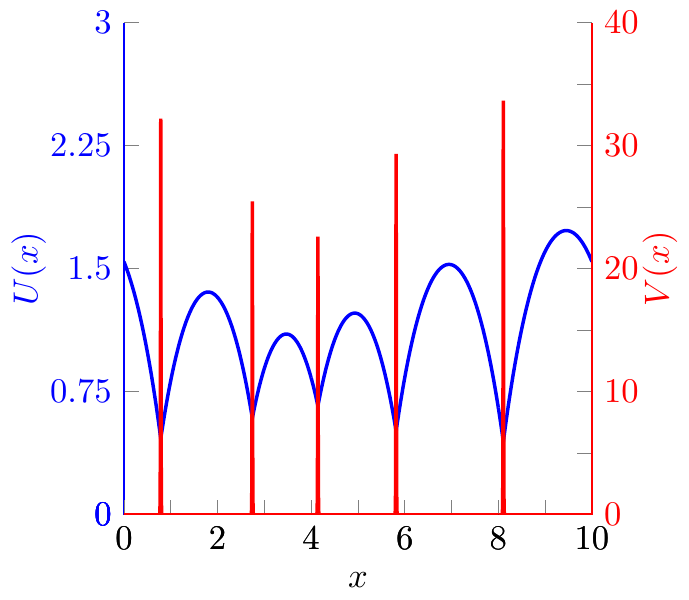}

			\caption{sideview}\label{fig:STAB-SS-random-ex-Hopf}
		\end{subfigure}
~
		\begin{subfigure}[t]{0.27\linewidth}
			\centering

			\includegraphics[width=\textwidth]{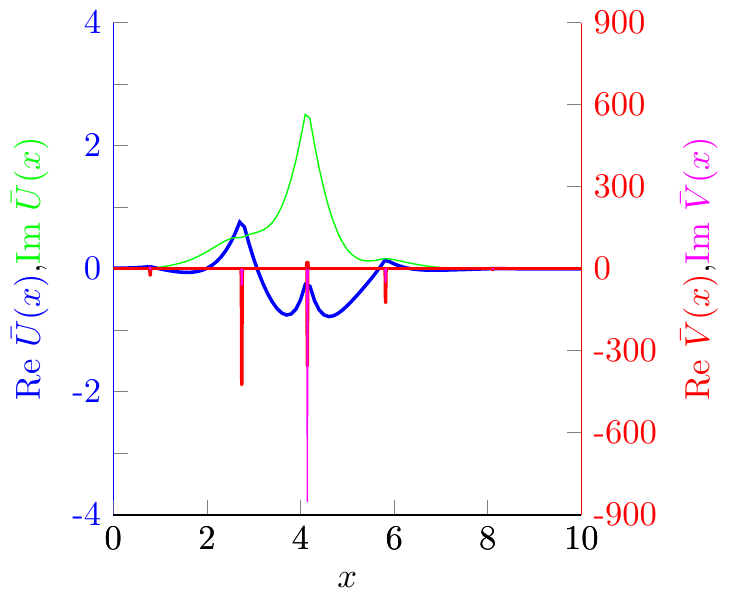}

			\caption{destabilising eigenfunction}\label{fig:STAB-SS-random-stab-Hopf}
		\end{subfigure}
~
		\begin{subfigure}[t]{0.4\textwidth}
			\centering
			\includegraphics[width=\linewidth]{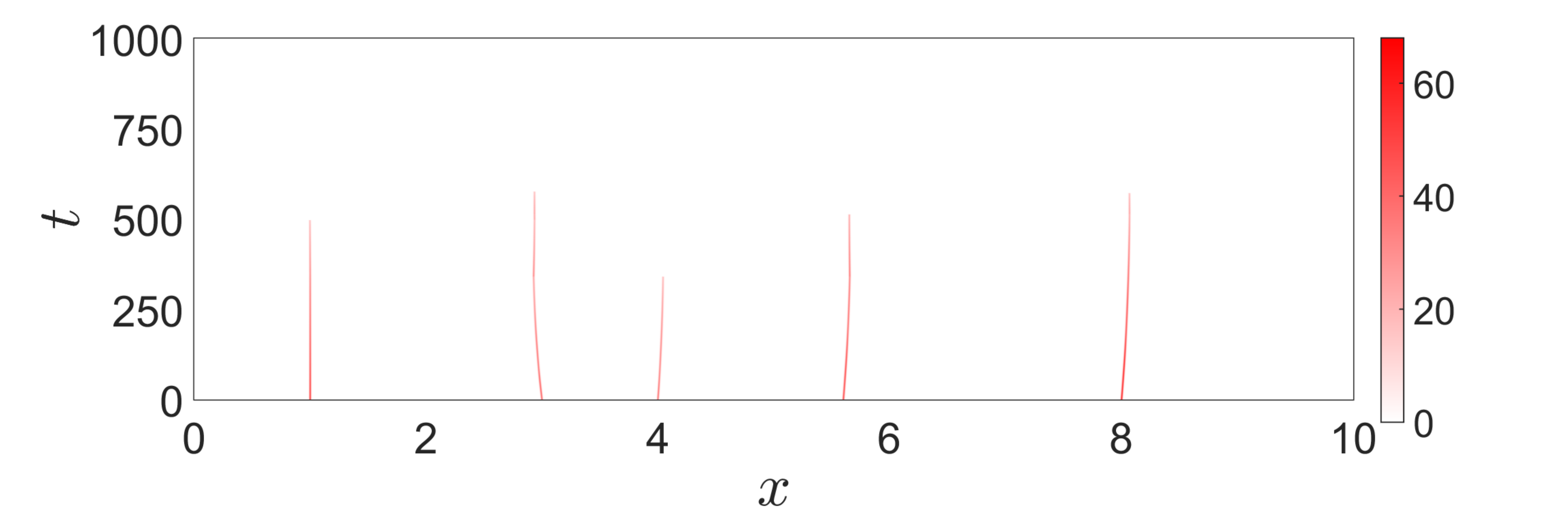}
			\caption{simulation of the full PDE}\label{fig:STAB-SS-random-patternplot-Hopf}
		\end{subfigure}
	\caption{Sideview (a) of a $5$-pulse configuration near a bifurcation and the destabilising eigenfunction (b) for this bifurcation (on a bounded domain with Neumann boundary conditions) and a simulation of the full PDE (c) with $m = 10$, $H = 0$, $D = 0.01$, $L = 10$ and $a_{bif} = 2.96$. The quasi-steady eigenvalue here is $\hat{\lambda} \approx 0.018 \pm 0.472 i$, suggesting a Hopf bifurcation (in agreement with our theory for $m > m_c(H)$, see~\eqref{eq:stab-mcritical} and the surrounding text). Here we again see that the most unstable pulse is the pulse with the lowest $V$-peak, as is also shown in the eigenfunction plot where this pulse has the highest peak. The simulation of the full PDE shows that the third pulse is indeed the first one to become unstable, as was predicted by the linear stability analysis. In the PDE simulation we let the rainfall parameter $a$ decrease starting from $a = 5$.}
	\label{fig:STAB-SS-random-Hopf}
\end{figure}

In numerous PDE simulations we have only ever seen the pulses disappear that have the highest $u_{0j}$-values (i.e. lowest $V$). We have tried to find situations for which this reasoning does not hold, but were unable to find those. Interestingly enough this rule of thumb is good, even when the destabilising eigenfunction does not have an easily recognisable biggest peak. In Figure~\ref{fig:NUM-IRR-whichDies} we encounter such a case. Here one could think from the eigenfunction that pulse 3 should annihilate. However pulse 2 -- the one with the lowest peak in $V$ -- is the one to disappear (and pulse 4 quickly follows).

\begin{figure}[t!]
		\centering
		\begin{subfigure}[t]{0.27 \linewidth}
			\centering

			\includegraphics[width=\textwidth]{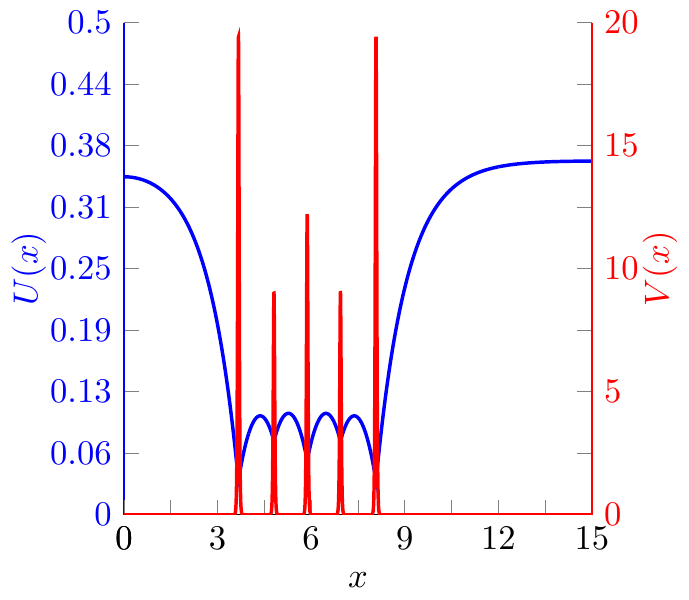}

			\caption{sideview}
		\end{subfigure}
~
		\begin{subfigure}[t]{0.27\linewidth}
			\centering

			\includegraphics[width=\textwidth]{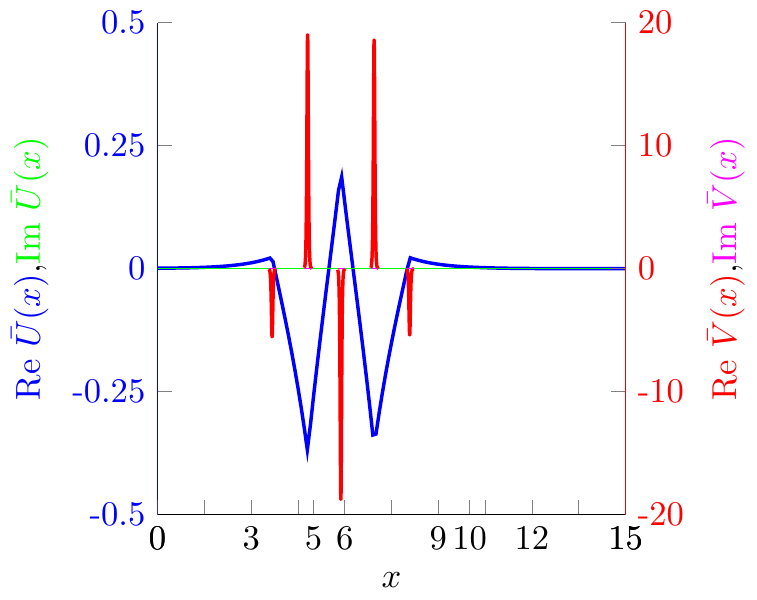}

			\caption{destabilising eigenfunction}
		\end{subfigure}
~
		\begin{subfigure}[t]{0.4\textwidth}
			\centering
			\includegraphics[width=\linewidth]{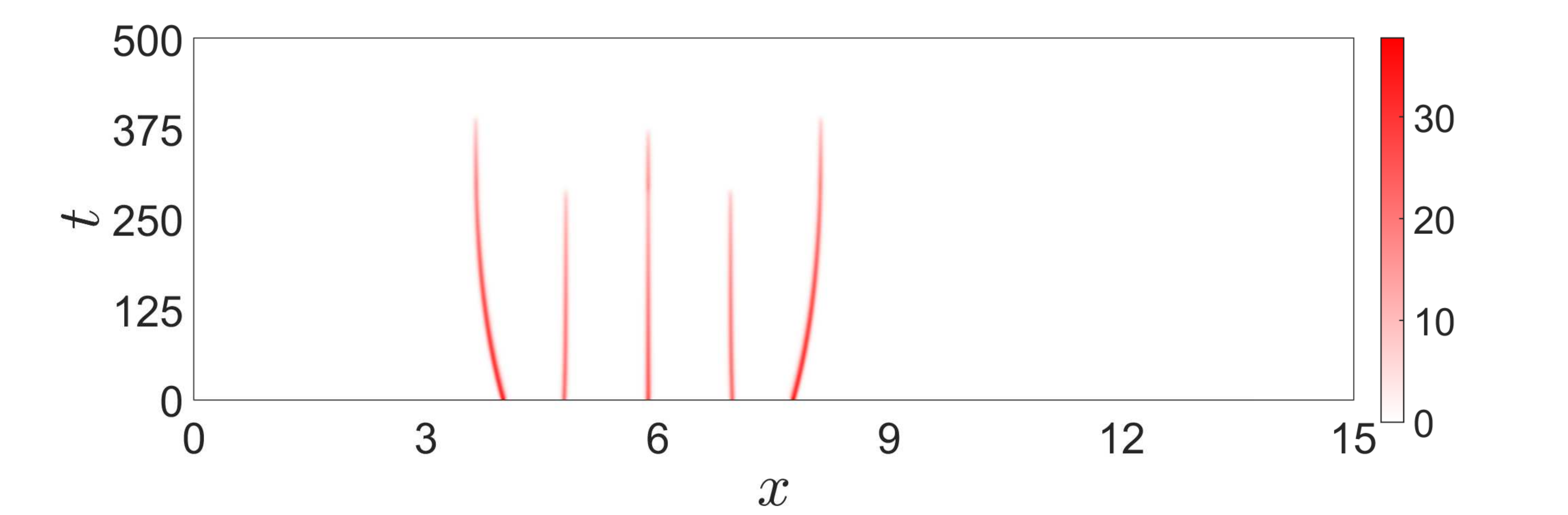}
			\caption{simulation of the full PDE}\label{fig:NUM-IRR-wd-patternplot}
		\end{subfigure}

	\caption{Sideview (a) of a specifically constructed $5$-pulse configuration near a bifurcation and the corresponding destabilising eigenfunction (b) for this bifurcation (on a bounded domain with Neumann boundary conditions) and a simulation of the full PDE (c) with $m = 0.45$, $h(x) \equiv 0$, $D = 0.01$, $L = 15$ and $a_{bif} = 0.36$. The quasi-steady eigenvalue is $\hat{\lambda} \approx 0$, suggesting a saddle-nodee bifurcation. Here we can see that the most unstable pulse is still the pulse with the lowest $V$-peak (pulse 2), although this cannot be seen easily from the eigenfunction here. In the PDE simulation we let the rainfall parameter $a$ decrease starting from $a = 0.75$. Note that another bifurcation, in which pulse 4 dies follows quickly after the disappearing of pulse 2 in the PDE simulation.}
	\label{fig:NUM-IRR-whichDies}
\end{figure}

This all give rise to the following conjecture on the stability of (irregular) $N$-pulse configurations.
\newtheorem{conj}{Conjecture/Observation}
\begin{conj}[Generalised Ni]
	When a multi-pulse pattern is sufficiently irregular, the localised $V$-pulse with the lowest maximum (highest $u_{0j}$-value) is the most unstable pulse, and thus the one to disappear first.
\end{conj}

This conjecture can be seen as a generalisation of Ni's conjecture~\cite{Ni-conjecture}. The value of $u_{0j}$ is determined through the distance between pulses. When pulses are far apart the value of $u_{0j}$ decreases. Consequently the homoclinic pulse, the solitary $V$-pulse, is furthest away from any other pulses and has the lowest $u_{0j}$-value. It should therefore be the most stable configuration, as stated by Ni~\cite{Ni-conjecture, Hopf-Dances,Ddestab}.

This conjecture also helps in the search for the most stable $N$-pulse configuration. Judging from our conjecture, the quasi-steady stability (in the PDE sense) of a $N$-pulse configuration is determined by the maximum of all $u_{0j}$-values, i.e. by $\max_{j\in\{1,\ldots,N\}} u_{0j}$. Therefore the most stable $N$-pulse configuration is the configuration in which all pulses have the same value for $u_{0j}$. Put differently, as long as the manifold $\mathcal{M}_N$ exist, it contains the regularly spaced configuration -- which only becomes unstable under the PDE flow the moment that $\mathcal{M}_N$ is no longer a hyperbolic invariant manifold.

\subsubsection{Regular patterns -- regularly spaced pulses}\label{NUM-flatTerrain-REG}

Understanding the stability and bifurcations of these regular patterns (see Figure~\ref{fig:STAB-SS-reg-ex}) is more difficult. In these configurations all $V$-pulses have the same height and also the values for $u_{0j}$ are equal. Therefore we can no longer speak of the most unstable pulse. We have determined the eigenfunctions and found two different cases depending on the value of $m$. A precise distinction between these two cases -- similar to the critical value $m_c(H)$ in the homoclinic pulse stability study in section~\ref{sec:stab-1-pulse} -- could not be found; it seems this critical value of $m$ might even depend on the number and precise location of all pulses. However, in the asymptotic cases $m \ll 1 + H^2/4$ and $m \gg 1 + H^2/4$, the parameter $m$ definitely is `small' respectively `large'.

\subsubsection*{$m$ small}

When $m$ is small, we only found critical eigenfunctions with alternating one pulse upwards and one pulse downward\footnote{Or a configuration that is closest to this: for instance with an odd number of pulses and periodic boundary conditions there necessarily are two pulses pointing in the same direction next to each other.}, like the example depicted in Figure~\ref{fig:STAB-SS-reg-stab}. This type of eigenfunctions suggests that adjacent pulses evolve differently when the configuration becomes unstable: one of the pulses grows and the other shrinks. PDE simulations back this idea in general. However it is not clear at all from the eigenfunction \emph{which} pulses disappear: the odd ones or the even ones. PDE simulations indicate that both possibilities can happen; it seems to be very sensitive to the initial conditions.

Moreover, it can happen that a (naive) PDE simulation does not follow the critical destabilising eigenfunction but the next most unstable one, see Figure~\ref{fig:STAB-SS-reg-patternplot-symmetry}. This has to do with the symmetry breaking that is necessary to follow the most unstable eigenfunction. Since the PDE (simulation) wants to preserve its symmetry, it only follows eigenfunctions that satisfy the same symmetry -- though that eigenfunction still does resemble a period doubling as much as possible. This issue is easily solved when we apply a non-symmetric perturbation to the initial condition of the PDE.

We also observed that the eigenvalues, corresponding to these destabilizing eigenfunctions, always have $\lambda \approx 0$ (i.e. no imaginary part). This would suggest a saddle-node bifurcation. It was proven in~\cite{Lottes} that there are two periodic $N$-pulse solutions in the Gray-Scott system. One of these is stable and the other unstable, which underpins the possibility of a saddle-node bifurcation~\cite{Lottes}. Moreover, a recent study in a similar model indicates that such kind of saddle-node bifurcations generally are preceded by a period-doubling bifurcation or a sideband bifurcation~\cite{BjornEigenvalues}. Our numerical observations are thus in agreement with these recent discoveries.

\begin{figure}[t!]
\centering
		\begin{subfigure}[t]{0.27 \linewidth}
			\centering

			\includegraphics[width=\textwidth]{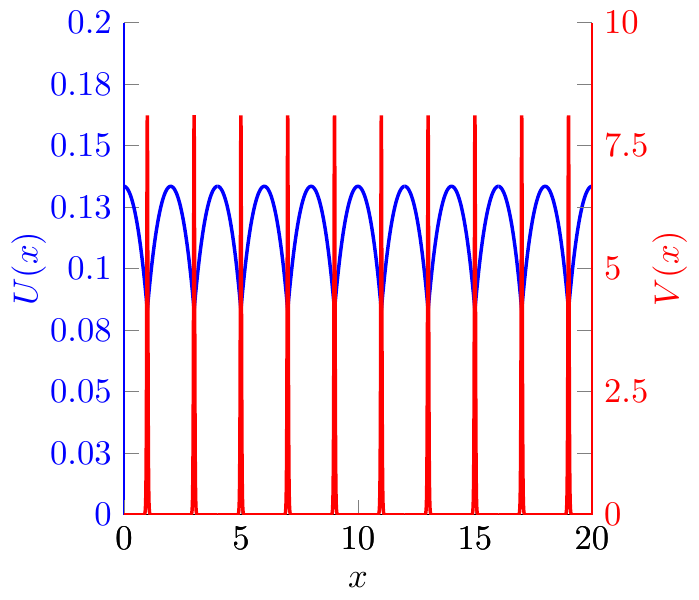}

			\caption{sideview}\label{fig:STAB-SS-reg-ex}
		\end{subfigure}
~
		\begin{subfigure}[t]{0.27 \linewidth}
			\centering

			\includegraphics[width=\textwidth]{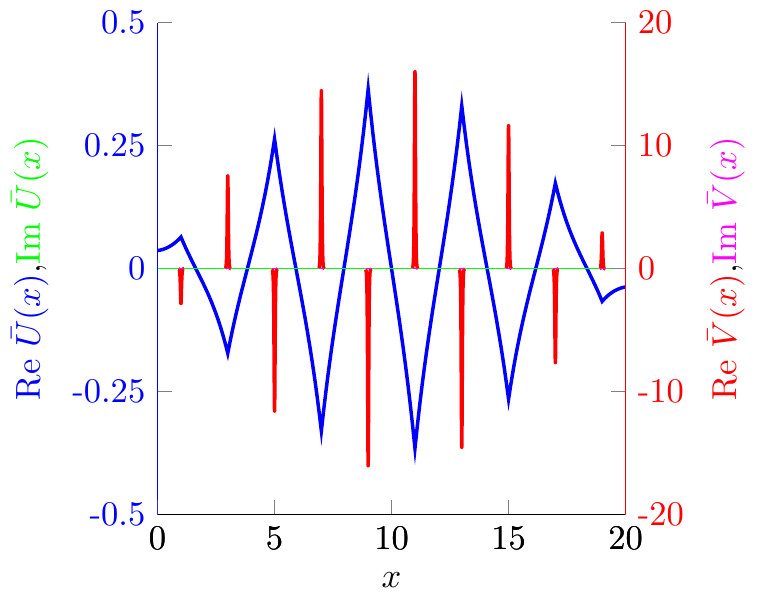}

			\caption{eigenfunction}\label{fig:STAB-SS-reg-stab}
		\end{subfigure}
\\
		\begin{subfigure}[t]{0.4\textwidth}
			\centering
			\includegraphics[width=\linewidth]{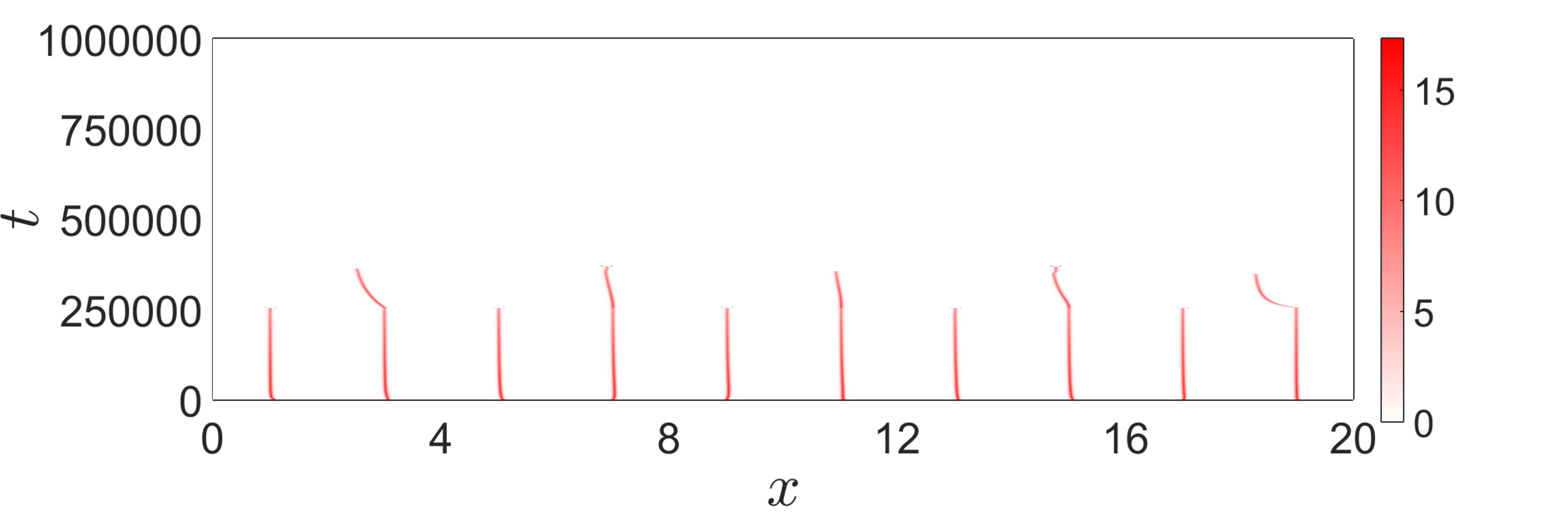}
			\caption{simulation of the full PDE}\label{fig:STAB-SS-reg-patternplot-symmetry}
		\end{subfigure}
~
		\begin{subfigure}[t]{0.4\textwidth}
			\centering
			\includegraphics[width=\linewidth]{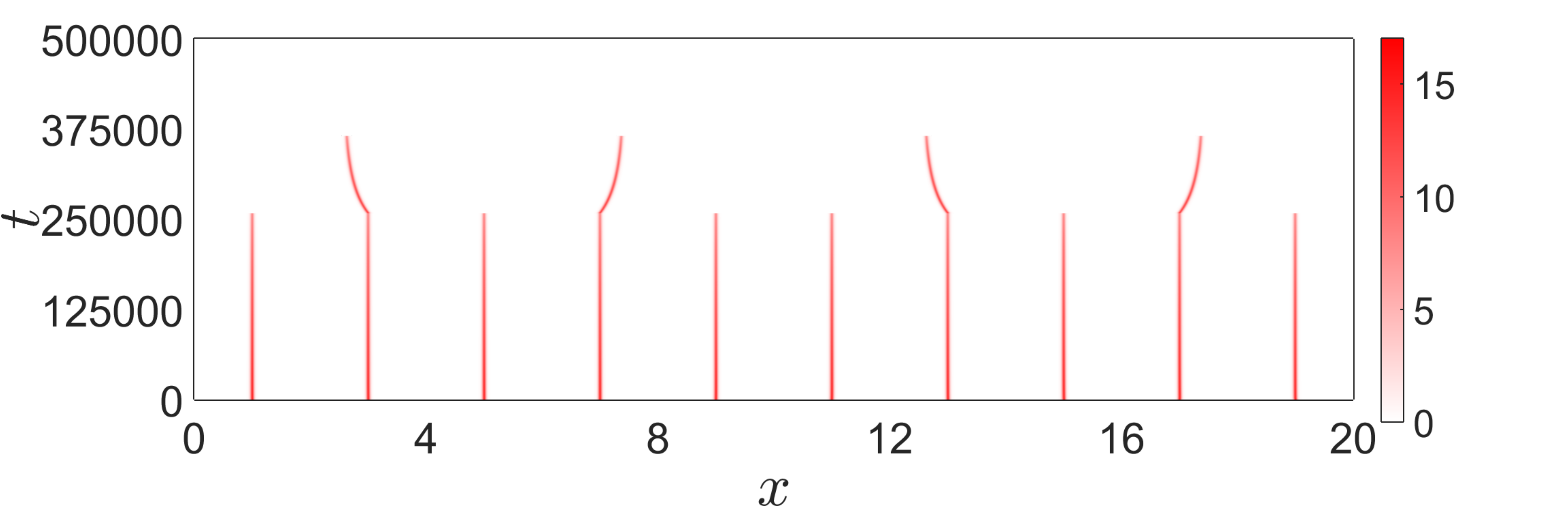}
			\caption{simulation of the full PDE with symmetry issues}\label{fig:STAB-SS-reg-patternplot-symmetry}
		\end{subfigure}

	\caption{Sideview (a) of a $10$-pulse configuration near a bifurcation and the destabilising eigenfunction (b) for this bifurcation (on a bounded domain with Neumann boundary conditions) for parameters $m = 0.45$, $H = 0$, $L = 20$, $D = 0.01$ and $a_{bif}= 0.226$. The eigenvalue here is $\hat{\lambda} \approx 0$, suggesting a saddle-node bifurcation. The PDE simulations (c-d) ran with $a$ decreasing from $a = 0.3$ to $a = 0$, but the initial condition was chosen to be perfectly symmetric in (d). The PDE simulation in (d) does not follow the destabilising eigenfunction from (b), because the solution wants to maintain its symmetry. Therefore the regular $10$-pulse configuration persists for longer time (lower $a$) as well.}
	\label{fig:STAB-SS-reg}
\end{figure}

This gives rise to another conjecture
\begin{conj}[Regular Patterns I]
	When vegetation $V$-pulses form a regular pattern and $m$ is sufficiently small, destabilisation happens via a period doubling bifurcation and the critical eigenvalue crosses $\hat{\lambda} = 0$.
\end{conj}

\subsubsection*{$m$ large}

When $m$ is large, we have encountered two sorts of eigenfunctions in our numerical simulation. Both are Hopf bifurcations; one of them suggests a period doubling bifurcation and the second a full collapsing bifurcation (an example of the latter is shown in Figure~\ref{fig:STAB-SS-reg-stab-Hopf}). As stated in section~\ref{sec:DSP}, this backs the idea of Hopf dances near the tip of the Busse Balloon~\cite{Hopf-Dances, Ddestab}. However in this situation our linear predictions have only limited value. When the eigenfunction point to a period doubling, it is possible that the PDE simulation shows a full collapse and vice versa. The faster we decrease our bifurcation parameter, the more likely it is that this happens -- see Figure~\ref{fig:STAB-SS-reg-Hopf} for an example. Recalling the asymptotic analysis in section~\ref{sec:DSP}, this can be understood as follows: when $m \gg 1 + H^2/4$ the eigenvalues for the period doubling and the full collapse have the same value to leading order. They only differ in higher order. Hence, one cannot expect a priori that the linear stability analysis provides a good grip on the non-linear stability and thus the numerically observed behaviour cannot be too surprising.

\begin{figure}[t]
\centering
		\begin{subfigure}[t]{0.27 \linewidth}
			\centering

			\includegraphics[width=\textwidth]{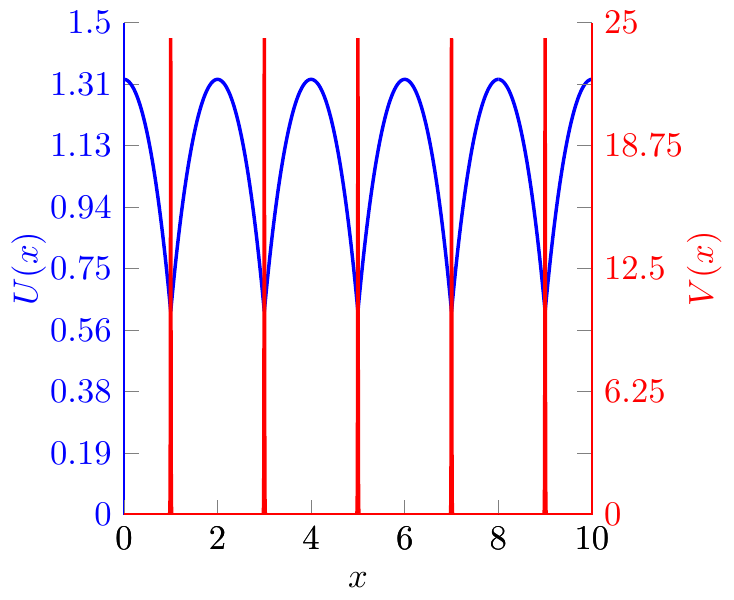}

			\caption{sideview}
		\end{subfigure}
~
		\begin{subfigure}[t]{0.27 \linewidth}
			\centering

			\includegraphics[width=\textwidth]{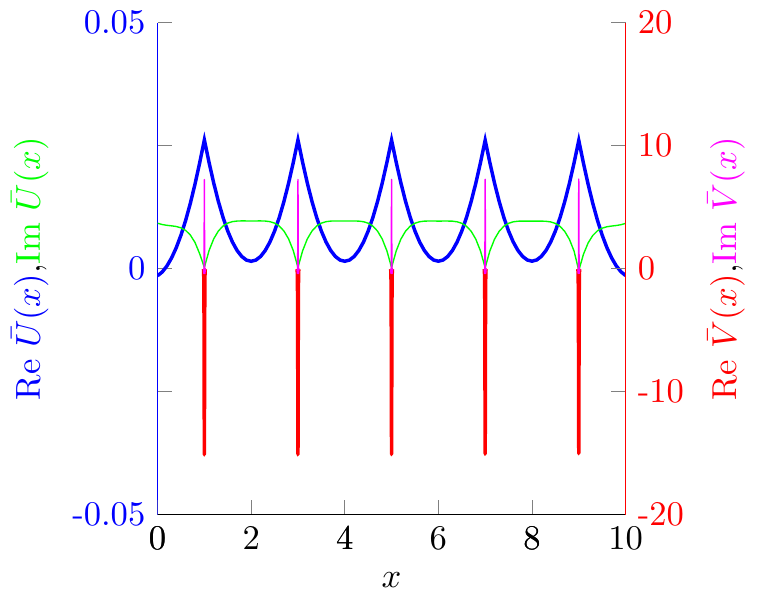}

			\caption{eigenfunction}\label{fig:STAB-SS-reg-stab-Hopf}
		\end{subfigure}
\\
		\begin{subfigure}[t]{0.4\textwidth}
			\centering
			\includegraphics[width=\linewidth]{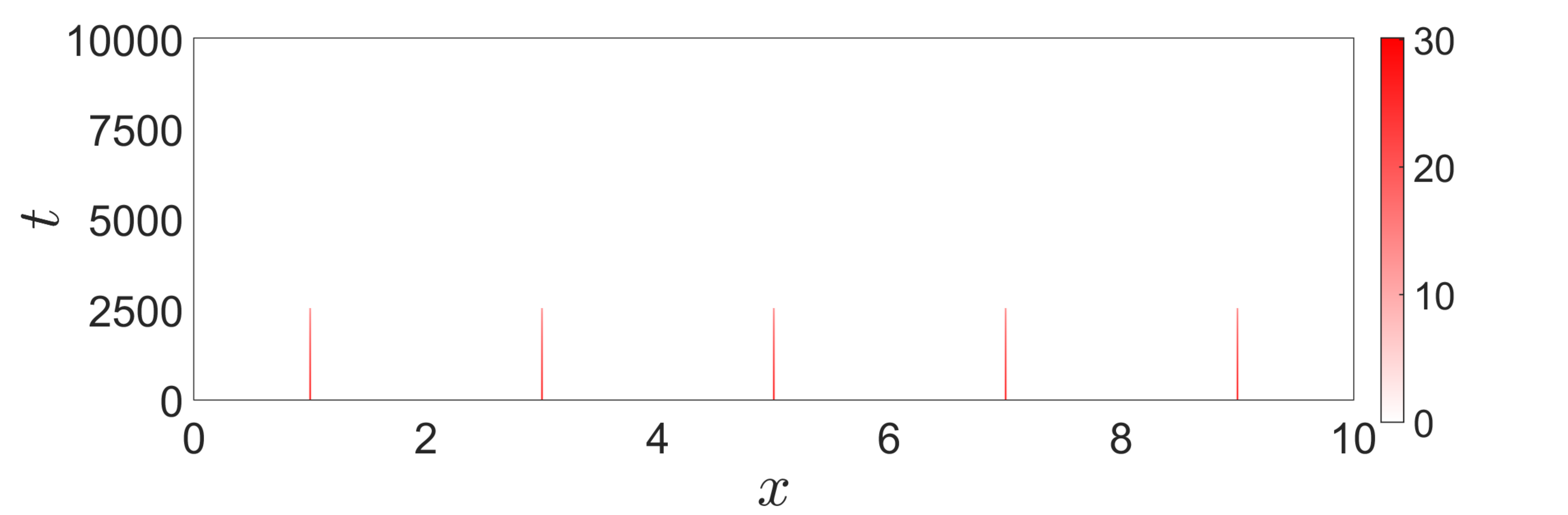}
			\caption{simulation of the full PDE}
		\end{subfigure}
~
		\begin{subfigure}[t]{0.4\textwidth}
			\centering
			\includegraphics[width=\linewidth]{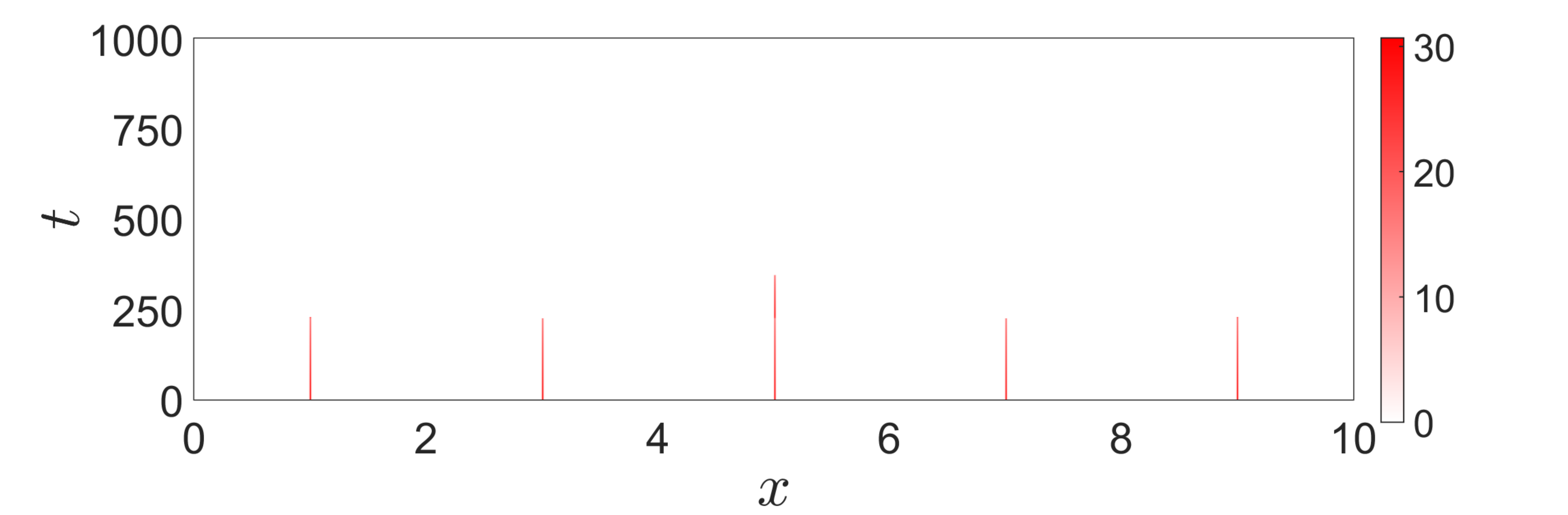}
			\caption{simulation of the full PDE with symmetry issues}
		\end{subfigure}

	\caption{Sideview (a) of a $5$-pulse configuration near a full collapse (Hopf) bifurcation and the destabilising eigenfunction (b) for parameter values $m = 10$, $h(x) \equiv 0$, $L = 10$, $D = 0.01$ and $a_{bif} = 2.622$. The eigenvalue here is $\hat{\lambda} \approx \pm 0.48 i$ indicating a Hopf bifurcation. The PDE simulations ran with $a$ decreasing from $a = 3$ to $a = 0$. The PDE simulation in (d) does not follow the destabilising eigenfunction from (b), possibly because the decrease in the bifurcation parameter was too fast, or because the eigenvalues are close together in this situation (see main text).}
	\label{fig:STAB-SS-reg-Hopf}
\end{figure}

From this all we come to the following conjecture
\begin{conj}[Regular Patterns II]
	When vegetation $V$-pulses form a regular pattern and $m$ is sufficiently large, destabilisation happens either via a period doubling Hopf bifurcation or a full collapse Hopf bifurcation. In these cases the critical eigenvalue has a non-zero imaginary part.
\end{conj}


\subsection{The effect of sloped terrains}\label{sec:NUM-slopedTerrain}

When the terrain is no longer flat, new phenomena occur. To illustrate this, we first consider a constantly sloped terrain, i.e. $h(x) = H x$ on a domain with periodic boundary conditions (section~\ref{sec:NUM-slopePeriodic}) and on a bounded domain with Neumann boundary conditions (section~\ref{sec:NUM-slopeBounded}). Subsequently, we present preliminary results on a terrain that does have a slope that varies in $x$ in section~\ref{sec:NUM-varSlope}.

\subsubsection{Periodic domains}\label{sec:NUM-slopePeriodic}
	
Our study of $N$-pulse solutions on domains with periodic boundary conditions, indicates that these solutions always converge to a configuration in which all pulses are equidistant; i.e. to a regular pattern. This is in agreement with our proofs for the situation $\frac{m\sqrt{m}D}{a^2} \ll 1$ in section~\ref{sec:ODEFixedPoints} and forms a natural extension of our findings in the flat terrain setting of section~\ref{sec:NUM-flatTerrain}. Moreover, the story about eigenfunctions and eigenvalues is also similar to the flat terrain setting: bifurcations of irregular configurations favour single-pulse extinction, whereas regular configurations bifurcate with either a period doubling or a full desertification (depending on the magnitude of $m$ -- see section~\ref{NUM-flatTerrain-REG}). These regular patterns are again the most stable configuration possible for a $N$-pulse solution.

Not everything is the same though: pulses tend to move uphill and therefore solutions are never stationary. Since pulses also try to repel each other this not necessarily means that all pulses always migrate uphill -- this depends on the precise location of all pulses and the size of the slope $H$. However, for regular patterns -- the attracting configuration of the pulse location ODE~\eqref{eq:ODEc1} -- all pulses move uphill. Moreover, we can explicitly determine their migration speed.

\subsubsection*{Uphill migration speed of regular patterns}

We consider a regular pattern with $N$ pulses on a domain with size $L$. For these regular patterns, all pulses are equally far apart from each other. We define this seperation distance -- i.e. the wavelength of the pattern -- as $d := \Delta P_j = L / N$. Substitution of this separation distance in equation~\eqref{eq:ODEforgenH} -- that is derived under assumption (A3) -- gives the speed $\hat{c}_0$ of these regularly spaced pulse configurations as
\begin{equation}
\hat{c}_0(d) = \frac{D a^2}{m \sqrt{m}}\ \frac{\sqrt{H^2+4}}{6}\  \frac{ \cosh( H d / 2) - \cosh( \sqrt{H^2 + 4}\ d / 2) }{\sinh( \sqrt{H^2+4}\ d / 2) } \left( - H + \sqrt{H^2+4} \frac{\sinh(H d / 2)}{\sinh(\sqrt{H^2+4}\ d / 2)} \right).\label{eq:speedperiodic}
\end{equation}
Under the weaker assumption (A3'), the value $\overline{U}(P_j) = \frac{m \sqrt{m} D}{a^2} u_{0j}$ is not necessarily approximately $0$, though the value is the same for all pulses. So we define this value as $u_0 := u_{0j}$. Therefore we may use equation~\eqref{eq:ODEforgenH-deltaO1} to find the speed $\hat{c}$ in this situation as
\begin{equation}
	\hat{c}(d) = \left(1 - \frac{m \sqrt{m} D}{a^2} u_0(d) \right)^2 \hat{c}_0(d).
	\label{eq:speedperiodic-deltaO1}
\end{equation}
with $\hat{c}_0(d)$ as in~\eqref{eq:speedperiodic}.
Note that $u_0$ is not determined in this form and that its value depends on $d$ and $H$. To find this value we need to solve $\vec{F}(\vec{u}_0) = 0$, as explained in section~\ref{sec:pulseODE-SaddleNode}. Although the algebraic equation that needs to be solved is only quadratic in this case, we use a numerical approximation to find the value of $u_0$. In general one sees that the larger the value of $\frac{m \sqrt{m} D}{a^2} u_0$, the slower a (regular) pattern moves -- though it will always move uphill.

\begin{figure}[t]
	\centering
		\begin{subfigure}[t]{0.31 \textwidth}
				\centering

				\includegraphics{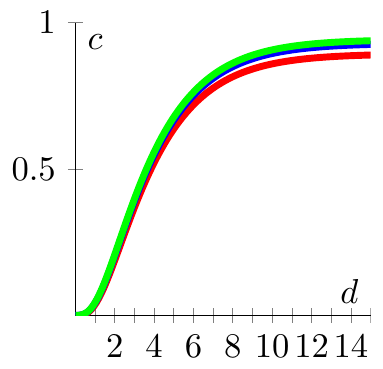}
			\caption{$H = 2$}
		\end{subfigure}
%
%
%
%
~
		\begin{subfigure}[t]{0.31 \textwidth}
				\centering

				\includegraphics{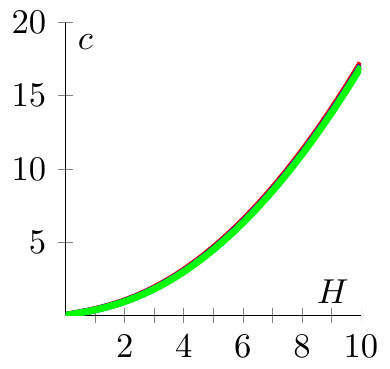}

			\caption{Speed of homoclinic pulse}
		\end{subfigure}
	\caption{Rescaled speed $c$ as a function of $d$ in (a), where $\hat{c}(d) = \frac{D a^2}{m \sqrt{m}} c(d)$. In (b) we show the (rescaled) speed of a homoclinic pulse as a function of the slope $H$. In these plots the green line is a plot of the corresponding parameter-independent equation that are valid under assumption (A3) (i.e. equation~\eqref{eq:speedperiodic} for a and equation~\eqref{eq:speedhomoclinic} for b). The red and blue lines show the evaluations under assumption (A3), when the equations become parameter-dependent (via $u_{0j}$), for $D = 0.01$, $m = 0.45$, $a = 0.5$ (red line) and for $D = 0.01$, $m = 10$, $a = 10$ (blue line).}
	\label{fig:periodic-speeds}
\end{figure}

In Figure~\ref{fig:periodic-speeds} we have plotted the movement speed $\hat{c}_0(d)$ and $\hat{c}(d)$ for several values of $H$. From this it is clear that the farther the pulses are apart, the faster they move. In the limit $d \rightarrow \infty$ we expect them to move at the speed at which a (solitary) homoclinic pulse would move. It follows from equation~\eqref{eq:EX-speed-homoclinic} that these homoclinic pulses move at the speed $\hat{c}_{h}$ given by
\begin{equation}
\hat{c}_h = \left( 1 - \frac{m \sqrt{m} D}{a^2} u_0 \right)^2 \frac{D a^2}{m \sqrt{m}} \frac{H \sqrt{H^2+4}}{6},\label{eq:speedhomoclinic}
\end{equation}
which indeed is also the limit of equation~\eqref{eq:speedperiodic-deltaO1} when we take $d \rightarrow \infty$. Note that $u_0$ is known in this case, see equation~\eqref{eq:EX-u0-homoclinic}. To find the homoclinic speed under assumption (A3) we can simply set $\frac{m \sqrt{m} D}{a^2} u_0 = 0$.

Note that for $d \downarrow 0$ the pulses get closer together. When these pulses get too close together, the linear stability theory of section~\ref{sec:linstabsections} indicates that the configuration is unstable under the PDE flow. Therefore there is a minimum wavelength $d_{min}$ corresponding to a pattern that is marginally stable. Only if $d \geq d_{min}$ we expect to see (stable) periodic patterns. Because the speed of a pattern is a monotonic function of its wavelength -- as directly follows from equation~\eqref{eq:speedperiodic}, see also Appendix~\ref{sec:app-FixedPointProofs} -- we also know that stable periodic configurations can only have speed that is between $\hat{c}(d_{min})$ and $\hat{c}_h$. This agrees with previous theoretical results on the speed of homoclinic pulse solutions~\cite[Equation (5.3)]{Lottes}.

\subsubsection{On bounded domains}\label{sec:NUM-slopeBounded}

Next, we consider $N$-pulse solutions with a constantly sloped terrain on a bounded domain with Neumann boundary conditions. Once again, the fixed point analysis of section~\ref{sec:ODEFixedPoints} -- that was valid under assumption (A3) -- is verified by PDE simulations. Moreover, the results again carry over to the situation in which (A3') holds (i.e. $U(P_j) \approx 0$ does not hold); both simulations of the pulse-location ODE~\eqref{eq:ODEforgenH-deltaO1} as direct PDE simulations always show that all $N$-pulse configuration that start on manifold $\mathcal{M}_N$ evolve to a specific configuration that depends on the parameters of the model (but not on the initial conditions). This specific configuration is the (stable) fixed point of the pulse-location ODE~\eqref{eq:ODEforgenH-deltaO1}. In Figure~\ref{fig:FP-sloped} we have plotted these fixed points as function of the slope $H$ and different number of pulses. These fixed points are obtained as the outcome of simulations of the pulse-interaction ODE~\eqref{eq:ODEforgenH-deltaO1}, with the method as explained in section~\ref{sec:pulseODE-SaddleNode}. From these plots we see an increase in the terrain's slope leads to fixed points that get closer to the boundary of the domain. Moreover, it shows that a simplification of assumption (A3') to (A3) generally leads to the same fixed points, unless the system is close to a saddle-node bifurcation.

\begin{figure}[t]
	\centering
		\begin{subfigure}[t]{0.24 \textwidth}

			\centering

				\includegraphics[width=\textwidth]{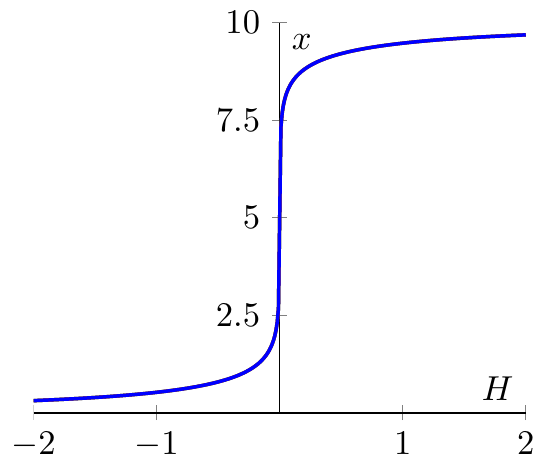}

			\caption{$N=1$, $L = 10$}
		\end{subfigure}
		\begin{subfigure}[t]{0.24 \textwidth}

			\centering

				\includegraphics[width=\textwidth]{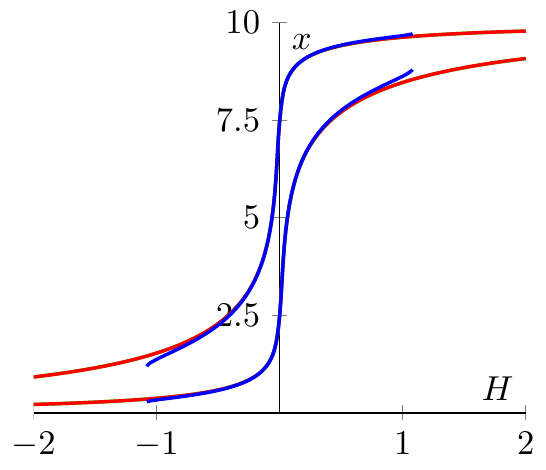}

			\caption{$N=2$, $L=10$}
		\end{subfigure}
		\begin{subfigure}[t]{0.24 \textwidth}

			\centering

				\includegraphics[width=\textwidth]{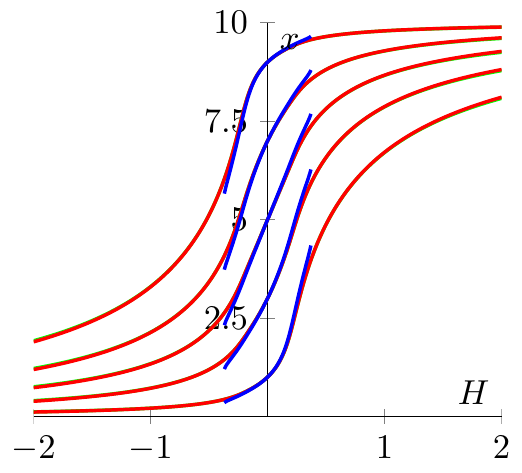}

			\caption{$N=5$, $L=10$}
		\end{subfigure}
		\begin{subfigure}[t]{0.24 \textwidth}

			\centering

				\includegraphics[width=\textwidth]{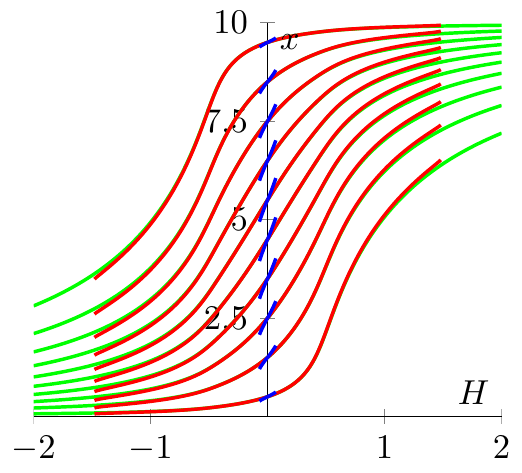}

			\caption{$N=10$,$L=10$}
		\end{subfigure}
	\caption{Stable fixed points of the pulse position ODE on a bounded domain of length $L = 10$ with Neumann boundary conditions for various number of pulses $N$. The green lines (generally laying below the other lines) indicate the fixed point of equation~\eqref{eq:ODEforgenH}, which is valid under assumption (A3). The blue and red lines indicate the fixed points of equation~\eqref{eq:ODEforgenH-deltaO1}, which is valid under assumption (A3'), for parameters $D = 0.01$, $m = 0.45$ and $a = 0.5$ (blue) or $a = 5$ (red). These lines are only plotted when the numerical solver could solve $\vec{F}(\vec{u}_0) = 0$; when it couldn't, a stationary pulse solution does not exist and a saddle-node bifurcation has happened, see section~\ref{sec:pulseODE-SaddleNode}.}
	\label{fig:FP-sloped}
\end{figure}

It should be noted that these (stable) fixed points of the ODE do not need to be stable fixed points of the full PDE. In fact, it can happen that a $N$-pulse configuration evolves under the ODE-flow to another $N$-pulse configuration that is unstable under the flow of the complete PDE -- even in the case of fixed parameter values. That is, a $N$-pulse configuration crosses the boundary of manifold $\mathcal{M}_N$. In Figure~\ref{fig:PDE-FP-unstable} we show a simulation in which this happens. Here we see that the pulses move uphill -- as indicated by the ODE flow -- and then annihilate -- by the PDE flow. This could also be predicted from Figure~\ref{fig:FP-sloped}, since the ODE does not have a fixed point for these parameters. These simulations also back our generalised-Ni conjecture~1: once again the pulse with the lowest $V$-peak disappears at the bifurcation.

Moreover, we also see that the decoupled stability check (DSP) also captures the PDE behaviour very well. This is remarkable here, since the corresponding asymptotic analysis in section~\ref{sec:DSP} is only valid for $m \gg 1 + H^2/4$, whereas here $m = 0.45 < 1$. At first glance the two simulations seem identical. However a better look reveals that the (DSP) ODE simulation gets rid of pulse slightly too early -- though it does give a good prediction on the pulse that is going to disappear. This effect gets exaggerated when more pulses are added to the simulation. In Figure~\ref{fig:NUM-FP-unstable-N10} we have done a simulation with $N = 10$ pulses. Here the mismatch between ODE and PDE simulation can be seen more easily.

\begin{figure}
	\centering
		\begin{subfigure}[t]{0.4 \textwidth}
			\centering
			\includegraphics[width = \textwidth]{"Figs/eps/ODE-slopeunstable-eps-converted-to"}
		\caption{ODE (DSP)}
		\end{subfigure}
		\begin{subfigure}[t]{0.4 \textwidth}
			\includegraphics[width = \textwidth]{"Figs/eps/PDE-slopeunstable-eps-converted-to"}
		\caption{PDE}
		\end{subfigure}
\\
		\begin{subfigure}[t]{0.27 \textwidth}
				\centering

				\includegraphics[width=\textwidth]{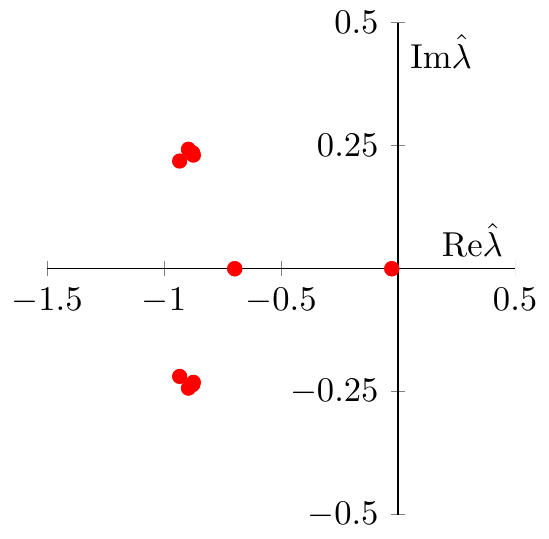}

			\caption{Location of eigenvalues (CSP)}
		\end{subfigure}
		\begin{subfigure}[t]{0.35 \textwidth}
				\centering

				\includegraphics[width=\textwidth]{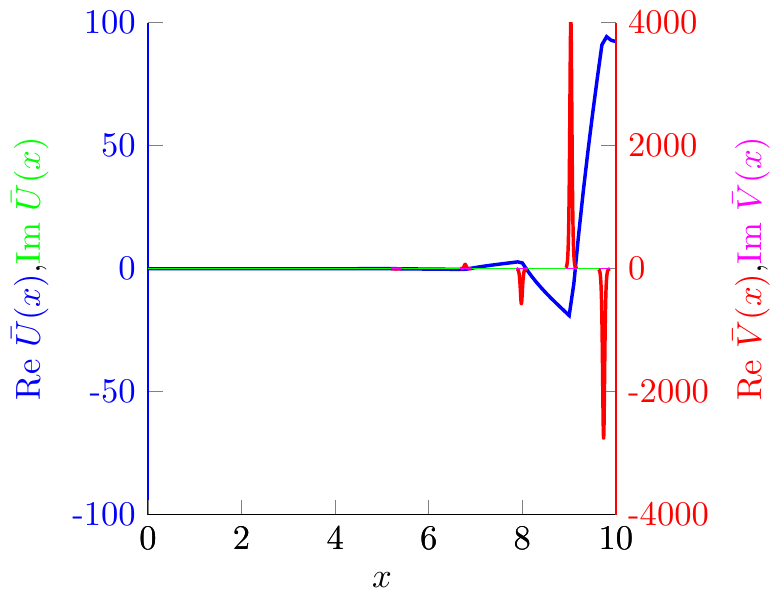}

			\caption{Destabilising eigenfunction (CSP)}
		\end{subfigure}

	\caption{Simulations of the extended Klausmeier model in which the fixed point of the ODE is unstable under the PDE flow, leading to annihilation of the most uphill pulse. The ODE simulation in (a) uses the decoupled stability approximation (DSP). In these simulations, the pulses start as a regular pattern and we have used the parameters $a = 0.5$, $H = 1$, $D = 0.01$, $m = 0.45$, $L = 10$. Moreover we show the location of the eigenvalues close to the moment the first pulse dies out in (c) and the destabilising eigenfunction (corresponding to $\hat{\lambda} \approx 0$) at the same moment is given in (d) -- both are determined using the coupled stability approach (CSP).  Note that (a) and (b) are also shown in Figure~\ref{fig:intro-SixPulsesOnHill}.}
\label{fig:PDE-FP-unstable}
\end{figure}

\begin{figure}
	\centering
		\begin{subfigure}[t]{0.4 \textwidth}
			\centering
			\includegraphics[width = \textwidth]{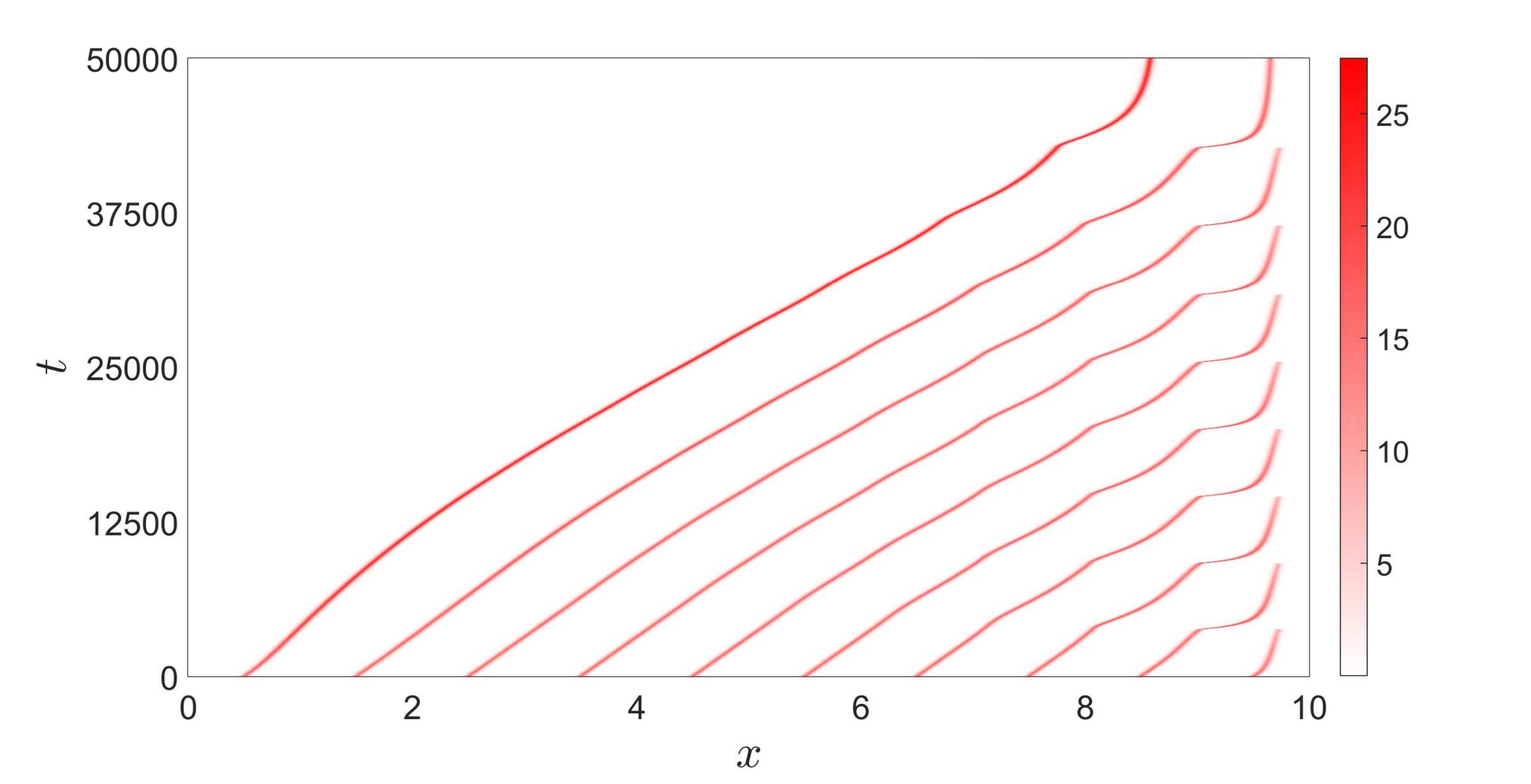}
		\caption{ODE (DSP)}
		\end{subfigure}
~
		\begin{subfigure}[t]{0.4 \textwidth}
			\includegraphics[width = \textwidth]{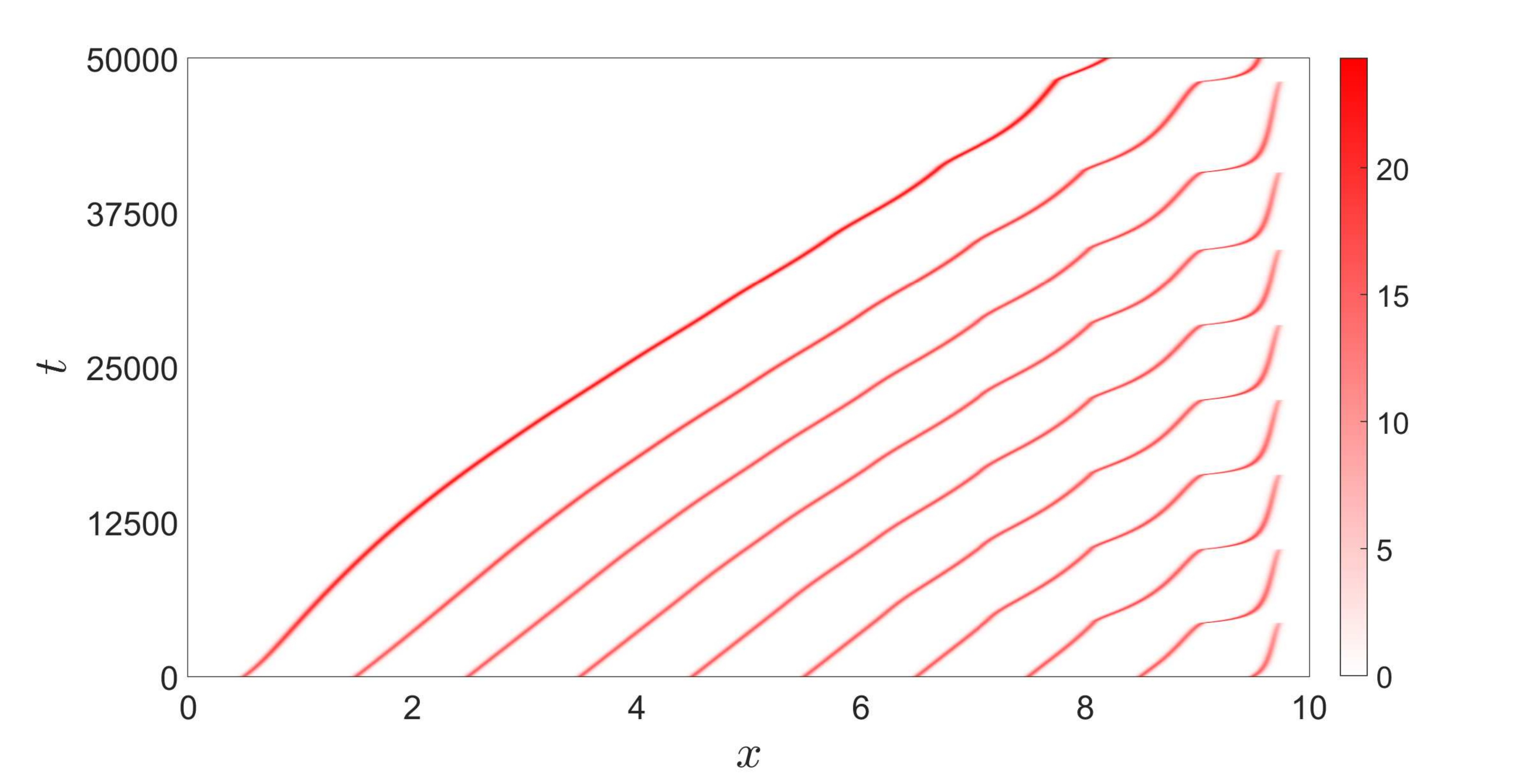}
		\caption{PDE}
		\end{subfigure}
\caption{Simulations of the extended Klausmeier model with $10$ pulses in which the fixed point of the pulse-location ODE is unstable under the PDE flow. The ODE simulation in (a) uses the decoupled stability approximation (DSP). In these simulations pulses start as a regular patterns. The mismatch between the ODE simulation in (a) and the direct PDE simulation in (b) can be seen from comparing these plots. The parameters used are $a = 0.5$, $H = 1$, $D = 0.01$, $m = 0.45$,  $L = 10$.}
\label{fig:NUM-FP-unstable-N10}
\end{figure}

\subsubsection{Varying Terrain}\label{sec:NUM-varSlope}

In the previous sections we have studied the extended Klausmeier model on terrains with a constant slope, i.e. $h(x) = H x$. In these situations it was possible to find an exact form of the solution in the outer regions. When we inspect a terrain with non-constant slope, it is in general not possible to find an exact solution in the outer region because these terms make the outer problem a non-autonomous problem, see~\eqref{eq:outer-ODE-U-A3}. It is therefore more complicated to study a varying terrain problem. In this section we briefly consider some cases, in which we use assumption (A3), i.e. $\tilde{U}(P_j) = 0$. We use a numerical boundary value problem solver, to find numerical approximations of the solution $\tilde{U}$ for the ODE~\eqref{eq:outer-ODE-U-A3} in the outer regions, between the pulses. In these situations we -- again -- see that the reduction gives a very good description of the movement of the pulses (see Figure~\ref{fig:expterrain1}). In the simulation of Figure~\ref{fig:expterrain1}, we have used a Gaussian function for the terrain, i.e. $h(x) = e^{-0.75 \left( x - \frac{L}{2} \right)^2}$, which resembles a hill with a top at $x = L/2$.

In section~\ref{sec:NUM-slopedTerrain} we saw that pulses on a constantly sloped terrain want to move uphill. Therefore one might be inclined to conclude that all pulses want to move uphill. Additional simulation with a single pulse reveal that it is also possible for a pulse to walk downhill. In Figure~\ref{fig:expterrain2} we show two simulations of the full PDE on a Gaussian terrain of the form $h(x) = \exp[ - B \left(x - \frac{L}{2} \right)^2]$. Here we see that the pulse moves uphill when $B$ is small and downhill when $B$ is bigger. This not necessarily contradicts the ecological intuition: we know that the movement of a pulse is determined through the water availability, see equation~\eqref{eq:ODEc1}. When the curvature of the terrain gets too big, it might happen that water streams downhill so fast that water builds up at the base of the hill. This would make this point, at the basis of the hill, the preferred spot for a pulse, because of the abundance of water and therefore the pulse moves downhill towards this point. The extended Klausmeier model with a more general varying terrain term is studied more in-depth in~\cite{varyingTerrainArticle}.

\begin{figure}[t!]
\centering
	\begin{subfigure}[t]{0.4 \textwidth}
			\centering
			\includegraphics[width = \textwidth]{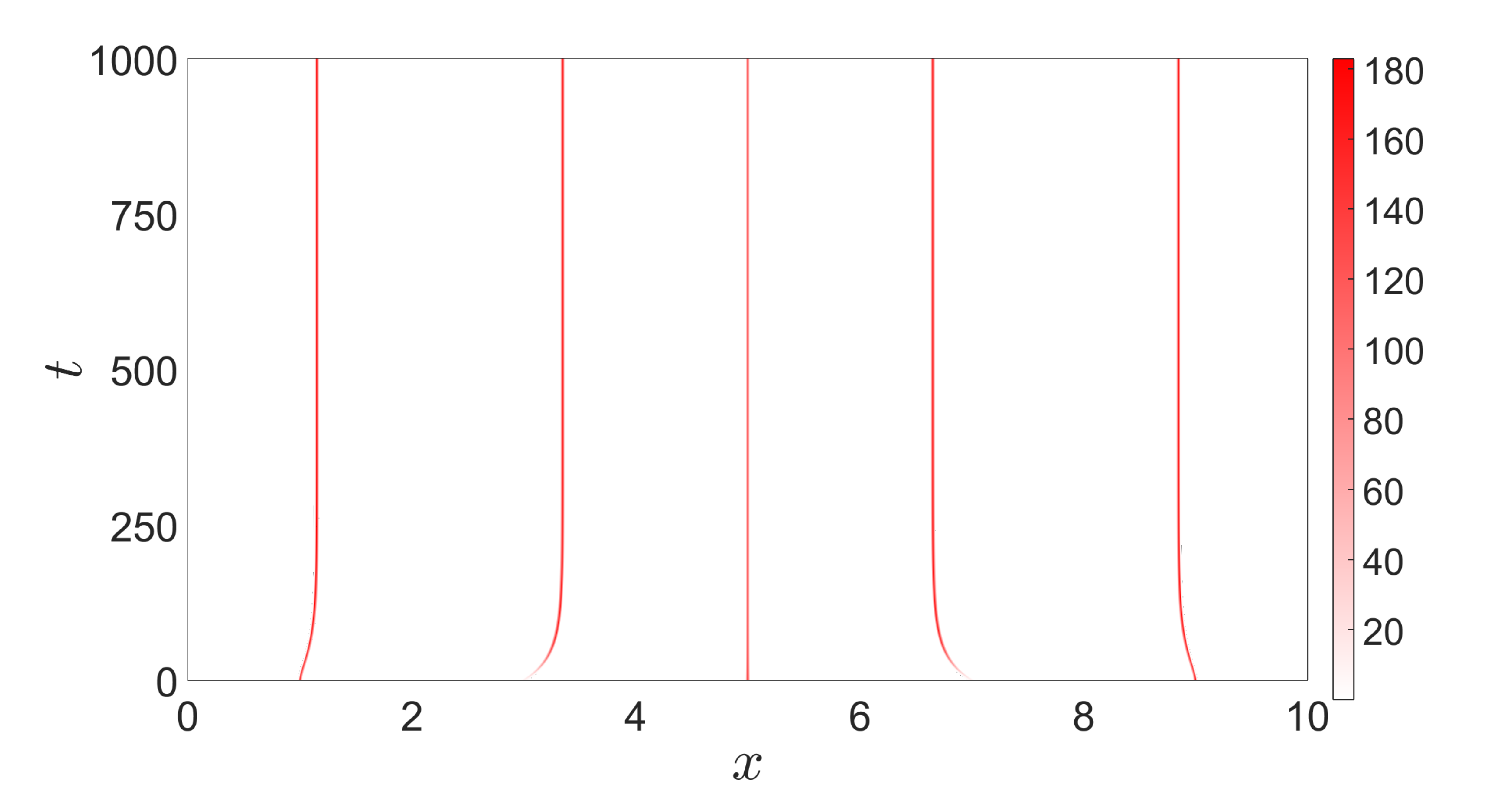}
		\caption{ODE}
	\end{subfigure}
~
	\begin{subfigure}[t]{0.4 \textwidth}
			\includegraphics[width = \textwidth]{"Figs/eps/exp_terrain_5_patternplot_ODE-eps-converted-to"}
		\caption{PDE}
	\end{subfigure}

\caption{The evolution of 5 pulses in simulations of the extended Klausmeier model with non-constantly sloped terrain $h(x) = \exp\left[ - 0.75 \left(x - \frac{L}{2} \right)^2\right]$, for the reduced pulse-location ODE (a) and the full PDE (b). In both simulations we have taken $a = 20$, $m = 20$, $D = 0.01$ and $L = 10$ and the starting configurations are the same, i.e. 5 pulses distributed equally over the domain. From these plots we again see that the ODE reduction agrees with the full PDE dynamics to a great extend.}
\label{fig:expterrain1}

\end{figure}

\begin{figure}[t!]
\centering
	\begin{subfigure}[t]{0.4 \textwidth}
			\centering
			\includegraphics[width = \textwidth]{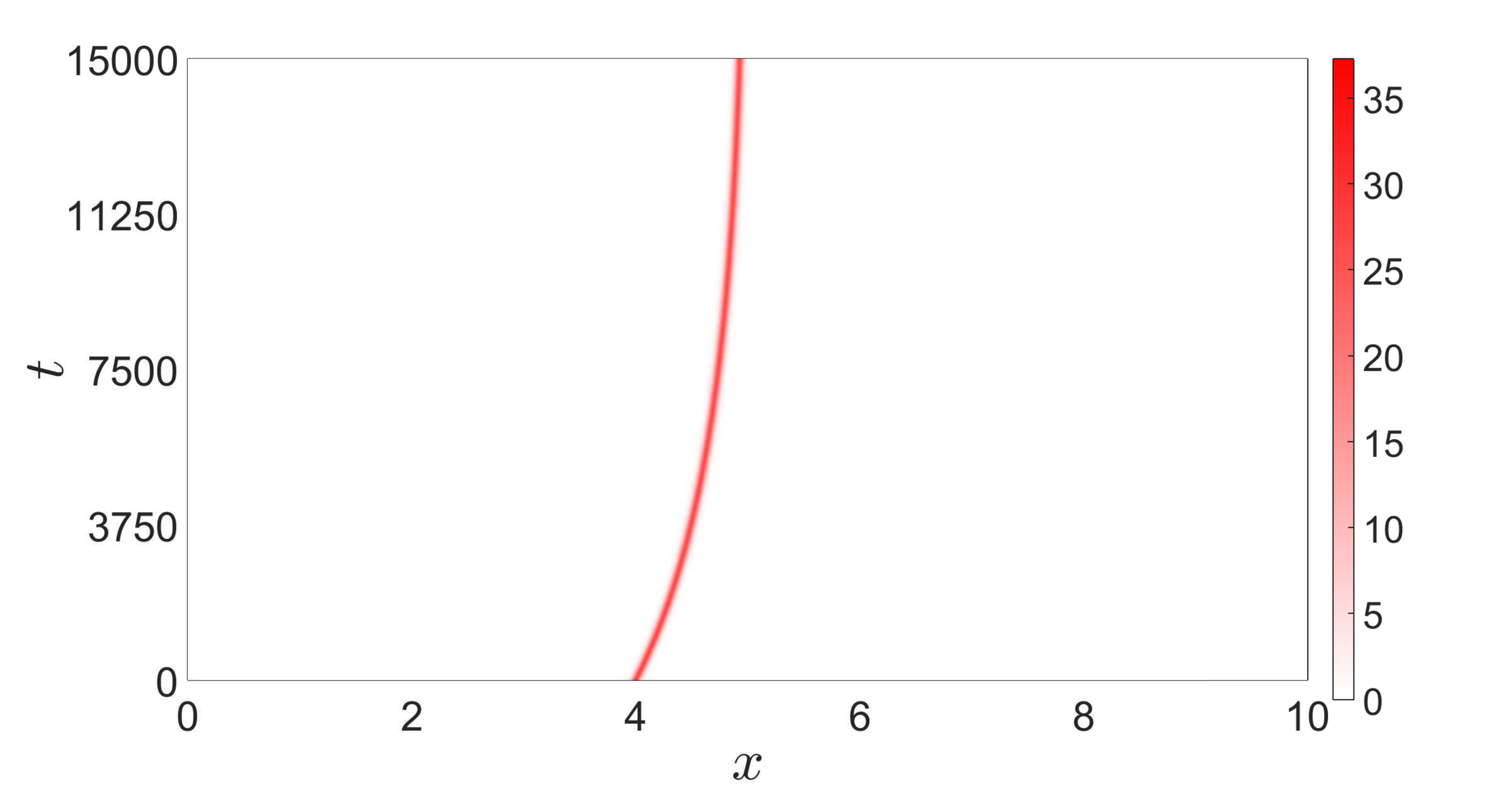}
		\caption{B=0.25: uphill movement}
	\end{subfigure}
~
	\begin{subfigure}[t]{0.4 \textwidth}
			\includegraphics[width = \textwidth]{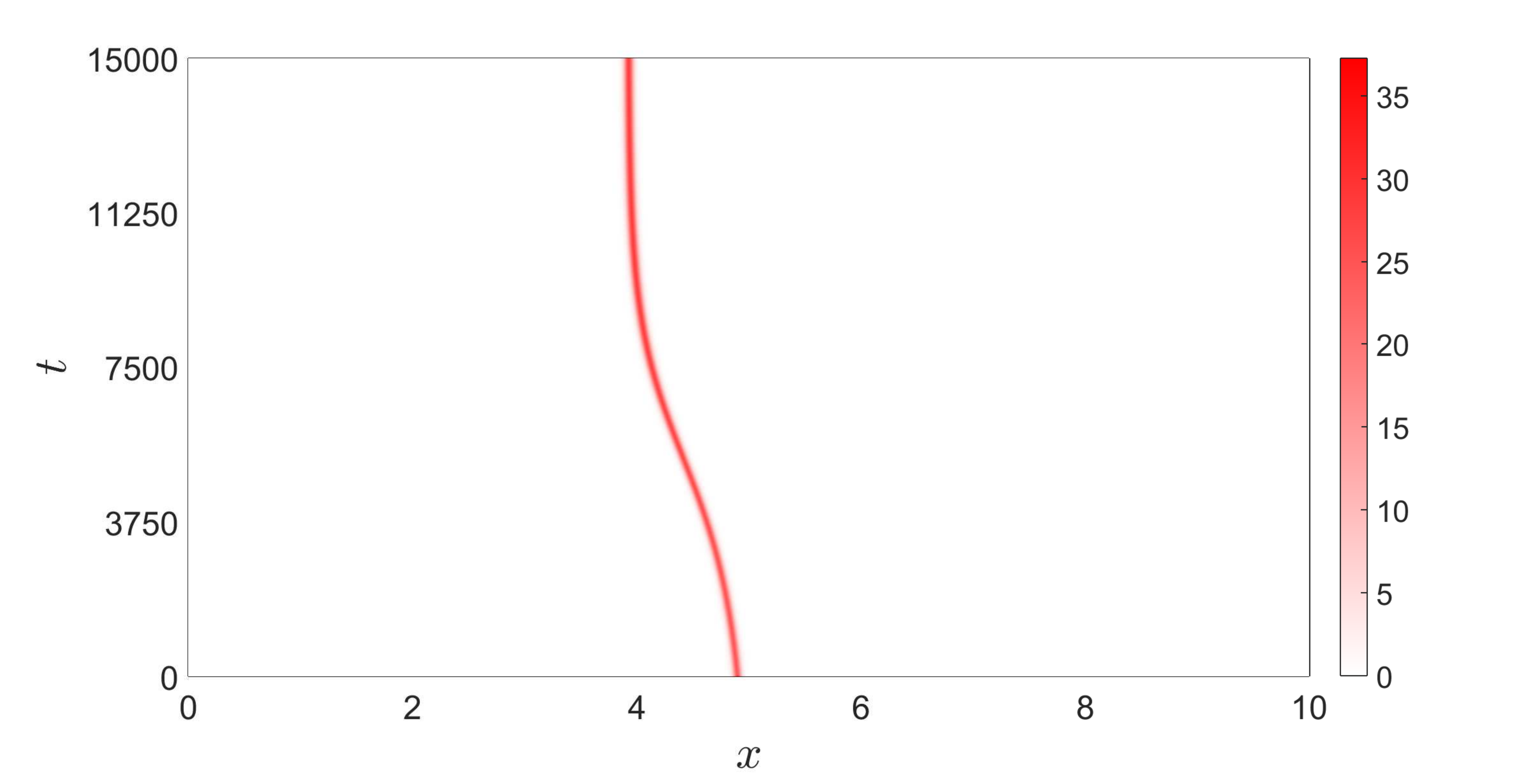}
		\caption{B=1: downhill movement}
	\end{subfigure}
\\

	\begin{subfigure}[t]{0.4 \textwidth}
			\centering
			\includegraphics{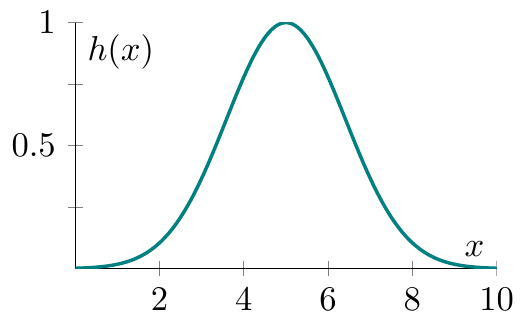}

		\caption{Height function with $B = 0.25$}
	\end{subfigure}
~
	\begin{subfigure}[t]{0.4 \textwidth}
			\centering
			\includegraphics{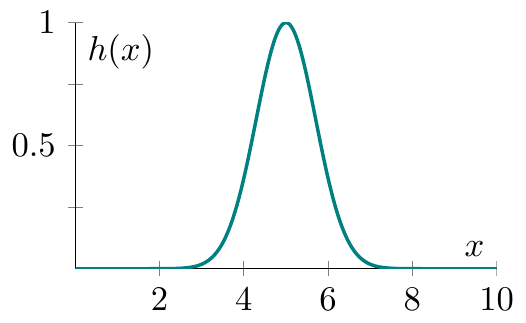}

		\caption{Height function with $B = 1$}
	\end{subfigure}
\caption{PDE Simulations of the full extended Klausmeier PDE model with a terrain with non-constants slope, $h(x) = \exp[-B\left(x-\frac{L}{2}\right)^2]$ for $B = 0.25$ (a) and $B = 1$ (b). Here we see that pulses can move downhill when the width of the hill becomes small. In both simulations we have taken $a = 0.5$, $m = 0.45$, $D = 0.01$, $L = 10$ and used Neumann boundary conditions. The form of the terrains $h(x)$ is plotted in Figures (c) and (d). The pulse-location ODE simulations show similar results (not shown).}
\label{fig:expterrain2}

\end{figure}

\subsubsection{Infiltration of vegetation in bare soil}

Finally, as an illustration of the applicability of our pulse-location ODE, we turn our attention to the phenomenon of colonisation. Observations of vegetation in semi-arid regions in the Sahel showed an inverse relation between the wavelength of vegetation patterns and the slope of the terrain~\cite{Eddy199957}: when the slope increased, the wavelength decreased. Recently, using numerical methods, it was shown that colonization of bare ground leads to the same inverse relationship, suggesting that those regions in the Sahel may once have been deserts~\cite{Sherratt-origins}. With our ODE description~\eqref{eq:ODEforgenH} it is possible to derive an analytic (approximate) expression for this inverse wavelength-slope relation.

The critical wavelength $d_c$ (i.e. the distance between the pulses), for a given terrain with constant slope $H$, is the wavelength for which the uphill moving effect due to the slope of the terrain is negated by the repulsive behaviour of the pulses uphill. If $d> d_c$ the lowest pulse moves uphill and colonization is argued to be unfeasible; if $d < d_c$ the lowest pulse moves downhill and colonization is possible. In our analysis we use assumption (A3), i.e. $\tilde{U}(P_j) = 0$. Therefore we can find the speed of the lowest pulse by only considering the distance to its neighbour pulse. Because the whole problem is symmetric in $H = 0$, we can assume for simplicity that $H \geq 0$. We let the lowest pulse be located at position $P_1$ and we let the distance to the neighbouring pulse uphill be denoted by $d$. We assume that there are no pulses further downhill (i.e. we put $P_0 = -\infty$). From equation~\eqref{eq:ODEforgenH} we then derive the speed of the first pulse as
\begin{equation}
\frac{dP_1}{dt} = \frac{D a^2}{m \sqrt{m}} \frac{1}{6} \left[ \left( \frac{H}{2} - \frac{\sqrt{H^2+4}}{2} \frac{e^{H d / 2} - \cosh( \sqrt{H^2+4}\ d / 2)}{\sinh( \sqrt{H^2+4}\ d / 2) } \right)^2 \right.\left. - \left( \frac{H}{2} - \frac{\sqrt{H^2+4}}{2}\right)^2 \right]. \label{eq:colonization1}
\end{equation}
To find the critical values for the wavelength $d_c$, we need to find the value $d$ for which $\frac{dP_1}{dt} = 0$. That is, we need to find the roots of the terms between the brackets in equation~\eqref{eq:colonization1}. In Figure~\ref{fig:colonization2} the resulting plot is shown. This indeed gives the inverse relationship between the slope $H$ and $d_c$ as reported in~\cite{Sherratt-origins}. It should be noted that these results match up very good when the slope $H$ is large, but start to differ when the slope $H$ is small\footnote{According to~\cite[Figure 5.c]{Sherratt-origins} the critical rainfall value $a_c$ increases when the slope $H$ increases. Therefore small slopes lead to small rainfall parameters, which in turn lead to a violation of assumption (A3).}; Unsurprisingly, precisely for these small slopes assumption (A3) is no longer valid.

\begin{figure}[t!]

	\centering

	\includegraphics[width=0.4\textwidth]{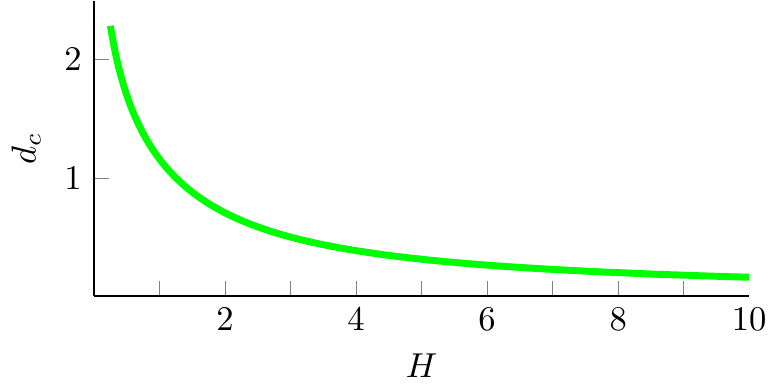}

	\caption{Inverse relationship between the slope $H$ and the critical wavelength $d_c$ for which colonization of bare ground becomes possible. This plot consists of the roots of equation~\eqref{eq:colonization1}. This graphic agrees with plot 5D in~\cite{Sherratt-origins}. Because we use a different scaling both the wavelength and the slope should be divided by$\sqrt{D}$ in~\cite{Sherratt-origins} to obtain the same qualitative plot. One can see that the plots indeed are in good agreement for steeper slopes, i.e. for higher $H$. For less steep slopes, the plots differ. Here the system is closer to the saddle-node bifurcation, which means that assumption (A3) is no longer valid.}
	\label{fig:colonization2}
\end{figure}

\section{Discussion and Outlook}
\label{sec:disc}

In this paper, we extended existing approaches and developed novel methods to study the dynamics of interacting pulse solutions in singularly perturbed 2-component reaction-diffusion systems with parameters that may vary in time and/or space, focusing on the extended Klausmeier -- or generalized Klausmeier-Gray-Scott -- model \eqref{eq:extKmodel1} as prototypical subject of study. We have (formally) shown that the PDE evolution of $N$-pulse patterns can be described by an $N$-dimensional dynamical system and that the solutions of this system live on an (approximate) $N$-dimensional invariant manifold $\mathcal{M}_N$. The stability -- and thus attractivity -- of this manifold is determined by the quasi-steady spectrum that we have determined by Evans function techniques. This analysis also provides insight in the location and nature of the (various components of the) boundary $\partial \mathcal{M}_N$ of $\mathcal{M}_N$, and in the nature of the (linear) destabilization mechanisms associated to $N$-pulse configurations crossing through $\partial \mathcal{M}_N$. Thus, we have found that the dynamics of $N$-pulse patterns can be splitted in two. Firstly, there is the (slow) dynamics \emph{on} the manifold $\mathcal{M}_N$ -- we captured this behaviour in an ODE~\eqref{eq:ODEc1} that describes the evolution of the pulse locations. Secondly, there is (fast) dynamics \emph{off} of $\mathcal{M}_N$, towards a lower-dimensional (approximate, attracting invariant) manifold $\mathcal{M}_M$ (with $M < N$). We have determined the linearized nature of this fall; the  hybrid numerical-asymptotic method developed in this paper predicts the value of $M$, describes the evolution of the resulting $M$ pulses on $\mathcal{M}_M$, and the cascade of jumps towards subsequent manifolds $\mathcal{M}_{\tilde{M}}$.

Our formal approach triggers various themes of further research. The validity of the very first step -- the reduction of the PDE dynamics to $\mathcal{M}_N$ -- is so far only established rigorously for a restricted region in parameter space -- see \cite{bellsky2013}. Moreover, the analysis of \cite{bellsky2013} is in the classical setting of non-varying parameters. Some of the numerical experiments presented in this article were conducted under similar conditions for which the results of \cite{bellsky2013} can be expected to hold; others, however, used parameters way beyond the regions considered in \cite{bellsky2013}. Nevertheless even in those cases the (formally) reduced system usually captures the dynamical movement of the pulses remarkably well. More surprisingly, the ODE reduction even is correct when the prime small parameter of our asymptotic analysis, i.e. $\frac{a}{m}$, is in fact not small, but order $\mathcal{O}(1)$. All in all, the reduction method seems to be valid for settings way beyond the reaches of current validity proofs. It would be extremely valuable to further develop the rigorous theory to understand why the reduction method is so successful.

The behaviour of a $N$-pulse patterns on manifold $\mathcal{M}_N$ was studied using the reduced pulse-location ODE~\eqref{eq:ODEc1}. Under the assumption that the coefficients related to $h(x)$ in \eqref{eq:extKmodel1} do not explicitly vary in $x$ -- $h'(x) \equiv H$, a constantly sloped terrain -- we found on bounded domains with Neumann boundary conditions that $N$-pulse configurations always evolve towards a specific stable fixed point of the ODE; on domains with periodic boundary conditions, the configurations always evolve towards a uniformly traveling solution in which all pulses are equally far apart. These results were proven {\it for the derived ODE approximation} under assumption (A3) in Appendix~\ref{sec:app-FixedPointProofs} and numerics indicate that these still hold under the less restrictive assumption (A3'). Moreover, when $h'(x)$ is allowed to vary -- i.e. for more realistic topographies -- simulations indicate that the pulse-location ODE still has stable fixed points (though there can be multiple fixed points, including unstable ones). A better understanding of the dynamics generated by reduced systems \eqref{eq:ODEc1} is necessary, especially from the ecological point of view. For instance, intuitively, pulses are expected to always move uphill (towards the downhill flowing water). However this mechanism only seems to be valid for terrains with constant slope ($h'(x) \equiv H$); on more realistic terrains pulses can move both uphill and downhill -- depending on the terrain's curvature. This may explain observations of vegetation patterns, that indeed sometimes evolve counter-intuitively (i.e. not uphill) \cite{deblauwe2012determinants, dunkerley2014vegetation}. As a first step towards these goals -- rigorous validation of $\mathcal{M}_N$ and understanding the dynamics on $\mathcal{M}_N$ - one first needs to rigorously establish existence and stability of stationary pulse solutions of~\eqref{eq:extKmodel1} with non-trivial $h(x)$ - this is the subject of~\cite{varyingTerrainArticle}.

The biggest `leap of faith' our method takes is the assumption that insights obtained from the asymptotic analysis of the quasi-steady spectrum can be extrapolated to capture the nonlinear, fast, PDE dynamics of an $N$-pulse configuration crossing through $\partial \mathcal{M}_N$ and jumping from $\mathcal{M}_N$ to $\mathcal{M}_M$ (with $M < N$). Our analysis showed that $\partial \mathcal{M}_N$ corresponds to `quasi-steady bifurcations' -- i.e. bifurcations induced by the intrinsic dynamics of the evolving multi-pulse pattern -- of several types: saddle node bifurcations for small values of $m$ (in \eqref{eq:extKmodel1}) and Hopf bifurcations and decoupled eigenfunctions for large values of $m$. In fact, our linear analysis only yielded information on the appearance of quasi-steady Hopf destabilisations; since all observations of Hopf bifurcations in singularly perturbed reaction-diffusion systems of slowly linear type are subcritical -- see \cite{veerman2015breathing} and the references therein -- we have assumed that all quasi-steady Hopf bifurcations are subcritical. Numerical simulations indicated the correctness of these assumptions in a wide variety of situations; the linear destabilisation arguments predict the fast nonlinear jump mechanisms surprisingly well. Moreover, we found that approximating the stability problem as a decoupled stability problem works convincingly well, even when the leading order asymptotic analysis implied that eigenfunctions are coupled: this a priori oversimplified approximation typically correctly predicts which pulses disappear -- i.e. towards which manifold $\mathcal{M}_M$ an $N$-pulse configuration jumps as it crosses through $\partial \mathcal{M}_N$; it does underestimate the stability slightly, leading to pulses that disappear/jump too early. To obtain a fundamental understanding of the `desertification dynamics' of $N$-pulse patterns in singularly perturbed reaction-diffusion systems -- i.e. the dynamics of pulse patterns jumping from manifolds $\mathcal{M}_k$ to $\mathcal{M}_\ell$ (with $0 \leq \ell < k \leq N$) -- it is crucial to develop analytical insights in the relative locations of the invariant manifolds $\mathcal{M}_n$, $n=1,2,...,N$ within function space, and the nature of the PDE flow between these manifolds. In general, this is a formidable challenge, but such a multi-scale analysis is expected to be possible in specially constructed settings.

Finally, we found that there is a striking difference between the dynamics of regular and irregular patterns. We found that irregular configurations always destabilise gradually -- with pulses disappearing one by one -- whereas for regular configuration either half or all pulses disappear `catastrophically' when  $\partial \mathcal{M}_N$ is crossed\footnote{for small $m$ there always is a period doubling; for large $m$ both a period doubling and a full collapse can happen.}. On the other hand, we also deduced that regularly spaced $N$-pulse configurations are more stable than any other $N$-pulse configuration -- in fact, irregular patterns typically evolve toward regularity on domains with periodic boundary (under specified conditions on $h(x)$ and the nature of the domain and associated boundary conditions). Thus, in situations in which parameters change (slowly) in time -- as $a(t)$ in \eqref{eq:extKmodel1} -- there is a competition between two `desertification scenarios': the gradual one for `sufficiently irregular' patterns in which the pattern step by step jumps down from $\mathcal{M}_k$ to $\mathcal{M}_{k-1}$, and the catastrophic one in which a `sufficiently regular' $N$-pulse pattern looses half or all pulses. The relative time scales of the variation of $a(t)$ versus the intrinsic rate of change of the $N$-pulse pattern as it evolves over $\mathcal{M}_N$ is a decisive ingredient that shapes this competition. A more subtle, but at least as important, ingredient is the -- at present not understood -- (slow) dynamics of the quasi-steady eigenfunctions as they evolve from the irregular setting of being localized around one pulse location to the global Floquet-type eigenfunctions -- see \cite{BjornRiccati} and the references therein -- associated with regular spatially periodic patterns.

\section*{Acknowledgements}
{\it We thank Tom Bellsky for inspiring dicussions. This study was supported by a grant within the Mathematics of Planet Earth program of the Netherlands Organization of Scientific Research (NWO).}

\bibliographystyle{abbrv}
\bibliography{sources}

\appendix
\section{The movement of water on a varying terrain}
\label{A:hx}

\begin{figure}[t]
	\centering
	\begin{subfigure}[t]{0.4\textwidth}
		\centering
	\includegraphics{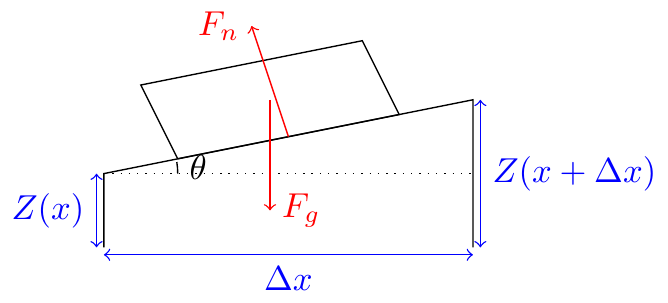}
	\caption{}
	\end{subfigure}
~
	\begin{subfigure}[t]{0.4\textwidth}
		\centering
	\includegraphics{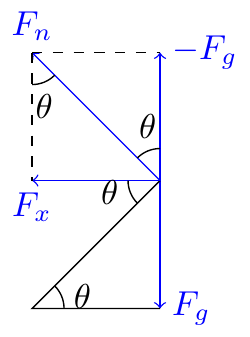}
	\caption{}
	\end{subfigure}
	\caption{Sketch of the classic `mass on incline' problem (a). A mass $M$ is positioned on a slope. Gravity has its effect on this mass and -- due to the normal force -- the box starts to move downwards. For water, the mass $M$ can be replaced by the density $\rho$. A sketch of the relevant forces and angles of the problem are given in (b).}
	\label{fig:appendixSlopeOnIncline}
\end{figure}

Previous versions of the extended Klausmeier model only considered terrains with a constant slope. The model studied in this article, however, is suitable for more generic terrains by the addition of the term $h_{xx} u$. In this appendix, we explain how this new term originates from a shallow water approximation. Here, we denote the concentration/height of water by $U(t,x,y)$, the height of the terrain by $H(x,y)$ and the speed of water by $\vec{v}(t,x,y,z)$. By the principles of mass conservation, a physical model should obey the continuity equation,
\begin{equation}
	\frac{dU}{dt}(t,\vec{x}) = - \vec{\nabla} \cdot \vec{j}(t,\vec{x}) + q(t,\vec{x}),
\end{equation}
where $\vec{j}(t,\vec{x})$ denotes the flux (of water) and $q$ includes all the sources and sinks of the model. In this situation, the flux constitutes of diffusion $\vec{j}_{diff} = - D \vec{\nabla} U$ and advection $\vec{j}_{adv} = \vec{v}U$. Since we want to understand the effect of a terrain, we focus only on the effect of advection. To describe the flow due to advection, we need to determine the velocity $\vec{v}$ of the water. The starting point for this are the momentum equations,
\begin{equation}
	\rho \frac{D \vec{v}}{Dt} = - \vec{\nabla}p + \vec{f}_{gravity} + \vec{f}_{friction}.
\end{equation}
Here, $\rho$ is the density of water, $p$ is the pressure, $\vec{f}$ denotes the forces that act on the water and $\frac{D}{Dt}$ is the material derivative. In this formulation friction is included as a force. Because the height of water (i.e. $U$) is small in semi-arid climates, a shallow water approximation can be made. Thus we assume that there is no movement in the $z$-direction, and that $\vec{v}$ and $\rho$ are constant as function of $z$. In addition, we assume that the pressure $p$ only depends on the $z$-coordinate, and that $\rho$ does not depend on $x$ or $y$. Therefore the $x$- and $y$-momentum equations simplify to
\begin{equation}
	\rho \frac{D \vec{v}}{Dt} = \vec{f}_{gravity} + \vec{f}_{friction}.
\end{equation}
For the force due to friction we assume Rayleigh friction, i.e. $\vec{f}_{friction} = - K \vec{v}$, where $K$ is a (Rayleigh) constant. The force due to gravity comes into play because of the sloped terrain. Ultimately, the computation of the contribution of $\vec{f}_{gravity}$ boils down to the `mass on incline' problem (see Figure~\ref{fig:appendixSlopeOnIncline}). In the continuum limit, this leads to $\vec{f}_{gravity} = -\rho g \tan(\vec{\theta}) = - \rho g \vec{\nabla}Z$, where $Z$ is the relevant height. There are several choices possible for this height. In this article we have chosen $Z = H$, the height of the terrain. Another often used choice is $Z = H + U$, the height of the terrain plus the height of the water; see~\cite{gilad2004ecosystem}.

As a final step, we use the diffusive wave approximation, i.e. $\frac{D\vec{v}}{Dt} = 0$. Combining everything yields the velocity
\begin{equation}
	\vec{v} = - \frac{\rho g}{K} \vec{\nabla} Z = - C \vec{\nabla} Z,
\end{equation}
where $C = \frac{\rho g}{K}$ is a constant. Therefore the advective flux is $\vec{j}_{adv} = - C U \vec{\nabla} Z$. Substitution in the continuity equation gives
\begin{equation}
	\frac{dU}{dt} = C \vec{\nabla} \cdot \left(U \vec{\nabla} Z \right)
\end{equation}
where we have suppressed the diffusive and reaction terms for clarity of presentation. The choice, $Z = H$, which we have made throughout this article, leads to
\begin{equation}
	\frac{dU}{dt} = C \vec{\nabla} \cdot \left(U \vec{\nabla} H \right)= C \vec{\nabla} U \cdot \vec{\nabla} H + C U \Delta H.
\end{equation}
The alternative choice, $Z = H + U$, leads to the expression that is used in e.g.~\cite{gilad2004ecosystem},
\begin{equation}
	\frac{dU}{dt} = C \vec{\nabla} \cdot \left(U \vec{\nabla} (H+U) \right)= \frac{C}{2} \Delta U^2 + C U \Delta H + C \vec{\nabla} U \cdot \vec{\nabla} Z
\end{equation}

\section{Fixed Points of the pulse-location ODE~\eqref{eq:ODEforgenH} -- Proofs}
\label{sec:app-FixedPointProofs}

In this appendix we give proofs of the claims in section~\ref{sec:ODEFixedPoints} about the fixed points of the pulse-location ODE~\eqref{eq:ODEforgenH}. Crucial in all these proofs is the fact that $\tilde{U}_x(P_k^\pm)$ is strictly increasing/decreasing as function of the distance to the neighbouring pulse. For notational simplicity we define the function $R_\pm$ as
\begin{align*}
	R_+(k) & := \left( \frac{H}{2} - \frac{\sqrt{H^2+4}}{2} \frac{e^{Hk/2} - \cosh\left( \sqrt{H^2+4}k/2\right)}{\sinh \left( \sqrt{H^2+4} k/2\right)} \right), \\
	R_-(k) & := \left( \frac{H}{2} + \frac{\sqrt{H^2+4}}{2} \frac{e^{-Hk/2} - \cosh\left( \sqrt{H^2+4} k/2 \right)}{\sinh \left( \sqrt{H^2+4}k/2 \right)} \right).
\end{align*}
The pulse-location ODE~\eqref{eq:ODEforgenH} can then be written as
\begin{equation*}
	\frac{dP_j}{dt} = \frac{D a^2}{m \sqrt{m}} \frac{1}{6} \left[ R_+(\Delta P_j)^2 - R_-(\Delta P_{j-1})^2\right]
\end{equation*}

\subsection{Properties of $R_\pm$}
Before we give the proofs for the fixed we point, we first need to study the functions $R_\pm$. First of all, straightforward limit computations reveal that
\begin{align*}
	\lim_{k \downarrow 0} R_\pm(k) & = 0, &
	\lim_{k \rightarrow \infty} R_\pm(k) & = \frac{H \pm \sqrt{H^2+4}}{2}.
\end{align*}
The derivative of $R_\pm$ is given by
\begin{equation*}
	R_\pm'(k) = \mp \frac{\sqrt{H^2+4}}{2} \frac{S_\pm(k)}{\sinh(\sqrt{H^2+4}k/2)^2},
\end{equation*}
where
{\small
\begin{align*}
	S_\pm(k) =
&\ \left[ \pm \frac{H}{2} e^{\pm H k / 2} - \frac{\sqrt{H^2+4}}{2}\sinh(\sqrt{H^2+4})\right] \sinh(\sqrt{H^2+4}k/2) - \frac{\sqrt{H^2+4}}{2} \cosh(\sqrt{H^2+4}k/2)\left[ e^{\pm Hk/2} - \cosh(\sqrt{H^2+4}k/2)\right] \\
=&\ \frac{\sqrt{H^2+4}}{2} + e^{\pm H k / 2} \left[ \pm \frac{H}{2} \sinh(\sqrt{H^2+4}k/2) - \frac{\sqrt{H^2+4}}{2} \cosh(\sqrt{H^2+4}k/2) \right] \\
= &\ \frac{\sqrt{H^2+4}}{2} + \frac{\pm H - \sqrt{H^2+4}}{4} e^{(\pm H + \sqrt{H^2+4})k/2} + \frac{\mp H - \sqrt{H^2+4}}{4}e^{(\pm H - \sqrt{H^2+4})k/2}.
\end{align*}}

That means that $R_\pm'(k)$ has a zero at $k$ only when $S_\pm(k) = 0$. With straightforward limit computations we can check that $S_\pm(0) = 0$ and that $\lim_{k \rightarrow \infty} |S_\pm(k)| = \infty$. Now, the derivative of $S_\pm$ is easy to compute:
\begin{align*}
	S_\pm'(k)
= &\ \frac{(\pm H - \sqrt{H^2+4})(\pm H + \sqrt{H^2+4})}{8} e^{(\pm H + \sqrt{H^2+4})k/2}  + \frac{(\mp H - \sqrt{H^2+4})(\pm H - \sqrt{H^2+4})}{8} e^{(\pm H - \sqrt{H^2+4})k/2} \\
= &\ - \frac{1}{2} e^{(\pm H + \sqrt{H^2+4})k/2} + \frac{1}{2} e^{(\pm H - \sqrt{H^2+4})k/2} \\
= &\ - e^{\pm H k / 2} \sinh(\sqrt{H^2+4}k/2).
\end{align*}
Hence $S_\pm'(k) = 0$ if and only if $k = 0$. Thus $S_+$ and $S_-$ are strictly decreasing in $k$. Since $S_\pm(0) = 0$ this means that $S_\pm$ has the same sign for all $k > 0$. Therefore $R_+'(k) > 0$ and $R_-'(k) < 0$ for all $k > 0$.

Finally we also need to know which function increases faster in absolute value. For that we can suffice to determine the sign of $R_+' + R_-'$, since $R_+$ is increasing from $0$ and $R_-$ is decreasing from $0$. That means we need to look at the sign of $- (S_+ - S_-)$. A direct computation reveals
\begin{equation}
	- \left[ S_+(k) - S_-(k) \right] = - H \cosh\left(\frac{H k}{2}\right) \sinh\left(\frac{\sqrt{H^2+4}k}{2}\right) + \sqrt{H^2+4} \sinh\left(\frac{H k}{2}\right) \cosh\left(\frac{\sqrt{H^2+4}k}{2}\right). \label{eq:app-fp-difference}
\end{equation}
Taking the derivative of this expression gives
\begin{equation*}
S_-'(k) - S_+'(k) = 2 \sinh(H k / 2) \sinh(\sqrt{H^2+4}k/2)
\end{equation*}
Unless $H = 0$ this expression is never zero for any $k > 0$. Combined with the fact that $R_+'(0) + R_-'(0) = 0$ this this reveals that $R_+' + R_-'$ does not change sign. We might now compute the limit for $k \rightarrow \infty$ to determine which grows faster. Taking the limit of~\eqref{eq:app-fp-difference} as $k \rightarrow \infty$ indicates $\sgn(R_+' + R_-') = \sgn(H)$. Thus if $H > 0$ we see that $R_+$ increases faster and if $H < 0$ $R_-$ increases faster in size; when $H = 0$ both increase at the same rate.

Summarizing everything from this section, we do know the following:
\begin{itemize}
	\item $R_+(k)$ is strictly increasing from $0$ to $\frac{H + \sqrt{H^2+4}}{2}$.
	\item $R_-(k)$ is strictly decreasing from $0$ to $\frac{H - \sqrt{H^2+4}}{2}$.
	\item If $H > 0$ $|R_+(k)|$ increases faster than $|R_-(k)|$; if $H < 0$ it is $|R_-(k)|$ that increases faster; if $H = 0$ they increase at the same rate.
\end{itemize}

\subsection{Unbounded domains}

\newtheorem{theorem}{Theorem}
\begin{theorem}
	On unbounded domains the pulse-location ODE~\eqref{eq:ODEforgenH} does not have any fixed points, unless $N = 1$ and $H = 0$.
\end{theorem}
\begin{proof}
	Without loss of generality we assume $H \geq 0$.

	To have a fixed point, we need to have $\frac{dP_j}{dt} = 0$ for all $j \in \{1,\ldots,N\}$. In particular we need $\frac{dP_N}{dt} = 0$. That is, $R_+(\Delta P_N)^2 = R_-(\Delta P_{N-1})^2$. Since $\Delta P_N \rightarrow \infty$ on unbounded domains we know that $R_+(\Delta P_N)^2 = \left(\frac{H + \sqrt{H^2+4}}{2}\right)^2$. However, we know that $R_-(k)^2 \in \left[0, \left( \frac{H - \sqrt{H^2+4}}{2} \right)^2\right]$ and that this function is strictly increasing. To have equality we therefore need $H = 0$ and $\Delta P_{N-1} \rightarrow \infty$. That is only possible if we only have one pulse, i.e. $N = 1$.
\end{proof}

\begin{theorem}
	On unbounded domains the pulse-location ODE~\eqref{eq:ODEforgenH} does not have a uniformly traveling solution in which all pulses move with the same speed, unless $N = 1$. The distance between the first and last pulse is always increasing.
\end{theorem}
\begin{proof}
	The situation in which $N = 1$ is trivially true. So we restrict ourselves to the cases $N > 1$.

	Now, if a solution with all pulses moving with the same speed would exist, then the distance between the first and last pulse needs to be constant, i.e. the following expression needs to hold true
\begin{equation*}
	0 = \frac{d}{dt} (P_N - P_1) = R_+(\Delta P_N)^2 + R_-(\Delta P_0)^2 - R_+(\Delta P_1)^2 - R_-(\Delta P_{N-1})^2
\end{equation*}
On unbounded domains we have $\Delta P_N \rightarrow \infty$ and $\Delta P_0 \rightarrow \infty$. Thus both $R_+(\Delta P_N)^2$ and $R_-(\Delta P_0)^2$ take on their maximum values. Since $R_\pm(k)^2$ are strictly increasing, the equality above can only hold true if $\Delta P_1 \rightarrow \infty$ and $\Delta P_{N-1} \rightarrow \infty$. That is not possible when $N > 1$. In particular we see that $\frac{d}{dt}(P_N-P_1) > 0$.
\end{proof}

\subsection{Bounded domains with periodic boundaries}

\begin{theorem}
	On bounded domains with periodic boundaries the pulse-location ODE~\eqref{eq:ODEforgenH} does not have any fixed points, unless $H = 0$.
\end{theorem}
\begin{proof}
	In the situation where $H = 0$ one can easily verify that a continuous family of pulse solutions exist by setting $\Delta P_j = L / N$ for all $j$.

	For all $H \neq 0$ we see that if such a fixed point exists, then the sum of the movement of all pulses needs to be zero, i.e. it is required that the following equality holds true
\begin{equation*}
	0 =  \sum_{j=1}^N \frac{dP_j}{dt} = \sum_{j=1}^N \left[ R_+(\Delta P_j)^2 - R_-(\Delta P_j)^2\right].
\end{equation*}
However, since $R_+(k)^2$ and $R_-(k)^2$ increase with a different rate, the terms $\left[ R_+(\Delta P_j)^2 - R_-(\Delta P_j)^2\right]$ are non-zero and carry the same sign for all $j$. Hence the equality does not hold and therefore the ODE does not have a fixed point.
\end{proof}

\begin{theorem}
	On bounded domains with periodic boundaries the pulse-location ODE~\eqref{eq:ODEforgenH} does have a continuous family of uniformly traveling solutions in which all pulses move with the same speed. The distance between pulses for those solutions is always given by $\Delta P_j = L / N$ for all $j$.\label{theorem:periodicTraveling}
\end{theorem}
\begin{proof}
	In the situation where $H = 0$, we know that $\sum_{j=1}^N \frac{dP_j}{dt} = 0$. Therefore each pulse needs to be stationary. That is, $R_+(\Delta P_j)^2 = R_-(\Delta P_{j-1})^2$. Because $R_+(k)^2$ and $R_-(k)^2$ increase at the same rate (when $H = 0$) this means that $\Delta P_j = \Delta P_{j-1}$ for all $j$. As we need that $\sum_{j=1}^N \Delta P_j = N$ this indicates that $\Delta P_j = L / N$.

	Without loss of generality we now assume $H > 0$. To find a solution that has the desired property we need $\frac{dP_j}{dt} = \frac{dP_k}{dt}$ for all $j,k$. In particular we thus need to have
	\begin{equation*}
		R_+(\Delta P_j)^2 - R_-(\Delta P_{j-1})^2 = R_+(\Delta P_{j+1})^2 - R_-(\Delta P_j)^2 \mbox{ for all $j$.}
	\end{equation*}
Since $R_+(k)^2$ and $R_-(k)^2$ are strictly increasing, we can deduce the following: if $\Delta P_j > \Delta P_{j-1}$ then we also need $\Delta P_{j+1} > \Delta P_j$. Repeating this argument reveals $\Delta P_1 > \Delta P_N > \ldots > \Delta P_1$. This obviously cannot hold true and therefore a solution cannot have $\Delta P_j > \Delta P_{j-1}$ for any pulse $j$. Similarly we can exclude the possibility that $\Delta P_j < \Delta P_{j-1}$ for any $j$.

Therefore the only possibility left indicates that $\Delta P_j = \Delta P_k$ for all $j,k$. Since $\sum_{j=1}^N \Delta P_j = N$ that means that $\Delta P_j = L / N$. It is straightforward to check that this indeed gives a solution with the desired property.

\end{proof}

\begin{theorem}
	On bounded domains with periodic boundary conditions, the continuous family of regularly spaced solutions, with $\Delta P_j = L / N$, is stable under the flow of the ODE.
\end{theorem}
\begin{proof}
	By Theorem~\ref{theorem:periodicTraveling} The regularly spaced solutions are fixed points of the related ODE
\begin{equation*}
	\frac{d}{dt} \Delta P_j = \frac{d P_j}{dt} - \frac{d P_{j-1}}{dt}.
\end{equation*}
We denote the fixed points of this equation by $\Delta P_j^*$ and we linearise around them by setting $\Delta P_j = \Delta P_j^* + r_j$, where $\sum_{j=1}^N r_j = 0$ because of the bounded domain. We then obtain
\begin{equation*}
	\frac{d r_j}{dt} = \frac{D a^2}{m \sqrt{m}} \frac{1}{3} \left[ R_+(\Delta P_{j+1}^*) R_+'(\Delta P_{j+1}^*) r_{j+1} - \left( R_+(\Delta P_{j}^*)R_+'(\Delta P_j^*) + R_-(\Delta P_j^*)R_-'(\Delta P_j^*) \right) r_j + R_-(\Delta P_{j-1}^*)R_-'(\Delta P_{j-1}^*) r_{j-1} \right]
\end{equation*}
Because $\Delta P_j^* = L / N$ for all $j \in \{1,\ldots,N\}$, we may define
\begin{align*}
	\alpha &:= R_+(\Delta P_j^*) R_+'(\Delta P_j^*) & \beta &:= R_-(\Delta P_j^*)R_-'(\Delta P_j^*) & C &:= \frac{D a^2}{m\sqrt{m}} \frac{1}{3}.
\end{align*}
Because $R_+$ is positive and increasing, and $R_-$ is negative and increasing, we know that $\alpha > 0$ and $\beta > 0$.

We then define $\vec{r} := \left(r_1,\ldots,r_N\right)^T$ and rewrite the linearised equation as
\begin{equation*}
\frac{d}{dt} \vec{r} = C M \vec{r},
\end{equation*}
where
\begin{equation*}
M := \left( \begin{array}{cccccc}
-(\alpha+\beta)	&	\alpha			&	0		& \cdots 	& 0 			&	\beta 	\\
\beta 			& -(\alpha+\beta) 	& \alpha		& 	0		& \cdots 	&	0		\\
0 				& \ddots 			& \ddots 	& \ddots 	& \ddots 	& \vdots 	\\
\vdots 			& \ddots 			& \ddots 	& \ddots 	& \ddots 	& 0 			\\
0 				& 					& \ddots 	& \ddots 	& \ddots 	& \alpha 	\\
\alpha 			& 0					& \cdots 	& 0 			& \beta 		& -(\alpha+\beta)
\end{array}\right)
\end{equation*}
The matrix $M$, with the additional constraint $\sum_{j=1}^N r_j = 0$ is negative definite, as a straight-forward computation yields (here $r_0 = r_N$ and $r_{N+1} = r_1$)
\begin{equation*}
f(\vec{r}) := \vec{r}^T M \vec{r} = - (\alpha + \beta) \sum_{j=1}^N r_j^2 + \alpha \sum_{j=1}^N r_j r_{j+1} + \beta \sum_{j=1}^N r_j r_{j-1} = - \frac{1}{2} \sum_{j=1}^N (r_j - r_{j+1})^2
\end{equation*}
Thus $f(\vec{r}) < 0$ unless $r_j = r_{j+1}$ for all $j$ -- which is excluded by the condition $\sum_{j=1}^N r_j = 0$. Thus the matrix associated with the linearisation is negative definite and therefore possesses only negative eigenvalues, proving that the regularly spaced configuration are stable under the flow of the ODE.
\end{proof}

\subsection{Bounded domains with Neumann boundary conditions}
\begin{theorem}
	On bounded domains with Neumann boundary conditions the pulse-location ODE~\eqref{eq:ODEforgenH} does always have precisely one fixed point.
\end{theorem}
\begin{proof}
	Without loss of generality we assume $H \geq 0$. For all $H \geq 0$ we have $R_+(k)^2 \geq R_-(k)^2$ for all $k > 0$. Thus for all $x \geq 0$ there is a $y = y(x) \geq 0$ such that $R_+(y(x))^2 = R_-(x)^2$. Since $R_\pm(k)^2$ is strictly increasing, we know that $y$ is strictly increasing in $x$ as well.

	Now, to have a fixed point $P_1^*, \ldots, P_N^*$ we need
\begin{equation}
	R_-(\Delta P_{j-1})^2 = R_+(\Delta P_j)^2 \mbox{ for all $j$. }
\end{equation}
Because of our reasoning above there are strictly increasing functions $y_j$ such that
\begin{equation}
	R_-(\Delta P_{j-1})^2 = R_+( y_j(\Delta P_{j-1}) )^2 \mbox{ for all $j$. }
\end{equation}
So we should choose $P_1, \ldots, P_N$ such that $\Delta P_{j} = y_j(\Delta P_{j-1})$. That is,
\begin{equation}
\Delta P_j = \left( y_j \circ \ldots \circ y_1 \right) (\Delta P_0).
\label{eq:app-yj-application}
\end{equation}
In particular we have $\Delta P_N = \left( y_N \circ \ldots \circ y_1 \right) (\Delta P_0)$. Because $\Delta P_0$ is strictly increasing in $P_1$, we know that this expression for $\Delta P_N$ is a strictly increasing in $P_1$.

At the same time our solution should fit in the domain and therefore we know that $\Delta P_N$ is strictly decreasing in
\begin{equation*}
P_N = P_1 + y_1(\Delta P_0) + \ldots + \left(y_N \circ \ldots y_1\right)(\Delta P_0).
\end{equation*}
Therefore this expression for $\Delta P_N$ is also strictly decreasing in $P_1$.

So we now have two descriptions of $\Delta P_N$ which should be equal. One of these is strictly increasing in $P_1$ starting from $0$ and the other is strictly decreasing in $P_1$ starting from $L$. Therefore there is precisely one location $P_1 = P_1^*$ that leads to equality of these descriptions. The other locations follow from equation~\eqref{eq:app-yj-application}. This leads to a unique fixed point of~\eqref{eq:ODEforgenH}.

\end{proof}

\begin{theorem}
	On bounded domains with Neumann boundary conditions, the unique fixed point solution of~\eqref{eq:ODEforgenH} is stable under the flow of the ODE.
\end{theorem}
\begin{proof}
We denote the fixed point as $P_1^*, \ldots, P_N^*$. Then we linearise by setting $P_j = P_j^* + r_j$, which results in
\begin{align*}
	\frac{d}{dt} r_1 & = \frac{D a^2}{m \sqrt{m}} \frac{1}{3} \left[ \alpha_1 r_2 - (\alpha_j + \beta_j) r_1 -  \beta_j \gamma_1 r_1 \right]; \\
	\frac{d}{dt} r_j & = \frac{D a^2}{m \sqrt{m}} \frac{1}{3} \left[ \alpha_j r_{j+1} - (\alpha_j+\beta_j) r_j + \beta_j r_{j-1} \right]; & (j = 2, \ldots, N_1)\\
	\frac{d}{dt} r_N & = \frac{D a^2}{m \sqrt{m}} \frac{1}{3} \left[ - \alpha_N \gamma_N r_N - (\alpha_N + \beta_N) r_N + \beta_N r_{N-1} \right],\\
\end{align*}
where
\begin{align*}
	\alpha_j 	& := R_+(P_{j+1}^*-P_{j}^*)R'_+(P_{j+1}^*-P_{j}^*)	&
	\beta_j 		& := R_-(P_j^* - P_{j-1}^*)R'_-(P_j^* - P_{j-1}^*)		&
	C 			& := \frac{D a^2}{m\sqrt{m}} \frac{1}{3}				
\end{align*}\begin{align*}
	\gamma_1 	& := - \frac{d}{dP_1} P_0(P_1^*)						&
	\gamma_N 	& := - \frac{d}{dP_N} P_{N+1}(P_N^*)
\end{align*}
Note that the function $R_+$ is positive and increasing, $R_-$ is negative and decreasing, $P_0$ is decreasing and $P_{N+1}$ is decreasing. Therefore $\alpha_j > 0$, $\beta_j > 0$, $\gamma_1 > 0$ and $\gamma_N > 0$.

We then define $\vec{r} := \left(r_1,\ldots,r_N\right)^T$ and rewrite the linearised equation as
\begin{equation*}	\frac{d}{dt} \vec{r} = C M \vec{r}, \end{equation*}
where
\begin{equation*}
	M = \left( \begin{array}{cccccc}
-(\alpha_1+\beta_1) - \gamma_1 \beta_1 	&	\alpha_1			& 0			& \cdots & \cdots 	& 0 			\\
\beta_2 									& -(\alpha_2+\beta_2) 	& \alpha_2	& \ddots & 			& \vdots 	\\
0 										& \ddots 				& \ddots 	& \ddots & \ddots 	& \vdots 	\\
\vdots 									& \ddots 				& \ddots 	& \ddots & \ddots 	& 0 			\\
\vdots 									& 						& \ddots 	& \beta_{N-1} & -(\alpha_{N-1}+\beta_{N-1}) & \alpha_{N-1} \\
0 & \cdots & \cdots & 0 & \beta_N & - (\alpha_N + \beta_N) - \gamma_N \alpha_N
\end{array}\right)
\end{equation*}
Because of the structure of $M$, the Gershgorin circle theorem~\cite{Gerschgorin} immediately indicates that all eigenvalues lie in a Greschgorin disc. Because $M$ is weak diagonal dominant, the only non-negative eigenvalue that is not yet excluded is $\lambda = 0$. The rest of this proof consists of proving that $\lambda = 0$ cannot be an eigenvalue.

If $\lambda = 0$ would be an eigenvalue then there is an eigenvector $\vec{x} = \left(x_1,\ldots,x_N\right)^T \neq 0$ such that $M \vec{x} = 0$. This vector needs to satisfy
\begin{align}
	- (\alpha_1 + \beta_1 + \gamma_1 \beta_1) x_1 + \alpha_1 x_2 			& = 0\\
	\beta_j x_{j-1} - (\alpha_j+\beta_j) x_j + \alpha_j x_{j+1} 				& = 0 \hspace{1cm} (j=2,\ldots,N_1) \\
	\beta_N x_{N-1} - (\alpha_N + \beta_N + \gamma_N \alpha_N) x_N 		& = 0 \label{eq:appEVcond}
\end{align}
From the first $N-1$ of these expressions one can formulate each $x_j$ in terms of $x_1$. We find $x_j = \delta_j x_1$, with
\begin{align*}
\delta_{j+1} & = \delta_j + \frac{\beta_j}{\alpha_j}(\delta_j - \delta_{j-1}), &
\delta_1 & = 1, &
\delta_2 & = 1 + \frac{\beta_1}{\alpha_1}\left(1+\gamma_1\right) > \delta_1.
\end{align*}
One might easily verify that $\delta_j > \delta_{j-1}$ for all $j$.

Finally, if $\vec{x}$ is an eigenvector it should also satisfy the $N$-th expression~\eqref{eq:appEVcond}. Substitution of the found expressions results in the condition
\begin{equation*}
\left[\beta_N (\delta_{N_1}-\delta_N) - (1+\gamma_N)\alpha_N \delta_N \right] = 0.
\end{equation*}
Because $\delta_N > \delta_{N_1}$ the left-hand side of this equation is always negative. Therefore this condition can never be fulfilled and hence $\lambda = 0$ cannot be an eigenvalue of $M$. Thus all eigenvalues of $M$ need to be negative and the fixed point is thus stable under the flow of the ODE.
\end{proof}

\end{document}